\documentclass{amsart}
\usepackage{amssymb}
\mathchardef\mhyphen="2D
\usepackage{comment}
\usepackage{xy}
\usepackage{tikz,placeins}
\usetikzlibrary{calc}
\usepackage{mathtools}
\xyoption{all}
\usepackage[T1]{fontenc}
\usepackage{csquotes}
\usepackage{hyperref}
\usepackage{cleveref}

\title{Ext-groups of graded comodules}
\date{June 2026}
\author{Andrew Salch}

\DeclareMathOperator{\Ab}{{\rm Ab}}
\DeclareMathOperator{\coker}{{\rm coker}}
\DeclareMathOperator{\op}{{\rm op}}
\DeclareMathOperator{\dist}{{\rm dist}}
\DeclareMathOperator{\Ext}{{\rm Ext}}
\DeclareMathOperator{\Mod}{{\rm Mod}}

\DeclareMathOperator{\gr}{{\rm gr}}
\DeclareMathOperator{\colim}{{\rm colim}}
\DeclareMathOperator{\holim}{{\rm holim}}
\DeclareMathOperator{\Comod}{{\rm Comod}}
\DeclareMathOperator{\pt}{{\rm pt.}}
\DeclareMathOperator{\id}{{\rm id}}
\DeclareMathOperator{\tors}{{\rm tors}}
\DeclareMathOperator{\tf}{{\rm tf}}
\DeclareMathOperator{\Aut}{{\rm Aut}}
\DeclareMathOperator{\tr}{{\rm tr}}

\theoremstyle{plain}
\newtheorem{prop}{Proposition}[section]
\newtheorem{lemma}[prop]{Lemma}
\newtheorem{theorem}[prop]{Theorem}
\newtheorem{corollary}[prop]{Corollary}
\newtheorem{definition}[prop]{Definition}
\newtheorem*{unnumberedtheorem}{Theorem}
\newcounter{lettered}
\newtheorem{letteredtheorem}[lettered]{Theorem}

\theoremstyle{definition}
\newtheorem{convention}[prop]{Convention}
\newtheorem{observation}[prop]{Observation}
\newtheorem{recollection}[prop]{Recollection}
\newtheorem{remark}[prop]{Remark}

\begin{document}
\begin{abstract}
Let $\Gamma$ be a coalgebra over a field. There is a well-known covariant embedding of the category of graded $\Gamma$-comodules into the category of graded modules over the dual algebra $\Gamma^*$. We show that this covariant embedding preserves Ext-groups in a range of motivating cases, including in some cases like Steenrod algebras where the covariant embedding has otherwise bad homological properties and {\em fails} to preserve Ext-groups on the underlying ungraded categories. 

As an application, we prove the old conjecture in stable homotopy theory that there are no nonzero projectives in the category of graded comodules over the dual Steenrod algebra. As another application, we carry out a study of the right adjoint to the extended comodule functor, including a topological interpretation of this right adjoint. This yields generalizations of formulas of Margolis and Lin characterizing maps from Eilenberg-Mac Lane spectra into finite-type bounded-below spectra; our generalizations drop the ``finite-type'' hypothesis.
\end{abstract}
\maketitle
\tableofcontents

\section{Introduction}

This paper is about a technical algebraic result and its topological applications. The technical result is that, for a certain class of graded coalgebras $\Gamma$, the covariant embedding of the category of graded $\Gamma$-comodules into the category of graded $\Gamma^*$-modules preserves $\Ext$-groups. The applications are:
\begin{itemize}
\item we prove an old conjecture about the Steenrod algebra, namely, there are no nonzero projective graded comodules over its dual Hopf algebra,
\item we give a description of the Adams spectral sequence $E_2$-page in terms of $\Ext$ over the Steenrod algebra (rather than $\Ext$ in comodules over the dual Steenrod algebra), without any finiteness hypotheses at all,
\item we give a topological interpretation and cohomological analysis of the right adjoint to the extended comodule functor,
\item and we give a generalization of the formula of Margolis, and some general versions of the formula of Lin, for the homotopy classes of maps from an Eilenberg-Mac Lane spectrum to a bounded-below finite-type spectrum.  Our generalizations remove the hypothesis that the spectrum is finite-type.
\end{itemize}

We state those results below, but first we had better explain some background about covariant embeddings of comodule categories into module categories. Given a commutative ring $k$ and a coalgebra $\Gamma$ which is flat over $k$, one has the abelian category $\Comod(\Gamma)$ of left $\Gamma$-comodules, and one has the abelian category $\Mod(\Gamma^*)$ of modules over the $k$-linear dual algebra $\Gamma^*$. If $\Gamma$ is furthermore projective over $k$, then there is a covariant embedding $\iota$ of $\Comod(\Gamma)$ into $\Mod(\Gamma^*)$. More precisely, we have a full, faithful, exact functor $\iota: \Comod(\Gamma) \rightarrow \Mod(\Gamma^*)$, which sends a left $\Gamma$-comodule $M$ to the underlying $k$-module of $M$, equipped with the adjoint action of $\Gamma^*$. This is classical material, with a textbook treatment in \cite{MR2012570}, but for the reader's convenience we give details and an introductory treatment in \cref{Generalities on...} in this paper.

We emphasize that $\iota$ is a {\em covariant} embedding. It is {\em not} the linear dualization functor from $\gr\Comod(\Gamma)$ to $\gr\Mod(\Gamma^*)$, which is generally more familiar to topologists, but which is contravariant. 
The linear dualization functor has very good behavior when its domain is restricted to the graded comodules which are finite-type (i.e., finite-dimensional in each degree) and bounded-below. In particular, when restricted to bounded-below finite-type comodules, linear dualization preserves $\Ext$-groups, so if one only ever cares about the bounded-below finite-type case, one might as well simplify one's life and work with modules throughout. Comodules are only necessary in order to deal with complications that arise when the finite-type condition fails, e.g. in topological applications (see below) involving spectra that are not finite-type. In such cases the covariant embedding $\iota$ has some desirable properties (like preserving $\Ext$-groups) not shared by the linear dualization functor, and which we prove and take advantage of in this paper.

One generally has good tools for calculating $\Ext$-groups in categories of modules. The problem of calculation of $\Ext$-groups in categories of comodules is, by comparison, far less studied, with far fewer available tools. Nevertheless one sometimes needs to calculate $\Ext$-groups in categories of comodules. The starting point of this paper is the idea that we might apply the covariant embedding $\iota$ to convert the problem of calculating $\Ext$-groups in the comodule category into the more tractable problem of calculating $\Ext$-groups in the module category. We explain below that this strategy cannot work in full generality, since for most graded coalgebras $\Gamma$, the covariant embedding $\iota: \gr\Comod(\Gamma)\rightarrow\gr\Mod(\Gamma^*)$ fails to preserve $\Ext$-groups. However, it turns out that, for the graded coalgebras $\Gamma$ most important for stable homotopy theory, the strategy is indeed successful, because we show (in Theorem \ref{main technical thm}, below) that in those cases, $\iota$ {\em does} preserve the $\Ext$-groups. 

The motivations behind this paper are entirely from stable homotopy theory\footnote{This is largely a paper about pure algebra, with applications to stable homotopy theory. We hope that this paper may have some interest to stable homotopy theorists and also to pure algebraists who are interested in categories of comodules and their homological properties. We have tried to write this paper so that it is readable to both potential audiences. We apologize to topologists who find that there is too much exposition of background ideas from topology, and to algebraists who find that there is too much exposition of background ideas from algebra.}, so here is a short review of why stable homotopy theorists try to calculate $\Ext$-groups in the category $\gr\Comod(\Gamma)$ of {\em graded} comodules over a {\em graded} coalgebra. 
For each prime number $p$ and each topological space or spectrum $X$, the mod $p$ homology groups $H_*(X;\mathbb{F}_p)$ have the natural structure of a graded comodule over a certain Hopf algebra, the {\em mod $p$ dual Steenrod algebra}, for which we will write $\Gamma$. Suppose that $X$ and $Y$ are spaces or spectra, and suppose that the stable homotopy groups of $Y$ are bounded-below (i.e., they are not nontrivial in arbitrarily low negative degrees) and $p$-adically complete. Then the $H\mathbb{F}_p$-Adams spectral sequence is of the form
\begin{align}
\label{adams ss 0} E_2^{s,t} \cong \Ext_{\gr\Comod(\Gamma)}^s\left(H_*(X;\mathbb{F}_p),H_*(Y;\mathbb{F}_p)\right)^t &\Rightarrow [\Sigma^{t-s}X,Y],\\
\nonumber d_r: E_r^{s,t} &\rightarrow E_r^{s+r,t+r-1},
\end{align}
where $[\Sigma^{t-s}X,Y]$ is the group of stable homotopy classes of maps from the $(t-s)$th suspension of $X$ to $Y$. This spectral sequence lets us pass from our knowledge of Ext-groups in graded comodules over the coalgebra $\Gamma$ to calculations of stable homotopy classes of maps, e.g. when $X$ is a sphere we are calculating the stable homotopy groups of $Y$. A few further notes on spectral sequence \eqref{adams ss 0}:
\begin{itemize}
\item Non-topologist readers may be interested in \cref{Review of the Steenrod algebras} in this paper's appendix, which is a brief, self-contained introduction to Steenrod algebras.
\item The spectral sequence \eqref{adams ss 0} was originally constructed by Adams in 1958 \cite{MR96219} in such a way that its $E_2$-page consists of Ext-groups in a {\em module} category, not a {\em comodule} category: 
\begin{align}\label{old adams e2 1} E_2^{s,t} &\cong \Ext_{\gr\Mod(\Gamma^*)}^s\left(H^*(Y;\mathbb{F}_p),H^*(X;\mathbb{F}_p)\right)^t.\end{align} Adams was quite clear, in \cite{MR96219}, that this description only applies when $X$ and $Y$ are each {\em finite-type}\footnote{Here is a slightly more careful account of the history. In the original paper, \cite{MR96219}, Adams only constructs the spectral sequence in the case that $X$ is the zero-sphere $S^0$. He indeed is careful about the fact that $Y$ must be assumed finite-type for the description \eqref{old adams e2 1}. By 1968, when Adams's lectures \cite{MR0251716} were given, it was known how to build the spectral sequence for general $X$ and $Y$, including non-finite-type spectra, and it was known also that failure of $X$ and $Y$ to be finite-type can cause the $E_2$-page description \eqref{old adams e2 1} to be incorrect. See \cite[pgs. 51-55]{MR0251716} for an illuminating discussion of this difficulty.
}, i.e., their stable homotopy groups are finitely-generated in each degree. Adams introduced the use of $\Ext$-groups in comodule categories, rather than module categories, in \cite{MR0251716} in order to get a reasonable description of the $E_2$-pages of various Adams spectral sequences in the cases where $X$ and $Y$ are not assumed to be finite-type. The description \eqref{adams ss 0} is a special case.
\item For the sake of this introduction, in \eqref{adams ss 0} we described only a very particular kind of Adams spectral sequence. There are many generalizations of what we described in \eqref{adams ss 0}, and in this paper at times we will work with various of those generalizations. We give fuller introductory accounts of generalized Adams spectral sequences in Proposition \ref{adams E2 description} and in \cref{rel ext and adams E2}.
\end{itemize}

Now that we have described some topological motivation for calculating $\Ext$ in the category of graded comodules over a graded coalgebra, we will state the main results in this paper.
The main technical result is the following. 
\begin{letteredtheorem} (Theorem \ref{ext iso 1}) \label{main technical thm}
Let $k$ be a field, and let $\Gamma$ be a connected graded $k$-coalgebra whose dual $\Gamma^*$ is a $\mathcal{P}$-algebra. Suppose that $\Gamma$ is finite-type, i.e., it is finite-dimensional in each degree. Then the covariant embedding $\iota: \gr\Comod(\Gamma)\rightarrow \gr\Mod(\Gamma^*)$ preserves $\Ext$-groups. That is, for any graded $\Gamma$-modules $M,N$ and any integer $n$, we have an isomorphism
\begin{align}
\label{ext iso 1 1a} \Ext^n_{\gr\Comod(\Gamma)}(M,N) &\cong \Ext^n_{\gr\Mod(\Gamma^*)}(\iota(M),\iota(N)).
\end{align}
\end{letteredtheorem}
Theorem \ref{ext iso 1} refers to ``$\mathcal{P}$-algebras.'' We recall Margolis's definition \cite{MR738973} of a $\mathcal{P}$-algebra below, in Definition \ref{def of p-alg}. The definition is slightly long, so rather than duplicating that definition here in the introduction, we prefer to put it like this: the $\mathcal{P}$-algebras are a certain class of infinite-dimensional\footnote{To be clear, it follows easily from the definition of $\mathcal{P}$-algebras that they are {\em always} infinite-dimensional. The analogue of Theorem \ref{main technical thm} for finite-dimensional coalgebras is extremely easy to prove.} connected graded $k$-algebras which includes, as our motivating examples, the Steenrod algebras. 

Our proof of Theorem \ref{ext iso 1} uses Margolis's structural results \cite[section III.13]{MR738973} on modules over $\mathcal{P}$-algebras, together with ideas and results from \cite{MR5014030} about the relationship between ``rational $\Gamma^*$-modules'' (i.e., $\Gamma^*$-modules in the essential image of the covariant embedding $\iota: \gr\Comod(\Gamma)\rightarrow\gr\Mod(\Gamma^*)$) and local cohomology of the graded algebra $\Gamma^*$. 

There is something a bit surprising about Theorem \ref{main technical thm}. By a theorem of Positselski \cite[Theorem 6.1]{MR5038878}, for a general connected graded (more generally, conilpotent) coalgebra $\Gamma$ over a field $k$, the ungraded analogue of \eqref{ext iso 1 1a}, the map
\begin{align}
\label{ext iso 1 1b} \Ext^n_{\Comod(\Gamma)}(M,N) &\rightarrow \Ext^n_{\Mod(\Gamma^*)}(\iota(M),\iota(N)),
\end{align}
fails to be surjective for the least integer $n$ such that $\Ext_{\Comod(\Gamma)}^n(k,k)$ is infinite-dimensional. If $\Gamma$ is the dual Steenrod algebra, then $\Ext_{\Comod(\Gamma)}^1(k,k)$ is infinite-dimensional (since Milnor \cite{MR0099653} constructs an infinite linearly-independent family of indecomposables in $\Gamma$), so Positselski's result implies that the covariant embedding of {\em ungraded} comodules over the dual Steenrod algebra, into {\em ungraded} modules over the Steenrod algebra, does not even preserve $\Ext^1$. It is a curiosity that, by Theorem \ref{main technical thm}, the covariant embedding preserves the $\Ext$-groups in the {\em graded} categories, when it does not preserve $\Ext$-groups in the underlying {\em ungraded} categories.

One application for Theorem \ref{ext iso 1} is a proof of an old minor conjecture in stable homotopy. This conjecture dates back 25 years to Palmieri's 2001 monograph \cite{MR1821838}. In \cite[Example 1.1.12(b)]{MR1821838} Palmieri writes:
\begin{displayquote}
{[I]}f $\Gamma = A_*$ is the dual of the mod $p$ Steenrod algebra, then it seems likely that there are no nonzero projective comodules. Essentially, an element in a projective comodule over $A_*$ should need to have infinitely many elements in its diagonal, so it would have to be viewed as a completed comodule, not a comodule proper.
\end{displayquote}
Hovey writes in 2007 \cite[Remark 1.7(b)]{MR2330512}:
\begin{displayquote}
$\Gamma$-comod does not, in general, have enough projectives. If we take $(A, \Gamma) =(\mathbb{F}_p, \mathcal{A})$, where $\mathcal{A}$ denotes the dual Steenrod algebra, it is generally believed
that there are no nonzero projective comodules.
\end{displayquote}
Using Theorem \ref{main technical thm} in tandem with Margolis's structural results about modules over $\mathcal{P}$-algebras, with a bit of work we obtain:
\begin{letteredtheorem}(Theorem \ref{no nonzero projs})\label{thm b}
Let $k$ be a field, and let $\Gamma$ be a finite-type connected graded Mitchell $k$-coalgebra. Then there are no nonzero projective graded $\Gamma$-comodules. 
\end{letteredtheorem}
The definition of a ``Mitchell coalgebra'' is technical and rather long. Rather than give the definition here, we find it more useful to simply point out that, by a theorem of Mitchell \cite{MR793186}, the mod $p$ dual Steenrod algebra is a Mitchell coalgebra for every prime $p$. Since the hypotheses of Theorem \ref{thm b} are sufficiently general to permit $\Gamma$ to be the dual Steenrod algebra, Theorem \ref{thm b} proves the conjecture about projective graded comodules over the dual Steenrod algebra.

We also derive some topological consequences of Theorem \ref{main technical thm} which generalize the following theorems of Margolis and T.Y. Lin from the 1970s:
\begin{unnumberedtheorem} {\bf (Margolis \cite[Theorem 3]{MR341488}.)} 
Let $Y$ be a bounded-below finite-type spectrum. Let $\Gamma^*$ denote the mod $p$ Steenrod algebra. Then we have an isomorphism of graded abelian groups
\begin{align*}
 \left[ \Sigma^* H\mathbb{F}_p,Y\right]
  &\cong \hom_{\gr\Mod(\Gamma^*)}\left( H^*(Y;\mathbb{F}_p), \Gamma^*\right).
\end{align*}
\end{unnumberedtheorem}
\begin{unnumberedtheorem} {\bf (Lin \cite[Theorem 4.2]{MR0402738}.)} Let $Y$ be the suspension spectrum of a finite-type CW-complex. For any abelian group $A$, we have an isomorphism of graded abelian groups
\begin{align*} \left[ \Sigma^* HA,Y\right] 
 &\cong \Ext_{\mathbb{Z}}^1\left( A\otimes_{\mathbb{Z}}\mathbb{Q},H_{*+1}(Y;\mathbb{Z})\right).
\end{align*}
\end{unnumberedtheorem}
Our generalizations of the theorems of Margolis and Lin involve weakening the hypotheses, in particular, removing the hypothesis that the spectrum $Y$ is finite-type. To state our generalizations, it is convenient to first make the following observation. Given a graded coalgebra (or, for that matter, a graded Hopf algebroid) $\Gamma$ which is flat over a commutative ring $A$, there is the well-known ``extended comodule'' functor $E: \gr\Mod(A) \rightarrow \gr\Comod(\Gamma)$ which sends a graded $A$-module $M$ to $E(M) = \Gamma\otimes_AM$. It is well-known that $E$ has a left adjoint, given by the forgetful functor $\gr\Comod(\Gamma)\rightarrow\gr\Mod(A)$. It is much less well-known---in fact, we have never seen or heard mention of it anywhere at all, although it is very easy to prove---that $E$ also has a {\em right} adjoint, which we call $V$, and which we study in \cref{V section}. Once one notices that $V$ exists, one easily derives the formula for it (Proposition \ref{I prop 1}):
\begin{align*}
V(M) &= \hom_{\gr\Comod(\Gamma)}(\Gamma,M).
\end{align*} 

The $\Gamma$-comodule maps out of $\Gamma$ are perhaps not as simply-behaved as one might hope, since $\Gamma$ is not free but is rather ``co-free'' in the category of $\Gamma$-comodules. 
To get a sense of the nontrivial behavior of $V$, consider the case where $\Gamma$ is the Steenrod algebra, and apply the covariant embedding $\iota: \gr\Comod(\Gamma)\rightarrow\gr\Mod(\Gamma^*)$:
we have 
\begin{align*}
 V(M) &\cong \hom_{\gr\Mod(\Gamma^*)}\left(\iota(\Gamma), \iota(M) \right), 
\end{align*}
i.e., $V(M)$ is the group of Steenrod-algebra-linear maps from the {\em dual} Steenrod algebra, with adjoint Steenrod algebra action, to $M$. If $f$ is such a map, then $f(1)$ must be divisible by each of the Steenrod powers; already one sees that for many choices of $M$, $V(M)$ will be trivial. 

The functor $V$ is left exact, and for general $\Gamma$ one expects that its right-derived functors $R^*V$ can be nonzero. In \cref{V and cotorsion} we recall the definition of ``cotorsion'' in terms of the injective trace functor, from the algebraic literature on cotorsion, in particular Martsinkovsky--Russell \cite{MR4045216}. In Proposition \ref{cotorsion and relative cotorsion}, we find that the comodules which are cotorsion relative to the relatively injective comodules coincide with those comodules $M$ such that $V(M) = 0$.
 Hence we say that a graded comodule $M$ is {\em relative cotorsion} if $V(M) = 0$. We say that $M$ is {\em derived cotorsion} if $R^nV(M) = 0$ for all $n$. See \cref{V and cotorsion} for discussion and comparison of variations of these definitions.

In Proposition \ref{vanishing above R1V} we use Theorem \ref{main technical thm}, together with another round of Margolis's structure theory for modules over $\mathcal{P}$-algebras, to show that, for a finite-type connected graded $k$-coalgebra $\Gamma$ whose dual is a $\mathcal{P}$-algebra, $R^nV(M)$ vanishes for all graded $\Gamma$-comodules $M$ and for all $n>1$, and furthermore $R^1V(M)$ is nontrivial for some $M$. In that sense, the dual Steenrod algebra has ``$V$-cohomological dimension exactly $1$.''

The result on $V$-cohomological dimension implies immediate collapse of certain Adams spectral sequences, which enables us to prove the following generalization of the theorem of Margolis.
\begin{letteredtheorem} \label{margolis letteredthm} (Theorem \ref{top thm 1}) Let $Y$ be a bounded-below spectrum. For each prime $p$, treat the homology $H_*(Y; \mathbb{F}_p)$ as a graded $\Gamma$-comodule, where $\Gamma$ is the mod $p$ dual Steenrod algebra. Then there exists a split short exact sequence of graded $\mathbb{F}_p$-vector spaces
\begin{equation}\label{ses 11049aa} 
0 \rightarrow
 R^1V\left(H_*(\Sigma Y;\mathbb{F}_p)\right) \rightarrow
 \left[ \Sigma^*H\mathbb{F}_p, Y\right] \rightarrow
V\left(H_*(Y;\mathbb{F}_p)\right) \rightarrow
 0.
\end{equation}
In particular, the graded $\Gamma$-comodule $H_*(Y; \mathbb{F}_p)$ is derived cotorsion if and only if the mapping spectrum $F(H\mathbb{F}_p,Y)$ is contractible.
\end{letteredtheorem}
For the sake of simplicity, the statement we have given of Theorem \ref{margolis letteredthm} here in the introduction is slightly less general than the statement of Theorem \ref{top thm 1} in the body of the paper.

Theorem \ref{margolis letteredthm} indeed is a generalization of Margolis's theorem. This is because we show that, for a {\em finite-type} bounded-below spectrum $Y$, the $R^1V$-term in \eqref{ses 11049aa} vanishes and $V(H_*(Y;\mathbb{F}_p))$ is isomorphic to $\hom_{\gr\Mod(\Gamma^*)}(H^*(Y;\mathbb{F}_p),\Gamma^*)$.

Here are some versions of Lin's theorem which, like Lin's theorem, provide descriptions of the homotopy groups of the mapping spectrum $F(HA,Y)$ for an abelian group $Y$ and a bounded-below spectrum $Y$, but unlike in Lin's theorem, $Y$ is not assumed to be finite-type. One pays a certain price for generality: none of these generalizations are as simply-stated as Lin's theorem itself.
\begin{letteredtheorem} (Theorem \ref{derived cotorsion top thm})
Let $Y$ be a bounded-below spectrum. 
 Let $P$ be a set of primes, and let $A$ be an abelian group. Make the following assumptions:
\begin{itemize}
\item Every prime in $P$ acts isomorphically on $A$, i.e., $A$ is a $\mathbb{Z}[P^{-1}]$-module.
\item For every prime number $p\notin P$, the comodule $H_*(Y; \mathbb{F}_p)$ over the mod $p$ dual Steenrod algebra is derived cotorsion\footnote{For the purpose of checking this hypothesis in cases of interest, the following sufficient condition can be useful. In Corollary \ref{top thm 1 cor} we show that, if $Y$ is the suspension spectrum of a finite-type CW-complex {\em or} if the homology groups of $Y$ are concentrated in finitely many degrees, then $H_*(Y;\mathbb{F}_p)$ is indeed derived cotorsion.}.
\end{itemize}
Then for each integer $n$, there is a short exact sequence
\begin{equation*}
0 \rightarrow \Ext^1_{\mathbb{Z}}\left( \mathbb{Q}\otimes_{\mathbb{Z}}A, \pi_{n+1}Y\right) \rightarrow \left[ \Sigma^n HA,Y\right] \rightarrow \hom_{\mathbb{Z}}\left(\mathbb{Q}\otimes_{\mathbb{Z}}A, \pi_nY\right)\rightarrow 0.\end{equation*}
\end{letteredtheorem}
\begin{letteredtheorem}
(Theorem \ref{torsion group top thm}) Let $p$ be a prime, and let $A$ be an abelian $p$-group, i.e., an abelian group, not necessarily finite, in which every element has order a power of $p$. Write $A[p^n]$ for the $p^n$-torsion subgroup of $A$. Then there exists a conditionally convergent spectral sequence
\begin{align*}
 E_1^{s,t} &\cong \hom_{\mathbb{F}_p}\left( A[p^s]/A[p^{s-1}],\left(R^1V(H_*(\Sigma Y;\mathbb{F}_p))\oplus V(H_*(Y;\mathbb{F}_p))\right)^t \right)\\
  & \Rightarrow \left[ \Sigma^t HA,Y\right] \\
 d_r: E_r^{s,t} &\rightarrow E_r^{s-r,t+1},
\end{align*} 
where $\left(R^1V(H_*(\Sigma Y;\mathbb{F}_p))\oplus V(H_*(Y;\mathbb{F}_p))\right)^t$ denotes the degree $t$ summand in the graded $\mathbb{F}_p$-vector space $R^1V(H_*(\Sigma Y;\mathbb{F}_p))\oplus V(H_*(Y;\mathbb{F}_p))$.
\end{letteredtheorem}
The interested reader can read \cref{I in the case of the Steenrod alg} for further discussion, other generalizations of Lin's theorem, and various corollaries.

There is one more curious small corollary of Theorem \ref{ext iso 1}. In Corollary \ref{adams e2 cor} we obtain a description of the $E_2$-page of the Adams spectral sequence converging to $[\Sigma^*X, \hat{Y}_{H\mathbb{F}_p}]$ which, like Adams's original description \eqref{old adams e2 1} in \cite{MR96219}, is given entirely in terms of modules, {\em not} comodules. However, Corollary \ref{adams e2 cor} does not require the usual finite-type assumptions from Adams's original description. Adams originally introduced stable homotopy theorists to the use of $\Ext$ in comodule categories in \cite{MR0251716} as a way to describe the Adams spectral sequence $E_2$-page without finite-type hypotheses on the spectra involved. Using Corollary \ref{adams e2 cor}, the staunch opponent of comodules can describe the Adams $E_2$-page entirely in terms of modules, without the traditional need to restrict to finite-type spectra.

In \cref{Topological interpretation of V...} we give an analysis of the homological properties of the functor $V: \gr\Comod(E_*E)\rightarrow\gr\Mod(\pi_*(E))$ for some ring spectra $E$ of general interest aside from the case $E = H\mathbb{F}_p$ studied throughout the rest of the paper. In particular, \cref{Topological interpretation of V...} considers the cases where $E$ is the rational Eilenberg-Mac Lane spectrum $H\mathbb{Q}$ (which leads to quite trivial results) or the height $1$ Johnson--Wilson theory $E(1)$, i.e., the Adams summand of $p$-local periodic complex $K$-theory.

This is an old-fashioned paper: aside from the tools we bring in from \cite{MR5014030}, all the tools and methods used to prove the results in this paper were available by the mid-1980s, and in particular, a great deal of work is done by a few results of Margolis \cite{MR341488},\cite{MR738973} on modules over $\mathcal{P}$-algebras. 

We point out that this is not the first paper to use Margolis's results on $\mathcal{P}$-algebras to derive consequences for categories of comodules: the recent preprints \cite{bakerPalgebraspreprint1} and \cite{bakerPalgebraspreprint2} of A. Baker also do this. However, Baker's arguments and results are entirely unlike those in this paper. Rather than a result like Theorem \ref{main technical thm} that compares $\Ext$-groups in comodules to $\Ext$-groups in modules, Baker's approach is to apply Margolis's results on $\mathcal{P}$-algebras to the Cartan--Eilenberg spectral sequence for comodule $\Ext$-groups arising from an extension of Hopf algebras. This yields some novel topological consequences. Baker also arrives at generalizations of Lin's theorem, but his generalizations are unlike ours: while we generalize Lin's result by lifting the finite-type hypotheses, Baker maintains the finite-type hypothesis but replaces the role of $H\mathbb{F}_p$ with other spectra, e.g. to prove contractibility results on the mapping spectrum $F(BP,Y)$ rather than the mapping spectrum $F(H\mathbb{F}_p,Y)$.

\subsection{Conventions}
\label{conventions}
\begin{itemize}
\item Throughout, all our comodules will be {\em left} comodules.
\item We often work with graded rings and graded modules. Given a graded module $M$, we use the topologists' suspension notation, $\Sigma M$, to denote $M$ with the grading degree of each element increased by $1$.
\item Given graded $R$-modules $M,N$, we write $\hom_{\gr\Mod(R)}(M,N)$ for the {\em graded} abelian group of $R$-module maps from suspensions of $M$ to $N$. The degree $n$ summand of $\hom_{\gr\Mod(R)}(M,N)$ is the group of strictly grading-preserving $R$-module homomorphisms $\Sigma^nM \rightarrow N$.
\item We use the term ``finite-type'' in the way which is standard in algebraic topology, not the way which is standard in algebraic geometry. That is, we will say that a graded vector space is ``finite-type'' if it is finite-dimensional in each individual degree.
\item Let $k$ be a field. We will say that a graded $k$-{\em algebra} is connected if it is concentrated in {\em nonnegative} degrees and its degree $0$ summand is one-dimensional. We will say that a graded $k$-{\em coalgebra} is connected if it is concentrated in {\em nonpositive} degrees and its degree $0$ summand is one-dimensional. We ask the reader to please trust us that, whenever doing careful work with the covariant embedding functor $\iota: \gr\Comod(\Gamma)\rightarrow\gr\Mod(\Gamma^*)$, this grading convention avoids many headaches, and is worth using. For a justification for this convention, see \cref{Generalities on...}.

The topological way to understand this convention is that it is the {\em cohomological}, rather than homological, grading convention. In particular, whenever $X$ is a space or a spectrum, with our grading conventions we treat its cohomology $H^*(X; \mathbb{F}_p)$ as a graded module over the mod $p$ Steenrod algebra in the usual way, and we must treat its homology $H_*(X; \mathbb{F}_p)$ as a graded comodule over the mod $p$ dual Steenrod algebra {\em with the degrees inverted}, i.e., we put the $n$th homology group $H_n(X; \mathbb{F}_p)$ in degree $-n$. For this reason, we write $H_{-*}(X; \mathbb{F}_p)$, rather than $H_*(X;\mathbb{F}_p)$, when we consider the mod $p$ homology of a space as a graded comodule. 

This means we have to exercise some care about the gradings in Adams spectral sequences as well. 
We try to explain the gradings carefully each time there seems to be some opportunity for confusion. 
\item We use the terms ``bounded above'' and ``bounded below'' with their standard meanings in algebra and topology. Explicitly, a graded object (module, comodule, etc.) is {\em bounded above} if it is trivial in all sufficiently large positive degrees, and {\em bounded below} if it is trivial in all sufficiently large negative degrees. 

In topological settings: a spectrum is said to be bounded above or bounded below if its homotopy groups are.
\item We write $\mathbb{N}^{\op}$ for the poset of natural numbers with the opposite of their usual order. Consequently, given a category $\mathcal{C}$, a functor $\mathbb{N}^{\op}\rightarrow\mathcal{C}$ is simply an inverse sequence in $\mathcal{C}$. Whenever convenient, we write $\mathcal{C}^{\mathbb{N}^{\op}}$ for the category of inverse sequences in $\mathcal{C}$.
\item Most of this paper concerns comodules over coalgebras over fields, but occasionally (specifically in \cref{V section} and \cref{v for e1}) we will get some topological mileage from considering comodules over Hopf algebroids. We recommend \cite[appendix 1]{MR860042} for readers who would like an introductory account of Hopf algebroids and the basics of homological algebra in their comodule categories. We think that the reader does not need background knowledge about Hopf algebroids to read this paper, though, since we have made an effort to provide the relevant background ideas wherever they are used in this paper. Readers who are not motivated by the Hopf algebroids in algebraic topology and who are interested only in coalgebras (or Hopf algebras), not Hopf algebroids, can reasonably skip over the material about Hopf algebroids in this paper.
\item Occasionally we will refer to ideas from relative homological algebra, e.g. relative injectives, relative cotorsion, and relative $\Ext$-groups in certain categories of comodules. The reader is not required to care about this, and indeed it only ever comes up when Hopf algebroids have already come up, so readers who skip the material on Hopf algebroids will find themselves skipping over the relative homological-algebraic content as well. 

In most of the main results in this paper which involve a relative homological-algebraic construction, we prove that the relative construction winds up being isomorphic to the corresponding {\em absolute} (i.e., classical) homological-algebraic construction. 
Consequently the reader who would prefer to be sloppy about the relative vs. absolute distinction can safely ignore the distinction, for the purposes of this paper. 

We give a brief review of the relevant background on relevant homological algebra in \cref{rel hom alg review}. In \cref{rel ext and adams E2} we explain why a careful reader might expect we would need to be careful about relative vs. absolute homological-algebraic constructions in this paper, and we explain why the two types of construction wind up being isomorphic in every situation we study here.
\item Whenever we have a graded Hopf algebroid $(A,\Gamma)$ with $\Gamma$ flat over $A$, when we write $\Ext^*_{\gr\Comod(\Gamma)}(M,N)$, we shall always mean the right-derived functors of $\hom$ in the category $\gr\Comod(\Gamma)$ of graded left $\Gamma$-comodules. 
In the rare cases where there is some need to consider relative $\Ext$ groups, we will write $\Ext^*_{\text{rel}\gr\Comod(\Gamma)}(M,N)$ for the {\em relative} right-derived functors of $\hom$ in the category $\gr\Comod(\Gamma)$ of graded left $\Gamma$-comodules, relative to the standard allowable class on $\gr\Comod(\Gamma)$, defined in \cref{rel hom alg review}.
\item We use the localization notation of Atiyah--Macdonald \cite{MR0242802}, e.g. $\mathbb{Z}_{(p)}$ is the ring of $p$-local integers.
\item Whenever we have occasion to consider homotopy groups in this paper, they shall always be {\em stable} homotopy groups. Although occasionally we mention topological spaces in this paper, we have made an effort to only write $\pi_*(X)$ in situations where $X$ is a spectrum, not a space, so that there can be no mistaking that $\pi_*$ must mean the {\em stable} homotopy groups. For emphasis, we state it explicitly here: in this paper, $\pi_*$ always means {\em stable} homotopy groups.
\item We make many spectral sequence arguments in this paper, none of which are difficult. In many cases the relevant spectral sequence is a Grothendieck spectral sequence, and in those cases we do not fuss about convergence questions, since convergence is straightforward for first-quadrant spectral sequences of that kind. At times we use other spectral sequences which do not have such straightforward and obvious convergence properties, and in those cases we check the conditions for either strong convergence (after Boardman \cite{MR1718076}) or complete convergence (after Bousfield \cite{MR551009}), as the case warrants.
\end{itemize}

\subsection{Acknowledgments}

We are grateful to Alex Martsinkovsky for illuminating conversations about ideas from the paper \cite{MR4045216}. The influence of the ideas of Martsinkovsky and Russell on \cref{V section} of this paper should be very clear.
 
No artificial intelligence tools were used for any part of this paper.

\section{Review of relevant ideas}

\subsection{Review of the covariant embedding of the category of $\Gamma$-comodules into the category of $\Gamma^*$-modules}
\label{Generalities on...}

Let $k$ be a field. The $k$-linear dual of a $k$-coalgebra $\Gamma$ is a $k$-algebra $\Gamma^*$, and there are two well-known ways of producing a graded $\Gamma^*$-module from a graded $\Gamma$-comodule $M$:
\begin{itemize}
\item Take the $k$-linear dual $M^*$. This construction will be immediately familiar to any topologist, since the mod $p$ homology $H_*(X; \mathbb{F}_p)$ of any space or spectrum $X$ will be a graded comodule over the dual Steenrod algebra, and its $\mathbb{F}_p$-linear dual is the cohomology $H^*(X; \mathbb{F}_p)$ of $X$, naturally a module over the Steenrod algebra. See \cite{MR686116} for a careful and structured discussion of possible variations on this point.

Sending a graded $\Gamma$-comodule to its $k$-linear dual $\Gamma^*$-module yields a {\em contravariant} functor from graded $\Gamma$-comodules to graded $\Gamma^*$-modules.
\item Take $M$ itself, with the adjoint action of $\Gamma^*$. This construction is well-known among algebraists (e.g. see the textbook treatment in \cite{MR2012570}), but seems obscure among topologists. The author learned of it from a conversation with Mike Boardman many years ago; in \cite[section 6]{MR686116}, Boardman calls the resulting fact that the Steenrod algebra acts on the {\em homology} of any space or spectrum ``one of the best-kept secrets of stable homotopy theory.''

The ``adjoint action'' is the map $\Gamma^*\otimes_k M \rightarrow M$ given by sending $f\otimes m$ to the image of $m$ under the composite map 
\[ M \stackrel{\psi_M}{\longrightarrow} \Gamma\otimes_k M \stackrel{f\otimes M}{\longrightarrow} k\otimes_k M \stackrel{\cong}{\longrightarrow} M,\]
where $\psi_M$ is the coaction map.

Sending a graded $\Gamma$-comodule $M$ to $M$ itself, with adjoint action of $\Gamma^*$, yields a {\em covariant} functor from graded $\Gamma$-comodules to graded $\Gamma^*$-modules. We will write $\iota$ for this functor $\gr\Comod(\Gamma)\rightarrow \gr\Mod(\Gamma^*)$. 
\end{itemize}
The linear dualization functor often fails to preserve $\Ext$-groups between non-finite-type graded objects; see \cite[pgs. 51-52]{MR0251716} for discussion of this behavior over the Steenrod algebra. We shall ultimately see, in Theorem \ref{ext iso 1}, that $\iota$ does not have this problem. Consequently it is the functor $\iota$ that is of most direct importance for us. Here are some of its established properties, mostly from subsections 1.4.1, 1.4.3, 1.4.7, 1.7.1, and 3.20.1 of \cite{MR2012570}.
\begin{itemize}
\item The functor $\iota$ is known to be faithful and full; more generally, as long as $\Gamma$ is a coalgebra over a commutative ring $A$, the functor $\iota: \gr\Comod(\Gamma)\rightarrow \gr\Mod(\Gamma^*)$ is faithful, and it is full if $\Gamma$ is projective as an $A$-module. Here $\Gamma^*$ is the $A$-linear dual $A$-algebra of $\Gamma$.
\item The functor $\iota: \gr\Comod(\Gamma) \rightarrow \gr\Mod(\Gamma^*)$ is an equivalence of categories if and only if $\Gamma$ is finitely generated and projective as an $A$-module. 
\item The graded $\Gamma^*$-modules in the image of $\iota$ are said to be {\em rational} comodules. This usage of the word ``rational'' is traditional in the literature on comodules. It is apparently unrelated to the usage of the word ``rational'' to mean ``$\mathbb{Q}$-linear'' elsewhere in algebra.
\item Since $\iota$ is faithful and full, the category of rational graded $\Gamma^*$-modules is equivalent to the category of graded $\Gamma$-comodules.
\item The rational comodules form an abelian subcategory of $\gr\Mod(\Gamma^*)$, closed under kernels, cokernels, and coproducts (but not necessarily products).
\item The functor $\iota$ has a right adjoint, written $\tr$, and called the ``trace functor'' or the ``rational functor.'' It sends each graded $\Gamma^*$-module to its maximal rational submodule, considered as a $\Gamma$-comodule.
\item Specialize now to the case where $A=k$, some field. Then the $k$-vector space $\Gamma$ is finite-dimensional if and only if every graded $\Gamma^*$-module is rational.
\item We must grade $\Gamma^*$ so that linear functionals on the degree $n$ summand of $\Gamma$ are considered to be in degree $-n$ in $\Gamma^*$. (This is so that the functor $\iota: \gr\Comod(\Gamma)\rightarrow\gr\Mod(\Gamma^*)$ is indeed the identity functor on the underlying graded $k$-vector spaces, rather than inverting all the grading degrees.)
\end{itemize}

Because of the last point, we must adopt some convention on what ``connected'' ought to mean. Recall that, for a field $k$, it is standard to say that a graded $k$-algebra is ``connected'' if it is concentrated in nonnegative degrees, and its degree $0$ summand consists only of the $k$-linear span of the identity element $1$. In this paper our algebras arise as linear duals of coalgebras, so by the last point in the above list, if we are going to have our algebras concentrated in nonnegative degrees, we must grade our coalgebras so that they are concentrated in nonpositive degrees. This is why we adopt the following convention, also described in \cref{conventions}:
\begin{convention}\leavevmode
\begin{itemize}
\item We will say that a graded {\em algebra} over a field is connected if it is concentrated in {\em nonnegative} degrees, and its degree $0$ summand is one-dimensional.
\item We will say that a graded {\em coalgebra} over a field is connected if it is concentrated in {\em nonpositive} degrees, and its degree $0$ summand is one-dimensional.
\end{itemize}
\end{convention}

It is potentially confusing for the meaning of ``connected'' for a graded coalgebra to be the reverse of its meaning for a graded algebra, but the author maintains (based on personal experience) that whatever confusion this causes is far less than the confusion caused by 
having the covariant embedding functor $\iota$ invert all the grading degrees. It is really essential to ensure that $\iota$ does not invert grading degrees, but is merely the identity functor on the underlying graded vector spaces, because that is what enables us to take the basic perspective of this paper (and of \cite{MR5014030}): identify the graded $\Gamma$-comodules with the rational graded $\Gamma^*$-modules, so that a graded comodule is nothing but a special kind of graded module, so that all our intuitions and techniques for modules can be brought to bear on comodules as well.

There is a standard meaning of the term ``connected coalgebra'' in the ungraded setting: namely, a coalgebra over a field is said to be {\em connected} if its coradical is one-dimensional. The {\em coradical} of a coalgebra is the sum of its simple subcoalgebras; see \cite[Definition 5.1.5]{MR1243637} for a textbook treatment. Our convention for the meaning of the term ``connected coalgebra,'' in the graded setting, is compatible with its standard meaning in the ungraded setting, in the following sense: given any $\mathbb{N}$-graded or $\mathbb{Z}^{\leq 0}$-graded coalgebra $\Gamma$, the sum of the simple {\em graded} subcoalgebras of $\Gamma$ is precisely the degree $0$ summand of $\Gamma$.

\subsection{Review of basics about Hopf algebroids}

Let $(A,\Gamma)$ be a graded Hopf algebroid. The standard reference for the subject is \cite[Appendix 1]{MR860042}. 
We use standard notations for the structure maps of $(A,\Gamma)$: we write $\eta_L,\eta_R: A \rightarrow \Gamma$ for the left and right unit maps, $\epsilon: \Gamma\rightarrow A$ for the augmentation map, $\Delta: \Gamma \rightarrow \Gamma\otimes_A\Gamma$ for the coproduct map, and $\chi: \Gamma\rightarrow\Gamma$ for the conjugation (i.e., antipode) map.

Hopf algebroids arise only occasionally in this paper: in \cref{v for e1} and \cref{rel ext and adams E2} due to topological examples, and in \cref{V and cotorsion} simply due to the observation that the construction of the right adjoint $V$ for the extended comodule functor makes sense over a Hopf algebroid, not only a Hopf algebra. Readers who do not need the generality of Hopf algebroids can mentally replace all instances of the phrase ``Hopf algebroid'' with ``Hopf algebra'' as they read this paper, and correspondingly to assume $\eta_L = \eta_R$, without affecting the truth of any of the results or the correctness of the proofs. 

Here are some well-known facts about categories of modules and comodules associated to $(A,\Gamma)$. 
\begin{itemize}
\item The following conditions are equivalent: 
\begin{itemize}
\item the left unit map $\eta_L$ is flat,
\item the right unit map $\eta_R$ is flat,
\item both unit maps $\eta_L,\eta_R$ are flat.
\end{itemize}
When these equivalent conditions are satisfied, then the graded Hopf algebroid $(A,\Gamma)$ is said to be {\em flat}. Flatness of $(A,\Gamma)$ ensures that the category $\gr\Comod(\Gamma)$ of graded $\Gamma$-comodules is abelian \cite[Lemma A1.1.3]{MR860042}. 
\item For flat $(A,\Gamma)$, a sequence of graded $\Gamma$-comodules is exact if and only if its underlying sequence of graded $A$-modules is exact.
\item Suppose $(A,\Gamma)$ is flat. Then the forgetful functor 
\begin{align}
\label{u 1} U: \gr\Comod(\Gamma) &\rightarrow \gr\Mod(A) \\
\nonumber M &\mapsto M
\end{align}
has a right adjoint, the {\em extended comodule functor},
\begin{align}
\label{E 1} E: \gr\Mod(A) &\rightarrow \gr\Comod(\Gamma) \\
\nonumber M &\mapsto \Gamma\otimes_A M.
\end{align}
The $\Gamma$-comodules of the form $\Gamma\otimes_A M$ for some $A$-module $M$ are called {\em extended comodules.} The summands of extended comodules are called the {\em relatively injective} $\Gamma$-comodules. 
\end{itemize}
%
%
%

\section{$V$, the right adjoint to the extended comodule functor}
\label{V section}

\subsection{$V$ and cotorsion}
\label{V and cotorsion}

Here is a simple observation which, despite its simplicity, we have never seen nor heard mentioned anywhere. The functor $E$ from \eqref{E 1} is easily seen to preserve coproducts and cokernels, hence all colimits, and furthermore its domain category $\gr\Mod(A)$ is co-complete, locally small, co-well-powered, and has a small generating set (namely, the set of all suspensions of the free module $A$)---that is, $E$ satisfies the hypotheses of Freyd's famous ``Special Adjoint Functor Theorem,'' so $E$ must have a right adjoint. 
\begin{definition}
We will write $V$ for the right adjoint 
\begin{align*}
 V:\gr\Comod(\Gamma) &\rightarrow \gr\Mod(A)\end{align*} 
of the extended comodule functor $E$.
\end{definition}

Once one notices that this right adjoint exists, it is very easy to give a formula for it, as follows.
\begin{prop}\label{I prop 1}
The functor $V$ is co-represented by the comodule $\Gamma$. That is, we have an isomorphism
\begin{align*}
 \hom_{\gr\Comod(\Gamma)}(\Gamma,M) &\cong V(M),
\end{align*}
natural in the variable $M$.
\end{prop}
\begin{proof}
We have isomorphisms
\begin{align*}
 \hom_{\gr\Comod(\Gamma)}(\Gamma,M) 
  &\cong\hom_{\gr\Comod(\Gamma)}(E(A),M) \\
  &\cong\hom_{\gr\Mod(A)}(A,V(M)) \\
  &\cong V(M)
\end{align*}
of graded abelian groups, natural in the variable $M$.
\end{proof}

The following definitions are generalizations of definitions in \cite{MR4045216}, which are there given only for categories of modules over a ring.
\begin{definition}\label{m-r defs}
Let $\mathcal{C}$ be an abelian category. We assume that $\mathcal{C}$ is equipped with some choice of allowable class (as in \cref{rel hom alg review}), so that we may speak of ``relative injectives'' with respect to this allowable class.
\begin{itemize}
\item The {\em relatively injective trace} of $M$, written $Itr(M)$, is the maximal subobject of $M$ such that every map from a relatively injective object of $\mathcal{C}$ to $M$ factors through this subobject\footnote{There is the question of whether the relatively injective trace of an object $M$ actually exists. Clearly this can depend on $\mathcal{C}$, on the allowable class, and on $M$ itself. If $\mathcal{C}$ is the category of modules over a ring and we take the allowable class to be the absolute allowable class (so that the relative injectives are the injectives), or if we take $\mathcal{C}$ to be the category of comodules over any flat Hopf algebroid and we take the allowable class to be the ``standard class'' as in \cref{rel hom alg review}, then the injective trace of $M$ exists for any $M$. A similar existence result also holds, for straightforward reasons, for the graded (co)module categories. In each of these cases, the argument for existence of $Itr(M)$ is the same: while there may be a proper class of relative injectives, there is only a {\em set} of subobjects of $M$, and the ones which are images of a map from a relative injective comprise a filtered poset. The relatively injective trace $Itr(M)$ is simply the colimit of the members of that filtered poset. One needs to know that this colimit is still a subobject of $M$, but this follows from all these categories satisfying Grothendieck's axiom $AB5$ (existence of colimits and exactness of filtered colimits).}. 
\item The {\em relative cotorsion of $M$}, written $\mathfrak{q}(M)$, is the quotient of $M$ by the relatively injective trace of $M$.
\item We say that $M$ is {\em relative cotorsion} if the natural map $M \rightarrow \mathfrak{q}(M)$ is an isomorphism. Equivalently, $M$ is relative cotorsion if the relatively injective trace of $M$ is trivial. Equivalently, $M$ is relative cotorsion if there are no nonzero maps from a relatively injective object to $M$. Equivalently, when the allowable class is the standard allowable class from \cref{rel hom alg review}: $M$ is relative cotorsion if and only if $V(M)$ is trivial. This last equivalence is explained below, in Remark \ref{cotorsion remark}.
\end{itemize}
When the relevant allowable class is the absolute allowable class, we will write {\em cotorsion} rather than ``relative cotorsion.''
\end{definition}
In the paper \cite{MR4045216}, Martsinkovsky and Russell prove various good properties of these definitions in the absolute case. For example, when $\mathcal{C}$ is the category of modules over an integral domain, the Auslander--Gruson--Jensen transform of the functor $\mathfrak{q}$ agrees with the functor sending a module to its submodule of torsion elements. When $\mathcal{C}$ is the category of modules over an arbitrary ring $\Lambda$, Martsinkovsky and Russell prove that the Auslander--Gruson--Jensen transform of $\mathfrak{q}$, applied to a finitely presented $\Lambda$-module, agrees with the Bass torsion of $M$, i.e., the kernel of the natural map to the double dual $M\rightarrow\hom_{\Lambda}(\hom_{\Lambda}(M,\Lambda),\Lambda)$. Since the AGJ transform is a kind of dualization in the category of additive functors, these results of Martsinkovsky--Russell establish a sense in which cotorsion is dual to torsion---hence the name ``cotorsion.''

The following variant definition will be useful later.
\begin{definition}\label{def of relative derived cotorsion}
Let $\mathcal{C}$ be an abelian category equipped with some choice of allowable class. Suppose that $\mathcal{C}$ has enough relative injectives, and that the relatively injective trace of each object in $\mathcal{C}$ exists. 
We will say that an object $M$ of $\mathcal{C}$ is {\em relative derived cotorsion} if the relative right-derived functors $R_{\text{rel}}^nV(M)$ are trivial for all $n\geq 0$. 

When the allowable class is the absolute allowable class, we will write {\em derived cotorsion} rather than ``relative derived cotorsion.''
\end{definition}

Unless otherwise specified, when we consider relative homological constructions on a category of comodules (e.g. relative cotorsion or relative derived cotorsion), the relevant allowable class will always be the standard class, as defined in \cref{rel hom alg review}.

\begin{prop}\label{cotorsion and relative cotorsion}
Let $(A,\Gamma)$ be a graded Hopf algebroid, with $\Gamma$ flat over $A$. Let $M$ be a graded $\Gamma$-comodule.
\begin{enumerate}
\item If $M$ is relative cotorsion, then $M$ is cotorsion. The converse is true if $A$ is a field. 
\item The following conditions on $M$ are equivalent: $V(M)\cong 0$; $M$ is relative cotorsion; $Itr(M)\cong 0$.
\item If $\Gamma$ is projective over $A$, then $R^nV(M) \cong R^n_{rel}V(M)$ for all $n$.
\end{enumerate}
\end{prop}
Let us be careful about Proposition \ref{cotorsion and relative cotorsion}, because it is an easy place to make a mistake: the last point in Proposition \ref{cotorsion and relative cotorsion} does {\em not} entail that every derived cotorsion comodule $M$ is also relative derived cotorsion. If one knows that $R^*V(M)$ vanishes, then the last point in Proposition \ref{cotorsion and relative cotorsion} indeed gives vanishing of $R^*_{rel}V(M)$, and consequently $M$ would be relative derived cotorsion. The problem is that one cannot deduce vanishing of $R^*V(M)$ from $M$'s being derived cotorsion. Similarly, without some special hypothesis (e.g. that $A$ is a field), one cannot deduce that every cotorsion comodule is relative cotorsion. 

Here is a very simple example to demonstrate the obstruction: let $(A,\Gamma) = (\mathbb{Z},\mathbb{Z}[x]/x^2)$, with $x$ primitive and in degree $-1$. Since $\Gamma$ is projective and finite-rank over $A$, the covariant embedding $\iota:\gr\Comod(\Gamma)\rightarrow\gr\Mod(\Gamma^*)$ is an equivalence of categories, so we will identify the graded $\Gamma$-comodules with the graded modules over the ring $\mathbb{Z}[y]/y^2$ with $y$ in degree $1$. 

We have the free-forgetful adjunction between $\mathbb{Z}[y]/y^2$ and $\mathbb{Z}$, and since the free functor in this case is exact, its right adjoint, the forgetful functor, preserves injectives. Hence every injective $\mathbb{Z}[y]/y^2$-module is divisible as an abelian group. The {\em relative} injectives, by contrast, are all the $\mathbb{Z}[y]/y^2$-modules which are tensored up from $\mathbb{Z}$.

In the case $M = \mathbb{Z}[y]/y^2$, there are no nonzero $\mathbb{Z}[y]/y^2$-module maps from a divisible module to $M$, but there are nonzero maps from a relative injective to $M$, namely, the identity map $M\rightarrow M$. Hence $M$ is cotorsion, but not relative cotorsion, even though $\Gamma$ is projective over $A$.
\begin{proof}[Proof of Proposition \ref{cotorsion and relative cotorsion}]
\begin{enumerate}
\item
Every injective is a relative injective, so if all the maps from relative injectives to $M$ are zero, the same is true for maps from injectives to $M$. This proves the first claim.
\item One shows easily that $Itr(M) \cong 0$ if and only if there are no nonzero maps from suspensions of $\Gamma$ to $M$. That is, using Proposition \ref{I prop 1}: $V(M)$ is zero if and only if the relatively injective trace $Itr(M)$ is zero, i.e., if and only if $M$ is relative cotorsion. 
\item
For an arbitrary graded $\Gamma$-comodule $L$, the assumption that $L$ is projective over $A$ ensures that $\Ext^n_{\gr\Comod(\Gamma)}(L,M) \cong \Ext^n_{\text{rel}\gr\Comod(\Gamma)}(L,M)$ for all $n,M$; see \cite[pg. 54]{MR0251716} and \cite[Lemma 3]{MR686121} for one argument (attributed to Ravenel by the author of \cite{MR686121}), or \cite[Lemma A1.2.9(b) and Corollary A1.2.12]{MR860042} for another argument due to Ravenel. 

In the case $L = \Gamma$, this yields the isomorphism $R^*V(M)\cong R^*_{\text{rel}}V(M)$, hence one vanishes if and only if the other vanishes.
\end{enumerate}
\end{proof}

\begin{remark}
\label{cotorsion remark}
One can read the second point in Proposition \ref{cotorsion and relative cotorsion} as meaning that $V(M)$ and $Itr(M)$ each ``detect the failure of $M$ to be relative cotorsion.'' Nevertheless it is not generally true that $V(M)$ is isomorphic to $Itr(M)$. Since the construction of $Itr(M)$ as a colimit yields a canonical map from the relative injective $E(V(M))$ to the relatively injective trace $Itr(M)$, by precomposition with the unit\footnote{This notation $U$ for the forgetful functor was defined above, in \eqref{u 1}.} map $V(M) \hookrightarrow U(E(V(M))$ there is at least a natural map of graded $A$-modules $f: V(M) \rightarrow Itr(M)$. The map $f$ is easily seen to be surjective. To see that it is not always injective, consider the case $M = \Gamma$. We have bijections
\begin{align*}
 V\Gamma
  &\cong \hom_{\gr\Mod(A)}(A,V\Gamma) \\
  &\cong \hom_{\gr\Comod(\Gamma)}(EA,EA) \\
  &\cong \hom_{\gr\Mod(A)}(UEA,A) \\
  &\cong \hom_{\gr\Mod(A)}(\Gamma,A) \\
  &\cong \Gamma^*.
\end{align*}
For most choices of $\Gamma$ (e.g. the dual Steenrod algebra), the graded comodule $\Gamma^*$, unlike $Itr(\Gamma)$, cannot embed into $\Gamma$ in a grading-preserving way. In such cases $V\Gamma$ cannot agree with $Itr(\Gamma)$.
\end{remark}

\subsection{Topological interpretation of $V$}
\label{Topological interpretation of V}

The formula $\hom_{\gr\Comod(\Gamma)}(\Gamma,M) \cong V(M)$ from Proposition \ref{I prop 1} gives us an easy topological interpretation of $V$ and its derived functors via a generalized Adams spectral sequence. 

We need to explain a few notations relevant to Adams spectral sequences. Most of these definitions and notations are standard in stable homotopy theory, but we mention them here because we hope that this is a paper that may be of interest to some pure algebraists who might not know this material from stable homotopy. Suppose we are given a ring spectrum $E$ and a spectrum $Y$. 
\begin{itemize}
\item The ring spectrum $E$ gives rise to a generalized homology theory $E_*$, called ``$E$-homology,'' via Whitehead's construction \cite{MR137117}: for a spectrum $X$, $E_*(X)$ is defined as the homotopy groups of the smash product $E\wedge X$. For a pointed space $X$, $E_*(X)$ is defined as the homotopy groups of the smash product spectrum $E\wedge \Sigma^{\infty} X$, where $\Sigma^{\infty} X$ is the suspension spectrum of $X$.
\item In particular, $E_*(E) = \pi_*(E\wedge E)$ is the ring of stable co-operations in $E$-homology. The graded rings $\pi_*(E)$ and $E_*(E)$ form a graded Hopf algebroid $(\pi_*E,E_*E)$ in a natural way. 
\item If $E_*E$ is flat over $\pi_*E$ then the category $\gr\Comod(E_*E)$ of graded comodules over the graded Hopf algebroid $(\pi_*E,E_*E)$ is abelian, and the extended comodule functor $\gr\Mod(\pi_*E)\rightarrow \gr\Comod(E_*E)$ admits a right adjoint $V: \gr\Comod(E_*E)\rightarrow \gr\Mod(\pi_*E)$. This is a special case of the $V$ functor studied in \cref{V section}.
\item $\hat{Y}_E$ denotes the $E$-nilpotent completion of $Y$, i.e., the totalization/homotopy limit of the cosimplicial spectrum
\[ \xymatrix{ 
E\wedge Y \ar@<1ex>[r] \ar@<-1ex>[r] & 
 E\wedge E\wedge Y  \ar[l]\ar@<2ex>[r]\ar[r]\ar@<-2ex>[r] & 
 E\wedge E\wedge E\wedge Y  \ar@<1ex>[l] \ar@<-1ex>[l] \ar@<3ex>[r] \ar@<1ex>[r] \ar@<-1ex>[r] \ar@<-3ex>[r] & \dots \ar@<2ex>[l]\ar[l]\ar@<-2ex>[l] . }\]
There is a natural map $Y\rightarrow \hat{Y}_E$. One says that $Y$ is {\em $E$-nilpotently complete} if this map is a weak equivalence.
\item If $E$ is furthermore assumed to be connective (i.e., $\pi_*(E)$ is trivial in negative degrees), then Bousfield proves \cite[Propositions 2.4 and 2.5 and Theorem 3.1]{MR551009}:
\begin{itemize}
\item if $\pi_0(E) \cong \mathbb{Q}$, then a bounded-below spectrum $X$ is nilpotently $E$-complete if and only if the homotopy groups of $X$ are $\mathbb{Q}$-vector spaces,
\item if $\pi_0(E) \cong \mathbb{Z}_{(p)}$, then a bounded-below spectrum $X$ is nilpotently $E$-complete if and only if the homotopy groups of $X$ are $\mathbb{Z}_{(p)}$-modules,
\item and if $\pi_0(E)\cong \mathbb{F}_{p}$, then a bounded-below spectrum $X$ is nilpotently $E$-complete if and only if the homotopy groups of $X$ are $L_0\Lambda_p$-complete, i.e., the natural map from $\pi_*(X)$ to the zeroth left derived functor of $p$-adic completion on $\pi_*(X)$ is an isomorphism.
\end{itemize}
\end{itemize}
\begin{prop}\label{adams E2 description} 
Let $Y$ be a spectrum. 
Let $E$ be a ring spectrum such that the graded ring $E_*E$ is projective as a module over the ring $\pi_*(E)$, and such that\footnote{This condition on the $E$-Adams $E_2$-page is well-known to be satisfied for very many ring spectra $E$ of practical interest, for example, the Eilenberg-Mac Lane spectra $H\mathbb{F}_p$ and $H\mathbb{Q}$, the periodic complex $K$-theory spectrum $KU$ and its $p$-localization $KU_{(p)}$ and its $p$-local Adams summand $L \simeq E(1)$, and more generally, any of the Johnson--Wilson spectra $E(n)$. For other examples see \cite[Proposition III.13.4]{MR1324104}. A standard reference for a sufficient condition to ensure that this property on the $E$-Adams $E_2$-page is satisfied is \cite[sections III.13 and III.15]{MR1324104}. Adams's sufficient condition is that $E$ is a filtered homotopy colimit of finite spectra $E_i$ whose Spanier--Whitehead duals $DE_i$ have the property that $E_*(DE_i)$ is projective over $\pi_*(E)$ and the natural map $F^*(DE_i) \rightarrow \hom_{\pi_*(E)}(E_*(DE_i),\pi_*(F))$ is an isomorphism for all $E$-module spectra $F$. } the $E$-Adams spectral sequence for the function spectrum $F(E,Y)$ has $E_2$-page $E_2^{s,t} \cong \Ext_{\gr\Comod(E_*E)}^{s,t}(E_*E, E_*Y)$. 

Then that $E$-Adams spectral sequence takes the form
\begin{align}
\label{e-adams ss 1} E_2^{s,t} \cong (R^sV(E_{-*}Y))^{-t} &\Rightarrow [\Sigma^{t-s} E,\hat{Y}_E]\\
\nonumber d_r: E_r^{s,t} &\rightarrow E_r^{s+r,t+r-1}.
\end{align}

The spectral sequence is conditionally convergent in the sense of Boardman \cite{MR1718076}. It is also strongly convergent in Boardman's sense \cite{MR1718076}, and completely convergent in Bousfield's sense \cite{MR551009}, if the derived limit $\underset{r\rightarrow\infty}{\lim^1} E_r^{s,t}$ vanishes for all $s,t$.

In the case that $E$ is rational homology, or the $p$-local Adams summand $E(1)$, or any of the higher Johnson--Wilson theories $E(n)$, we have that the $E$-nilpotent completion $\hat{Y}_E$ is weakly equivalent to the Bousfield $E$-localization $L_E Y$, and furthermore the spectral sequence converges completely in the sense of Bousfield \cite{MR551009}.
\end{prop}
The notation $(R^sV(E_{-*}Y))^{-t}$ in \eqref{e-adams ss 1} is supposed to denote the degree $-t$ summand in the graded abelian group $R^sV(E_{-*}Y)$. In that formula, the reason we see $-t$ rather than $t$, and $-*$ rather than $*$, is our cohomological grading conventions on comodules; see \cref{conventions}. 

To be clear, the right-derived functors $R^sV(E_{-*}Y)$ here are the usual, classical, {\em absolute} right-derived functors, not relative right-derived functors, but in fact the hypotheses of the proposition suffice to make the absolute right-derived functors $R^sV(E_{-*}Y)$ coincide with the relative right-derived functors $R^s_{\text{rel}}V(E_{-*}Y)$, by Proposition \ref{cotorsion and relative cotorsion}.
\begin{proof}[Proof of Proposition \ref{adams E2 description}]
The proposition is only a slight rephrasing of standard foundational work on the $E$-Adams spectral sequence
\begin{align}
\label{e-adams ss 2} 
 E_2^{s,t} \cong \left(\Ext_{\gr\Comod(E_*E)}^{s}(E_*X,E_*Y)\right)^{-t} &\Rightarrow [\Sigma^{t-s} X,\hat{Y}_E]
\end{align}
carried out in the 1970s, from \cite[section III.15]{MR1324104} (where one finds the claims about the $E_2$-page) and \cite[section 6]{MR551009} (where one finds the convergence claims). Again, the appearance of internal degree $-t$ rather than $t$ is a consequence of our cohomological grading conventions on comodules (\cref{conventions}), which are the reverse of the usual convention for gradings on homology comodules in topology, so one will find $t$ rather $-t$ in the analogue of formula \eqref{e-adams ss 2} in the literature.

In the case $X = E$, the $E_2$-page described as in \eqref{e-adams ss 2} reduces to what is described in \eqref{e-adams ss 1} as a consequence of Proposition \ref{I prop 1}. See \cref{rel ext and adams E2} for discussion, if desired, of relative vs. absolute $\Ext$-groups in the description of the $E$-Adams $E_2$-page.

In the case that $E$ is rational homology or, more generally, a Johnson--Wilson theory, the claim $\hat{Y}_E\simeq L_EY$ and the complete convergence claim follow from Bousfield's convergence analysis \cite[Theorem 6.10]{MR551009} together with the proof of Ravenel's conjecture that, for any Johnson--Wilson theory $E$, every spectrum is $E$-prenilpotent \cite{MR737778}, proven ultimately by Devinatz--Hopkins--Smith \cite{MR0960945},\cite{MR1652975},\cite{MR1192553}.
\end{proof}

The most interesting things we have to say in this paper about topological applications of $R^*V$ are a matter of using Proposition \ref{adams E2 description} in the case where $E$ is the Eilenberg-Mac Lane spectrum $H\mathbb{F}_p$, so that $E_*$ is classical mod $p$ homology $H_*(-; \mathbb{F}_p)$, and $E_*E$ is the dual mod $p$ Steenrod algebra. In \cref{Topological interpretation of V...} we also consider how Proposition \ref{adams E2 description} plays out for a few other choices of ring spectrum $E$, namely $E = H\mathbb{Q}$ and $E = E(1)$.

\subsection{$V$ for the coalgebra dual to a Poincar\'{e} algebra}
\label{poincare alg section}

We follow Moore--Peterson \cite{MR335572} and Margolis \cite[section 12]{MR738973} in the definitions of a Poincar\'{e} algebra and a Frobenius algebra. While the latter definition is standard and commonplace, to our knowledge the notion of a ``Poincar\'{e} algebra'' in this sense was original with Moore and Peterson.
\begin{definition}\label{def of poincare algebra}
Let $k$ be a field, and let $A$ be a finite-dimensional connected graded $k$-algebra.
\begin{itemize}
\item
We say that $A$ is a {\em Poincar\'{e} algebra} if there is a map of graded $k$-modules $e: A \rightarrow \Sigma^n k$, for some $n$, such that the resulting pairing
\[
 A^q \otimes_k A^{n-q} \stackrel{\nabla}{\longrightarrow} A^n \stackrel{e}{\longrightarrow} \Sigma^n k
\]
is nondegenerate.
\item 
We say that $A$ is a {\em Frobenius algebra} if for some $n$ there is an isomorphism of graded left $A$-modules $A \stackrel{\cong}{\longrightarrow} \Sigma^n A^*$.
\end{itemize}
\end{definition}
The module theory of a Poincar\'{e} algebra $A$ is quite constrained: for an $A$-module $M$---not necessarily finitely generated!---freeness, projectivity, flatness, and injectivity of $M$ are all equivalent, and if $M$ fails to be free, then the projective, weak, and injective dimensions of $M$ all coincide \cite[Theorem 4.2]{MR335572},\cite[Proposition 12.2.8]{MR738973}. This is really a consequence of well-known facts about Frobenius algebras, together with the theorem \cite{MR335572},\cite[Theorem 12.2.5]{MR738973} that a finite-dimensional connected graded $k$-algebra is Poincar\'{e} if and only if it is Frobenius. Finite-dimensional connected Hopf $k$-algebras, in particular, are Poincar\'{e} \cite{MR335572},\cite[Theorem 12.9]{MR738973}. 

It is also well-known (e.g. see \cite[subsection 1.4.7]{MR2012570}) that the covariant embedding $\iota: \gr\Comod(\Gamma)\rightarrow\gr\Mod(\Gamma^*)$ is an equivalence of categories if $\Gamma$ is finite-dimensional over a field. Hence, for a graded coalgebra $\Gamma^*$ whose dual is a Poincar\'{e} algebra, the functor $V: \gr\Mod(k)\rightarrow\gr\Comod(\Gamma)$ has extremely simple behavior:
\begin{align*}
 V(M) 
  &= \hom_{\gr\Comod(\Gamma)}(\Gamma,M) \text{\ by\ Proposition \ref{I prop 1}} \\
  &= \hom_{\gr\Mod(\Gamma^*)}(\iota\Gamma,\iota M) \text{\ by\ fullness\ of\ }\iota \\
  &= \hom_{\gr\Mod(\Gamma^*)}\left((\Gamma^*)^*,\iota M\right) \text{\ by\ finite-dimensionality\ of\ } \Gamma \\
  &= \hom_{\gr\Mod(\Gamma^*)}(\Sigma^{-n}\Gamma^*,\iota M) \text{\ for\ some\ $n$,\ since\ $\Gamma^*$\ is\ Frobenius} \\
  &= \Sigma^{n}M,
\end{align*}
where $n$ is the degree of the duality class (i.e., ``volume form'') in $\Gamma^*$, as in Definition \ref{def of poincare algebra}.

We find it useful to provide a very simple worked example, to clarify what is happening with the signs on the grading degrees. Let $\Gamma$ be the graded Hopf algebra $k[t]/t^2$, with $t$ primitive in degree $-1$. A homogeneous degree $0$ element of $VM$ is, by adjunction, given by a graded $\Gamma$-comodule map from $\Gamma = Ek$ to $M$. Such a map picks out a homogeneous degree $-1$ element of $M$, namely, the image of $t$. Hence $VM = \hom_{\gr\Comod(\Gamma)}(Ek,M) = \Sigma M$. 

\subsection{Recollections about $\mathcal{P}$-algebras}

In \cref{poincare alg section} we saw that the functor $V$ has extremely simple behavior in the case where the underlying coalgebra is a Poincar\'{e} algebra. One cannot expect such simple behavior for arbitrary connected graded algebras. There is an intermediate class between the two, the ``$\mathcal{P}$-algebras'' of Margolis \cite{MR738973}:

\begin{definition}\label{def of p-alg}
Let $k$ be a field. A connected graded $k$-algebra $A$ will be called a {\em $\mathcal{P}$-algebra} if $A$ is the union of an increasing chain of subalgebras $A(0)\subsetneq A(1) \subsetneq A(2) \subsetneq \dots$ such that
\begin{itemize}
\item each $A(n)$ is a Poincar\'{e} algebra (see Definition \ref{def of poincare algebra}),
\item and, for each $n$, $A(n+1)$ is flat as a right $A(n)$-module.
\end{itemize}
\end{definition}
While a $\mathcal{P}$-algebra is never Noetherian \cite[Proposition 13.1.2]{MR738973}, it is at least left coherent \cite[Proposition 13.1.1]{MR738973} and right coherent \cite[Corollary 13.1.7]{MR738973}. Margolis proves, in Theorems 13.2.5 and 13.3.12 of \cite{MR738973}, that $\mathcal{P}$-algebras enjoy some curious analogues of the nice homological properties of Poincar\'{e} algebras sketched in \cref{poincare alg section}:
\begin{theorem} {\bf (Margolis.)} \label{margolis thm 1}
Let $A$ be a $\mathcal{P}$-algebra, and let $M$ be a graded left $A$-module.
Then the following conditions are equivalent:
\begin{itemize}
\item $M$ is flat,
\item $M$ has projective dimension $\leq 1$,
\item $M$ has injective dimension $\leq 1$,
\item $M$ is free over $A(n)$ for each $n$.
\end{itemize}
If $M$ does not satisfy these conditions, then its projective, weak, and injective dimensions are all infinite.

Suppose furthermore that $M$ is bounded below, i.e., $M$ is zero in sufficiently low negative degrees. Then the following conditions are equivalent:
\begin{itemize}
\item $M$ is flat,
\item $M$ is projective,
\item $M$ is injective, 
\item $M$ is free over $A(n)$ for each $n$,
\item $M$ is free.
\end{itemize}
\end{theorem}

The motivating example of a $\mathcal{P}$-algebra is the mod $p$ Steenrod algebra, for each prime $p$. The $\mathcal{P}$-algebras were introduced by Margolis in \cite{MR341488} and in \cite{MR738973}, apparently because the results originally proved by Adams--Margolis in \cite{MR294450} and by Lin--Margolis in \cite{MR458426} generalize to all $\mathcal{P}$-algebras. In particular, Theorem \ref{margolis thm 1} was originally obtained by Lin--Margolis in \cite{MR458426} specifically for the Steenrod algebras, before Margolis's generalization of that result to $\mathcal{P}$-algebras in \cite{MR738973}. Aside from the Steenrod algebras, the other familiar examples of $\mathcal{P}$-algebras are exterior algebras on countably infinitely many generators in positive degrees, and connected Hopf algebras which are unions of their finite-dimensional Hopf subalgebras.

In \cref{Review of the Steenrod algebras} of the appendix we offer a quick review of the most important features of the mod $p$ Steenrod algebra. Pure algebraists who do not know the Steenrod algebras may find that review useful.

The definitions and recollections in this subsection will be used in \cref{Technical results on distinguished local cohomology} starting in Proposition \ref{better version of main thm 1}. In \cref{I in the case of a P-algebra} and \cref{I in the case of the Steenrod alg} we will prove basic homological properties of the behavior of $V: \gr\Comod(\Gamma)\rightarrow \gr\Mod(k)$ when $\Gamma^*$ is a $\mathcal{P}$-algebra.

\section{Ext in comodules as a case of Ext in modules}

Sections 2 and 3 consisted entirely of review, definitions, and a few straightforward observations. It is in this section that finally begin to put in some work and derive some new results.

\subsection{Technical results on distinguished local cohomology}
\label{Technical results on distinguished local cohomology}

The author apologizes for not having found a way to prove our main result comparing $\Ext$ in comodules to $\Ext$ in modules, Theorem \ref{ext iso 1}, without bringing in some annoyingly technical ideas from the author's own paper \cite{MR5014030}. We have at least confined these technical ideas to their own subsection.

In \cite{MR5014030} we made the following definitions: 
\begin{definition}
Suppose $k$ is a field and $\Gamma$ is a connected $k$-coalgebra. 
\begin{itemize}
\item
A homogeneous left ideal $I$ of $\Gamma^*$ is {\em strongly distinguished} if the graded $\Gamma^*$-module $\iota(\Gamma)$ contains a graded $\Gamma^*$-submodule isomorphic to a suspension of $\Gamma^*/\Gamma^*I$. 
\item 
A homogeneous left ideal $I$ of $\Gamma^*$ is {\em distinguished} if it contains an intersection of a finite set of strongly distinguished ideals. Write $\dist$ for the partially-ordered set of distinguished ideals of $\Gamma^*$.
\item 
Given a graded $\Gamma^*$-module $M$, we write $h^0_{dist}(M)$ for the graded subgroup \[ \coprod_{n\in\mathbb{Z}} \underset{I\in \dist}{\colim}\hom_{\gr\Mod(\Gamma^*)}\left( \Gamma^*/\Gamma^*I,M\right)\] of $M$. That is, $h^0_{dist}(M)$ is the left $I$-torsion in $M$, where $I$ ranges across all the distinguished ideals of $\Gamma^*$. Hence $h^0_{dist}$ is a functor $\gr\Mod(\Gamma^*)\rightarrow\gr\Ab$.
\item 
The {\em zeroth distinguished local cohomology functor} is the functor $H^0_{dist}: \gr\Mod(\Gamma^*)\rightarrow\gr\Mod(\Gamma^*)$ given by letting $H^0_{dist}(M)$ be the graded $\Gamma^*$-submodule of $M$ generated by the subgroup $h^0_{dist}(M)$ of $M$. That is, $H^0_{dist}(M)$ is the graded $\Gamma^*$-submodule generated by the left $I$-torsion in $M$, where $I$ ranges across all the distinguished ideals of $\Gamma^*$.
\item The {\em $n$th distinguished local cohomology functor} is the $n$th right derived functor of $H^0_{dist}$.
\end{itemize}
\end{definition}
Readers who are accustomed to the classical theory of local cohomology might be dismayed at the idea that $h^0_{dist}$ and $H^0_{dist}$ are not always and automatically identical. In classical local cohomology (e.g. as in the textbook \cite{MR3014449}, or as in the well-known applications in algebraic geometry in \cite{MR0222093}), one is used to the idea that, for an ideal $I$ in a commutative ring $R$ and for an $R$-module $M$, the $I$-power-torsion elements of $M$ comprise an $R$-submodule of $M$. 

However, for a noncommutative ring $R$, the left $I$-power-torsion elements of a left $R$-module $M$ will form an abelian subgroup of $M$, but not always a left $R$-submodule of $M$. This is the reason why, for a noncommutative algebra $\Gamma^*$, $h^0(M)$ might fail to be equal to $H^0(M)$, as subsets of $M$. See \cite[Remark 2.6]{MR5014030} for an explicit example, and see \cite[section 2]{MR5014030} for a discussion and development of some general theory. It turns out \cite[Theorem 5.9]{MR5014030} that, {\em when $\Gamma^*$ is the Steenrod algebra}, $h^0_{dist}(M)$ is already a $\Gamma^*$-submodule of $M$, and consequently $h^0_{dist}(M) = H^0_{dist}(M)$, but this is a surprisingly tricky thing to prove and uses some subtle properties of the Steenrod algebra.

In the literature on coalgebras (e.g. \cite{MR2012570} and \cite{MR1773034}), the modules in the essential image of the functor $\iota: \gr\Comod(\Gamma)\rightarrow\gr\Mod(\Gamma^*)$ are called the ``rational $\Gamma^*$-modules.'' Since $k$ is assumed a field, the category of rational graded $\Gamma^*$-modules is equivalent to the category of graded $\Gamma$-comodules. In terms of distinguished local cohomology, the rational graded $\Gamma^*$-modules are precisely the graded $\Gamma^*$-modules such that the natural inclusion $H^0_{dist}(M)\hookrightarrow M$ is an isomorphism \cite[Theorem 3.7]{MR5014030}.

\begin{lemma}\label{surj H0 lemma}
Let $B$ be a set, and let $d: B\rightarrow \mathbb{Z}$ be a function. Let $k$ be a field, let $\Gamma$ be a connected graded $k$-coalgebra, and let $M$ be a graded $\Gamma^*$-module.
The coproduct $\coprod_{b\in B} \Sigma^{d(b)} M$ embeds into the corresponding product, and we have the short exact sequence of graded $\Gamma^*$-modules
\begin{equation}\label{ses 209094} 
0 \rightarrow \coprod_{b\in B} \Sigma^{d(b)} M \rightarrow\prod_{b\in B} \Sigma^{d(b)} M \rightarrow Q \rightarrow 0,\end{equation}
where we write $Q$ as an abbreviation for the quotient module $\prod_{b\in B} \Sigma^{d(b)} M/\coprod_{b\in B} \Sigma^{d(b)} M$. The induced map 
\begin{align} \label{ses 209095}
 H^0_{dist}\left(\prod_{b\in B} \Sigma^{d(b)} M\right)
 &\rightarrow H^0_{dist}(Q)
\end{align}
is surjective.
\end{lemma}
\begin{proof}
We may describe an element of the product $\prod_{b\in B} \Sigma^{d(b)} M$ as a (perhaps infinite) tuple indexed by $B$, i.e., a function $f: B \rightarrow M$. Two such elements are equal in $Q$ if and only if they differ in only finitely many of their values. Consequently, given a left ideal $I$ of $\Gamma^*$, an element $f$ of $Q$ is $I$-torsion if and only if $f(b)$ is $I$-torsion for all but finitely many $b\in B$. More generally, an element $f$ of $\prod_{b\in B} \Sigma^{d(b)} M$, regarded as an element of $Q$, is in $h^0_{dist}(Q)$ if and only if there exists some distinguished ideal $I$ of $\Gamma^*$ such that $f(b)$ is $I$-torsion for all but finitely many $b\in B$. Let $B^{\prime}$ be a finite subset of $B$ such that $f(b)$ is $I$-torsion for all $b\notin B^{\prime}$. Then let $\tilde{f}: B\rightarrow M$ be given by the formula 
\begin{align*}
 \tilde{f}(b) &= \left\{ \begin{array}{ll} f(b) &\mbox{\ if\ } b\notin B^{\prime} \\ 0 &\mbox{\ if\ } b\in B^{\prime}.\end{array}\right. 
\end{align*}
The element $\tilde{f}$ of $\prod_{b\in B} \Sigma^{d(b)} M$ is also $I$-torsion, and $\tilde{f}$ represents the same element of $Q$ as $f$. Hence every element of $h^0_{dist}(Q)$ lifts to an element of $h^0_{dist}\left( \prod_{b\in B} \Sigma^{d(b)} M\right)$.

Hence the map \eqref{ses 209095} surjects onto a set of $\Gamma^*$-module generators of $H^0_{dist}(Q)$, namely the subset $h^0_{dist}(Q)$ of $H^0_{dist}(Q)$. Since \eqref{ses 209095} is also a $\Gamma^*$-module morphism, it must be surjective.
\end{proof}

The functor $\iota$ has a right adjoint $\tr: \gr\Comod(\Gamma)\rightarrow\gr\Mod(\Gamma^*)$, the ``rational trace'' functor. Despite the name, $\tr$ is not closely related to the injective trace functor $Itr$ defined in Definition \ref{m-r defs}. They are simply both special cases of the general notion of a trace functor, as in \cite[pg. 109]{MR1245487}. 
 
While $\iota$ is exact, the functor $\tr$ is in general only left exact. Its right-derived functors are described by \cite[Theorem 4.3]{MR5014030}:
\begin{theorem}\label{main thm 1 from derived prods paper}
Let $k$ be a commutative ring, and suppose that $\Gamma$ is a graded $k$-coalgebra which is projective as an $k$-module.
Then, for each nonnegative integer $n$ and each graded left $\Gamma^*$-module $M$, we have an isomorphism $H^n_{\dist}(M) \cong \iota\left( R^n\tr(M)\right)$, natural in the variable $M$.

Suppose furthermore that $k$ is a field, and that $\Gamma^*$ is finite-type\footnote{Recall, from \cref{conventions}, that ``finite-type'' means that $\Gamma^*$ is a finite-dimensional $k$-vector space in each individual degree.} and connected. Then the following claims are also each true:
\begin{itemize}
\item If $M$ is a bounded-above injective graded $\Gamma$-comodule, then the graded $\Gamma^*$-module $\iota(M)$ is injective. 
\item If $M$ is a bounded-above graded $\Gamma^*$-module, then the distinguished local cohomology groups $H^n_{\dist}(M)$ vanish for all $n>0$. 
\item Every bounded-above graded $\Gamma^*$-module is rational.
\end{itemize}
\end{theorem}
One might like to drop the adjective ``bounded-above'' from various places in the statement of Theorem \ref{main thm 1 from derived prods paper}, and indeed we shall have to do something like this to prove the main results of this paper. One cannot drop the adjective ``bounded-above'' from the claim about $\iota$ preserving injectives, because $\iota$ does {\em not} generally send non-bounded-above injectives to injectives. In particular, it is proven in the preprint \cite{selfinjectivitypreprint}
 that, in the case that $\Gamma$ is the dual Steenrod algebra, the coproduct of an infinite {\em non-bounded-above} family of suspensions of $\iota\Gamma$ fails to be injective, even though it is $\iota$ of an injective graded $\Gamma$-comodule. 

We are at least able to improve the result on vanishing of distinguished local cohomology in Theorem \ref{main thm 1 from derived prods paper}, weakening the ``bounded-above'' hypothesis, as follows.
\begin{prop}\label{better version of main thm 1}
Let $k$ be a field, and let $\Gamma$ be a finite-type, connected graded $k$-coalgebra 
such that $\Gamma^*$ is a $\mathcal{P}$-algebra. 
If $M$ is a rational graded $\Gamma^*$-module, then $H^n_{dist}(M)$ is trivial for all $n>0$.
\end{prop}
The author does not know whether Proposition \ref{better version of main thm 1} remains true without the hypothesis that $\Gamma^*$ is a $\mathcal{P}$-algebra.

Before beginning the proof of Proposition \ref{better version of main thm 1}, we point out one obvious way to try to prove the proposition, which we expect some readers will be thinking of, but which runs into great difficulties. Here is the relevant line of reasoning:
\begin{enumerate}
\item Since $M$ is assumed rational, the natural map $H^0_{dist}(M) \hookrightarrow M$ is an isomorphism.
\item Hence we just need to know that $H^n_{dist}(H^0_{dist}(M))$ vanishes for $n>0$. One should try to relate $H^*_{dist}(H^*_{dist}(M))$ to $H^*_{dist}(M)$ by a Grothendieck spectral sequence of the form $E_2^{s,t} \cong H^s_{dist}(H^t_{dist}(M))\Rightarrow H^{s+t}_{dist}(M)$, like what one has in classical local cohomology (dual to \cite[Lemma 2.7]{MR1172439}). 
\item Since $H^0_{dist}$ is idempotent, some kind of analysis of the spectral sequence ought to yield a contradiction if there were anything nonzero in the higher distinguished local cohomology of $H^0_{dist}(M)$, i.e., of $M$.
\end{enumerate}
The problem with this approach lies in the second step: one does not have such a Grothendieck spectral sequence unless one knows that $H^0_{dist}$ sends injectives (or at least enough injectives to build a resolution of $M$) to $H^0_{dist}$-acyclics. But at this point in the paper, we do not know that this property of $H^0_{dist}$ holds; it is rather close to what Proposition \ref{better version of main thm 1} itself claims. Evidently a different approach is needed, so we take a different approach now.
\begin{proof}[Proof of Theorem \ref{better version of main thm 1}]
We begin by proving the claim in the special case that $M$ is a coproduct of suspensions of $\iota(\Gamma)$. Write such a graded $\Gamma^*$-module $M$ as the direct sum $\coprod_{b\in B} \Sigma^{d(b)}\iota(\Gamma)$ for some set $B$ and some function $d: B\rightarrow \mathbb{Z}$. 
Since $\iota(\Gamma)$ is injective, by Theorem \ref{margolis thm 1} $\iota(\Gamma)$ must also have projective dimension $\leq 1$. Projective dimension is preserved under coproducts, so $\coprod_{b\in B} \Sigma^{d(b)} \iota(\Gamma)$ has projective dimension $\leq 1$. Another application of Margolis's theorem ensures that $\coprod_{b\in B} \Sigma^{d(b)} \iota(\Gamma)$ also has injective dimension $\leq 1$, hence has vanishing $H^n_{dist}$ for $n>1$, which is most of what we want to check.

We still need to check that $H^1_{dist}\left( \coprod_{b\in B} \Sigma^{d(b)} \iota(\Gamma)\right)$ vanishes. The coproduct $\coprod_{b\in B} \Sigma^{d(b)} \iota(\Gamma)$ embeds into the corresponding product, i.e., we have the short exact sequence of graded $\Gamma^*$-modules
\begin{equation}\label{ses 209094b} 0 \rightarrow \coprod_{b\in B} \Sigma^{d(b)} \iota(\Gamma) \rightarrow\prod_{b\in B} \Sigma^{d(b)} \iota(\Gamma) \rightarrow Q \rightarrow 0,\end{equation}
where $Q$ is the quotient module, as in Lemma \ref{surj H0 lemma}.
Since $\Gamma$ is a bounded-above injective graded $\Gamma$-comodule\footnote{To see that $\Gamma$ is bounded above, we remind the reader of our grading conventions for connective coalgebras, from \cref{conventions}: connective coalgebras are concentrated in {\em nonpositive} degrees.}, Theorem \ref{main thm 1 from derived prods paper} tells us that $\iota(\Gamma)$ is an injective graded $\Gamma^*$-module. Hence $\prod_{b\in B} \Sigma^{d(b)} \iota(\Gamma)$ is a product of injectives, hence it is itself injective, and has vanishing $H^n_{dist}$ for $n>0$.

The long exact sequence induced in $H^n_{dist}$ by \eqref{ses 209094b} consequently reduces to a four-term exact sequence
\begin{equation}\label{les 2300}\xymatrix{
 0 \ar[r] & 
  H^0_{dist}\left(\coprod_{b\in B} \Sigma^{d(b)} \iota(\Gamma)\right) \ar[r] &
  H^0_{dist}\left(\prod_{b\in B} \Sigma^{d(b)} \iota(\Gamma)\right) \ar[r] & H^0_{dist}(Q) \ar`r_l[ll] `l[dll] [dll] \\
 & H^1_{dist}\left(\coprod_{b\in B} \Sigma^{d(b)} \iota(\Gamma)\right) \ar[r] &
  0, &
}\end{equation}
which, together with Lemma \ref{surj H0 lemma}, gives us the vanishing of $H^1_{dist}\left(\coprod_{b\in B} \Sigma^{d(b)} \iota(\Gamma)\right)$.

Hence the statement of the proposition is true when $M$ is a coproduct of suspensions of $\iota(\Gamma)$. Every graded $k$-vector space is a coproduct of suspensions of $k$, hence every injective graded $\Gamma$-comodule is a retract of a coproduct of suspensions of $\Gamma$. Since $\iota$ is a left adjoint, it sends coproducts of suspensions of $\Gamma$ to coproducts of suspensions of $\iota(\Gamma)$. Hence, for every injective graded $\Gamma$-comodule $J$, $\iota(J)$ is a retract of a graded $\Gamma^*$-module on which $H^n_{dist}$ vanishes for all $n>0$. Hence $H^n_{dist}(\iota(J))$ vanishes as well, for $n>0$. (We remind the reader that, while $\iota$ sends bounded-above injectives to injectives, it does not always send arbitrary injectives to injectives, by the discussion following the statement of Theorem \ref{main thm 1 from derived prods paper}. Otherwise we would not have had to give the above argument for vanishing of $H^n_{dist}(\iota(J))$ for $n>0$.)

At this point, we could successfully use a Grothendieck spectral sequence argument of the kind sketched before we began the proof. For variety, we will finish the proof in a way which does not refer to a spectral sequence, but is roughly the same argument. Let $M$ be any graded $\Gamma$-comodule. Choose an injective resolution $I^{\bullet}$ for $M$. By exactness of $\iota$, the cochain complex $\iota(I^{\bullet})$ is a resolution of $\iota(M)$ by $H^0_{dist}$-acyclic graded $\Gamma^*$-modules. (Everything in this proof before this paragraph was an argument for precisely the $H^0_{dist}$-acyclicity of each of the graded modules $\iota(I^i)$.) 
Hence the cohomology of the cochain complex $H^0_{dist}(\iota(I^{\bullet}))$ is $H^*_{dist}(\iota(M))$. But $H^0_{dist} = \iota\circ \tr$, by Theorem \ref{main thm 1 from derived prods paper}, and the other composite $\tr\circ\iota$ is the identity functor. Hence we have isomorphisms
\begin{align*}
 H^n_{dist}(\iota(M)) 
  &\cong H^n\left(H^0_{dist}(\iota(I^{\bullet}))\right) \\
  &\cong H^n\left((\iota\circ \tr\circ \iota)(I^{\bullet})\right) \\
  &\cong H^n\left(\iota(I^{\bullet})\right) \\
  &\cong \iota(H^n\left(I^{\bullet}\right)) \\
  &\cong \left\{ \begin{array}{ll} \iota(M) &\mbox{\ if\ } n=0,\\ 0 &\mbox{\ if\ }n>0,\end{array}\right.
\end{align*}
as desired.
\end{proof}

One consequence of Proposition \ref{better version of main thm 1}, together with the natural equivalence $\iota\circ R^n\tr \simeq H^n_{dist}$ from Theorem \ref{main thm 1 from derived prods paper}, is a simple description of composites of distinguished local cohomology functors: we get the formula
\begin{align*}
 H^i_{dist}\circ H^j_{dist} &\simeq \left\{ \begin{array}{ll} 
  H^j_{dist} &\mbox{\ if\ } i=0,\\
  0 &\mbox{\ if\ } i>0.\end{array}\right.
\end{align*}
That seems worth pointing out, but we will not need to use it. A more pressing application of Proposition \ref{better version of main thm 1} is its role in proving the next lemma.

In Lemma \ref{easy limit in comodules}, we will use the notation $\lim^{\Gamma}$ for a limit {\em taken in the category of $\Gamma$-comodules.} This notation is standard, at least among a handful of topologists who write about such limits (e.g. in \cite{MR2337861} and in \cite{MR5014030}), and it is intended to avoid confusion with the limit taken in the underlying category of $k$-modules. Such confusion is a danger because the forgetful functor from $\Gamma$-comodules to $k$-modules generally does {\em not} preserve limits unless some extra hypothesis is assumed (e.g. that $\Gamma$ is finite-dimensional over $k$).
\begin{lemma}\label{easy limit in comodules}
Let $k$ be a field, and let $\Gamma$ be a finite-type, connected 
graded $k$-coalgebra such that $\Gamma^*$ is a $\mathcal{P}$-algebra. Given a graded $\Gamma$-comodule $M$ and an integer $n$, let $M_{\leq n}$ denote the graded $\Gamma$-comodule which is the quotient of $M$ in which all the homogeneous elements of $M$ of degree $>n$ are set to zero. Then the following claims are each true:
\begin{enumerate}
\item The natural map from $M$ to the limit of the sequence
\begin{equation}\label{seq 3003} \dots \rightarrow M_{\leq 3} \rightarrow M_{\leq 2}\rightarrow M_{\leq 1}\rightarrow\dots,\end{equation}
in the category of graded $\Gamma$-comodules, is an isomorphism.
\item If $\Gamma$ is finite-type and its dual algebra $\Gamma^*$ is a $\mathcal{P}$-algebra, then the $j$th derived limit $R^j\underset{n\rightarrow\infty}{\lim^{\Gamma}} M_{\leq n}$ vanishes for all $j>0$.
\end{enumerate}
\end{lemma}
The analogue of Lemma \ref{easy limit in comodules} in a category of graded {\em modules} is a triviality. The reason that Lemma \ref{easy limit in comodules} merits a bit of care is that limits {\em in categories of comodules} are trickier and much more difficult to calculate than limits in categories of modules; e.g. Sadofsky's preprint \cite{sadofsky2001homology} gave a means of calculating the homology of unbounded infinite products of spectra if one could calculate the derived functors of infinite products (which are nontrivial!) in the category of graded comodules over the dual Steenrod algebra, but there were essentially no tools for making such calculations in nontrivial cases until the paper \cite{MR5014030}.
\begin{proof}[Proof of Lemma \ref{easy limit in comodules}]
\begin{enumerate}
\item
The claim is simply that the graded comodule $M$ has the universal property of the limit of the sequence \eqref{seq 3003} in the category $\gr\Comod(\Gamma)$. It is routine to verify this claim: if we have some other cone over \eqref{seq 3003} (in the sense of Mac Lane's book \cite{MR1712872}), i.e., some other graded $\Gamma$-comodule $T$ equipped with compatible maps $f_n: T\rightarrow M_{\leq n}$, we see that we get a unique graded $\Gamma$-comodule map $f: T \rightarrow M$ compatible with the maps $\dots ,f_2,f_1,f_0,\dots$. In particular, $f$ is simply the map given as follows: for every homogeneous element $t$ of $T$, we let $f(t)$ be equal to $f_n(t)$ for any integer $n$ greater than the degree of $t$.  
%
\item 
We use one of the main theorems from \cite{MR5014030}: under the hypotheses that $\Gamma$ is finite-type and $\Gamma^*$ is connected, if we are given a sequence
\begin{equation}\label{seq fl24} \dots \rightarrow N_2 \rightarrow N_1 \rightarrow N_0\end{equation}
of morphisms of bounded-above\footnote{We emphasize that in neither the current situation nor in \cite[Theorem 7.1]{MR5014030} are the comodules required to be {\em uniformly} bounded-above.} graded $\Gamma$-comodules such that $R^1\lim$ vanishes on the resulting sequence of graded $\Gamma^*$-modules
\[ \dots \rightarrow \iota N_2 \rightarrow \iota N_1 \rightarrow \iota N_0,\]
then \cite[Theorem 7.1]{MR5014030} there is an isomorphism of $\Gamma^*$-modules 
\begin{align*}
\iota R^j\lim_n{}^{\Gamma}N_n &\cong H^j_{\dist}(\lim_n \iota N_n).\end{align*}

Consider the special case of \cite[Theorem 7.1]{MR5014030} in which the sequence \eqref{seq fl24} is the sequence \eqref{seq 3003}. The comodules in the sequence \eqref{seq 3003} are certainly all bounded above. The morphisms $\dots\rightarrow \iota M_{\leq 1}\rightarrow\iota M_{\leq 0}$ are all surjective, hence by the usual Mittag--Leffler criterion, the $R^1\lim$ vanishing hypothesis is satisfied.  Hence \cite[Theorem 7.1]{MR5014030} yields the isomorphism
\begin{align}\label{iso fl26}
 \iota R^j\lim_n{}^{\Gamma}M_{\leq n} &\cong H^j_{\dist}(\lim_n \iota M_{\leq n}).
\end{align}
Since $\iota M\cong \lim_n \iota M_{\leq n}$, the right-hand side of \eqref{iso fl26} is the distinguished local cohomology of a rational module, hence vanishes for $j>0$ by Proposition \ref{better version of main thm 1} (here is the place we use the assumption that $\Gamma^*$ is a $\mathcal{P}$-algebra). Hence $R^j\lim^{\Gamma}_nM_{\leq n} \cong 0$ for $j>0$, as desired.
\end{enumerate}
\end{proof}
Lemma \ref{easy limit in comodules} will be used below in our proof of Theorem \ref{ext iso 1}.

\subsection{Comodule Ext is module Ext}
\label{Comodule Ext is module Ext}

By theorems of Positselski \cite[Theorems 5.1 and 6.1]{MR5038878}, in the ungraded setting\footnote{These theorems of Positselski do not involve a grading on the coalgebra, but they do include the assumption that the coalgebra $\Gamma$ is conilpotent, i.e., it is a union of finite-dimensional subcoalgebras whose dual algebras are each nilpotent. This assumption is automatically satisfied by the coalgebras considered in this section and in much of this paper, since connected graded coalgebras are always conilpotent.} the functor $\iota$ induces an isomorphism
\begin{align}
\label{non-iso 0} \Ext^n_{\Comod(\Gamma)}(M,N) &\stackrel{\cong}{\longrightarrow} \Ext^n_{\Mod(\Gamma^*)}\left(\iota(M),\iota(N)\right)
\end{align}
for all $M,N$ if and only if the $\Ext$-group $\Ext^i_{\Mod(\Gamma^*)}(k,k)$ is a finite-dimensional $k$-vector space for each $i$.
In our motivating cases, the coalgebras dual to $\mathcal{P}$-algebras and more specifically Steenrod algebras, Positselski's finiteness criterion is not satisfied, 
so the map \eqref{non-iso 0} fails to be an isomorphism in our cases of interest.

Nevertheless, in this section we shall prove that for a certain class of coalgebras (which includes those whose duals are finite-type $\mathcal{P}$-algebras, and specifically the Steenrod algebras), the map of $\Ext$-groups in the {\em graded} categories,
\begin{align}
\label{non-iso 1} \Ext^n_{\gr\Comod(\Gamma)}(M,N) &\stackrel{}{\longrightarrow} \Ext^n_{\gr\Mod(\Gamma^*)}\left(\iota(M),\iota(N)\right),
\end{align}
{\em is} an isomorphism. The theorem to that effect is Theorem \ref{ext iso 1}. We first need a definition and several lemmas.

\begin{definition}\label{def of S}
Let $\Gamma$ be a connected graded coalgebra, and let $A$ be a connected graded algebra.
\begin{itemize}
\item
We will write $S_{A}$ for the functor $\gr\Mod(A) \rightarrow \gr\Mod(A)^{\mathbb{N}^{\op}}$ which sends a graded $A$-module $X$ to the sequence $\dots \rightarrow X_{\leq 2} \rightarrow X_{\leq 1}\rightarrow X_{\leq 0}$.
\item 
Similarly, write $S_{\Gamma}$ for the functor $\gr\Comod(\Gamma) \rightarrow \gr\Comod(\Gamma)^{\mathbb{N}^{\op}}$ which sends a graded $\Gamma$-comodule $X$ to the sequence $\dots \rightarrow X_{\leq 2} \rightarrow X_{\leq 1}\rightarrow X_{\leq 0}$, where just as in the module case (and as in Lemma \ref{easy limit in comodules}), we write $X_{\leq j}$ for the graded $\Gamma$-comodule quotient of $X$ in which all elements of degree $>j$ are set to zero.
\end{itemize}
\end{definition}
A circumspect reader can verify that $X_{\leq j}$ is indeed a well-defined graded comodule quotient of $X$, since $\Gamma$ is concentrated in nonpositive degrees. 

It is easy to see that $S_{\Gamma^*}\circ \iota \simeq \iota\circ S_{\Gamma}$, since the natural map between these two composites is an isomorphism on the underlying graded $k$-vector spaces.

Here is a very general lemma about injective objects in categories of inverse sequences, proven by standard techniques. Surely this lemma must have been known decades ago, but we know no reference for it.
\begin{lemma}\label{injectivity lemma}
Let $\mathcal{C}$ be an abelian category, and let
\[ \dots \rightarrow I_2\rightarrow I_1\rightarrow I_0\]
be a sequence in $\mathcal{C}$. Suppose that each $I_i$ is an injective object of $\mathcal{C}$, and suppose that each of the maps $I_{n+1}\rightarrow I_n$ in the sequence is a split epimorphism with injective kernel. Then $I_{\bullet}$ is an injective object in the category $\mathcal{C}^{\mathbb{N}^{\op}}$ of inverse sequences in $\mathcal{C}$.
\end{lemma}
\begin{proof}
Suppose we have a monomorphism $f_{\bullet}: X_{\bullet}\rightarrow Y_{\bullet}$ in $\mathcal{C}^{\mathbb{N}^{\op}}$, and a morphism $u_{\bullet}: X_{\bullet}\rightarrow I_{\bullet}$. To verify that $I_{\bullet}$ satisfies the universal property of an injective object, we need to verify that there exists a morphism of sequences $\ell_{\bullet}: Y_{\bullet}\rightarrow I_{\bullet}$ making the usual diagram commute:
\[\xymatrix{
X_{\bullet}\ar[r]^{f_{\bullet}} \ar[d]_{u_{\bullet}} & Y_{\bullet}\ar@{-->}[ld]^{\ell_{\bullet}} & \\
I_{\bullet}. &
}\]

We construct $\ell_{\bullet}$ by induction. For the initial step, one gets $\ell_{0}: Y_0\rightarrow I_0$ by a simple application of the universal property of the injective object $I_0$ to the diagram
\[\xymatrix{
X_{0}\ar[r]^{f_{0}} \ar[d]_{u_{0}} & Y_{0}\ar@{-->}[ld]^{\ell_{0}} \\
I_0. &
}\]
For the inductive step, suppose we have already constructed $\ell_j$ for $j=0,\dots, n-1$. To construct $\ell_n$, consider the diagram
\begin{equation}\label{comm diag 059459}\xymatrix{
& X_n\ar[r]^{f_n} \ar@/_/[rr]_{\phi} \ar[d]_{u_n} & Y_n\ar@{-->}[ld]^{\ell_n} \ar@/^/[rr]^{\psi} & X_{n-1}\ar[d]_{u_{n-1}}\ar[r]_{f_{n-1}} & Y_{n-1} \ar[ld]^{\ell_{n-1}} \\
\ker g \ar[r]_i & I_n \ar[rr]^g && I_{n-1} \ar@/^/[ll]^{\sigma} & ,
}\end{equation}
whose solid arrows comprise a commutative diagram except for the arrow $\sigma$, which satisfies only the commutativity condition $g\circ \sigma = \id_{I_{n-1}}$. We need to build a map $\ell_n$ filling in the dotted arrow to make the diagram commute.

We build the map $\ell_n$ in three steps. First, since $g\circ (u_n - \sigma\circ g\circ u_n)=0$, the map $u_n - \sigma\circ g\circ u_n: X_n\rightarrow I_n$ lifts to a map $k: X_n\rightarrow \ker g$ such that $i\circ k= u_n - \sigma\circ g\circ u_n$. 

Second, consider the diagram
\[\xymatrix{
 X_n\ar[r]^{f_n} \ar[d]_{k} & Y_n \ar@{-->}[ld]^{\tilde{\ell}} \\ \ker g. & 
}\]
Since $\ker g$ is injective, we have a map $\tilde{\ell}$ filling in the diagram and making it commute.

Third, let $\ell_n$ be the map $\sigma\circ \ell_{n-1}\circ \psi + i\circ \tilde{\ell}$. It is routine to check that this choice of $\ell_n$ satisfies the equations  $\ell_n \circ f_n = u_n$ and $g\circ \ell_n = \ell_{n-1}\circ \psi$, so that \eqref{comm diag 059459} indeed commutes, completing the inductive step.
\end{proof}

While $\iota$ does not preserve injectives, it at least sends {\em bounded-above} injectives to injectives, by Theorem \ref{main thm 1 from derived prods paper}. Here is a useful variant on that observation:
\begin{lemma}\label{enough injs lemma}
Let $k$ be a field, and let $\Gamma$ be a connected graded $k$-coalgebra. Let $N_{\bullet}$ denote some sequence
\begin{equation*}
\dots 
\rightarrow N_2 
\rightarrow N_1 
\rightarrow N_0
\end{equation*}
of bounded-above graded $\Gamma$-comodules\footnote{To be clear, we do {\em not} assume that the sequence $N_{\bullet}$ is {\em uniformly} bounded above.}. 

Then the sequence $N_{\bullet}$ admits an injective resolution $I^{\bullet}(N_{\bullet})$ in the category $\gr\Comod(\Gamma)^{\mathbb{N}^{\op}}$ of inverse sequences of graded $\Gamma$-comodules, such that every comodule in every sequence $I^j(N_{\bullet})$ in the resolution is bounded above.
\end{lemma}
\begin{proof}
As in \cref{u 1}, we write $U: \gr\Comod(\Gamma)\rightarrow\gr\Mod(k)$ for the forgetful functor. As in \cref{E 1}, we write $E: \gr\Mod(k)\rightarrow\gr\Comod(\Gamma)$ for the right adjoint of $U$, the extended comodule functor $E(V) = \Gamma\otimes_k V$. 

We will use $U$ and $E$ to construct $I^{\bullet}(N_{\bullet})$ by induction. Each object $I^j(N_{\bullet})$ of $I^{\bullet}(N_{\bullet})$ is not merely a single comodule, but a {\em sequence} of comodules, since $I^{\bullet}(N_{\bullet})$ is an injective resolution for $N_{\bullet}$ in the category of {\em sequences} of comodules. We need some notation for this: we will write $I^j(N_{\bullet})$ for the $j$th object (i.e., sequence) in the injective resolution, and we will write $I^j(N_{\bullet})(m)$ for the graded $\Gamma$-comodule which is the $m$th term in that sequence.

Now begin with the comodule $N_0$. Let $I^0(N_{\bullet})(0)$ be the extended $\Gamma$-comodule $E(U(N_0)) = \Gamma\otimes_k N_0$. The unit map of the adjunction yields a natural injection $N_0 \hookrightarrow E(U(N_0))$. Recall our conventions from \cref{conventions}: connected coalgebras are concentrated in {\em nonpositive} degrees, so $E(U(N_0))$ also vanishes in positive degrees. Hence $E(U(N_0))$, like $N_0$, is bounded above. 
Thus far, we have a diagram
\begin{equation}\label{diag 039490}\xymatrix{
 \dots \ar[r]^{\nu_3} & N_3\ar[r]^{\nu_2} & N_2\ar[r]^{\nu_1} & N_1\ar[r]^{\nu_0} & N_0\ar@{^{(}->}[d]^{f_0} \\
  & & & & I^0(N_{\bullet})(0),
}\end{equation}
Our next task is to construct a bounded-above injective comodule $I^0(N_{\bullet})(1)$ into which $N_1$ embeds and which will fit into the diagram \eqref{diag 039490}. 

We should set this up as the inductive step in an induction. Suppose we have already constructed a commutative diagram of graded $\Gamma$-comodules
\begin{equation}\label{diag 039491}\xymatrix{
 \dots \ar[r]^{\nu_{n+1}} & N_{n+1}\ar[r]^{\nu_n} & N_n\ar[r]^{\nu_{n-1}}\ar@{^{(}->}[d]^{f_n} & \dots \ar[r]^{\nu_0} & N_0 \ar@{^{(}->}[d]^{f_0} \\
  &  & I^0(N_{\bullet})(n) \ar@{->>}[r]^(.6){g_{n-1}}& \dots\ar@{->>}[r]^(.4){g_0} & I^0(N_{\bullet})(0)
}\end{equation}
in which the vertical maps are all injective, in which all the graded comodules on the bottom row are injective and bounded-above, and in which the horizontal maps in the bottom row are all split epimorphisms\footnote{As a reminder, in comodule categories, monomorphisms and epimorphisms behave in the expected and familiar way: the monomorphisms are just the morphisms which are one-to-one on the underlying graded vector spaces, and the epimorphisms are just the morphisms which are surjective on the underlying graded vector spaces.}. We want to extend diagram \eqref{diag 039491} by filling in its lower-left corner with a graded comodule $I^0(N_{\bullet})(n+1)$, a graded comodule monomorphism $f_{n+1}:N_{n+1} \hookrightarrow I^0(N_{\bullet})(n+1)$, and a split epimorphism of graded comodules $g_{n}: I^0(N_{\bullet})(n+1) \twoheadrightarrow I^0(N_{\bullet})(n)$ making the diagram commute.

To accomplish this, let $I^0(N_{\bullet})(n+1)$ be the direct sum $E(U(N_{n+1}))\oplus I^0(N_{\bullet})(n)$, and define the relevant maps as follows:
\begin{align*}
 f_{n+1}: N_{n+1}&\rightarrow E(U(N_{n+1}))\oplus I^0(N_{\bullet})(n) \\
 f_{n+1}(x) &= \left(\eta_{n+1}(x),(f_n\circ \nu_n)(x)\right)\\
 g_{n}: E(U(N_{n+1}))\oplus I^0(N_{\bullet})(n) &\rightarrow I^0(N_{\bullet})(n) \\
 g_{n}(y,z) &= z,
\end{align*}
where $\eta_{n+1}: N_{n+1}\rightarrow E(U(N_{n+1}))$ is the counit map of the adjunction $U\dashv E$ evaluated on $N_{n+1}$, i.e., $\eta_{n+1}(x) = 1\otimes x \in \Gamma\otimes_k N_{n+1}$. It is easy to see that $f_{n+1}$ is monic since its left-hand component is monic, to see that $g_{n}$ is split epic, and to see that these choices indeed make the diagram
\begin{equation*}
\xymatrix{
 \dots \ar[r]^{\nu_{n+1}} & N_{n+1}\ar[r]^{\nu_n}\ar@{^{(}->}[d]^{f_{n+1}} & N_n\ar[r]^{\nu_{n-1}}\ar@{^{(}->}[d]^{f_n} & \dots \ar[r]^{\nu_0} & N_0 \ar@{^{(}->}[d]^{f_0} \\
  & I^0(N_{\bullet})(n+1)\ar@{->>}[r]^{g_{n}} & I^0(N_{\bullet})(n) \ar@{->>}[r]^{g_{n-1}}& \dots\ar@{->>}[r]^{g_0} & I^0(N_{\bullet})(0)
}\end{equation*}
commute. We easily see also that $I^0(N_{\bullet})(n+1)$ is bounded above, since $E(U(N_{n+1}))$ and $I^0(N_{\bullet})(n)$ are each bounded above.

That completes the inductive step. At this point we have shown that every inverse sequence $N_{\bullet}$ of bounded-above graded $\Gamma$-comodules admits a monomorphism into some injective object $I^0(N_{\bullet})$ in the category of inverse sequences of graded $\Gamma$-comodules, such that $I^0(N_{\bullet})$ is levelwise bounded-above. From here we apply the usual process to construct an injective resolution in any abelian category with enough injectives: take the cokernel of the monomorphism $N_{\bullet}\stackrel{d^{-1}}{\longrightarrow} I^0(N_{\bullet})$ to obtain another inverse sequence of bounded-above graded $\Gamma$-comodules, and iterate, letting $I^1(N_{\bullet})$ be $I^0(\coker d^{-1})$, to obtain the injective resolution
\[ 0\rightarrow I^0(N_{\bullet}) \stackrel{d^0}{\longrightarrow} I^1(N_{\bullet}) \stackrel{d^1}{\longrightarrow} \dots \]
of $N_{\bullet}$ in $\gr\Comod(\Gamma)^{\mathbb{N}^{\op}}$. 

One more note is necessary, to explain why each of the sequences $I^{j}(N_{\bullet})$ is not merely levelwise injective, but is in fact an injective object in the category $\gr\Comod(\Gamma)^{\mathbb{N}^{\op}}$. This is a consequence of Lemma \ref{injectivity lemma}, since the transition maps in each sequence $I^j(N_{\bullet})$ are indeed split surjections whose kernel comodules are each injective.
\end{proof}
\begin{remark}
To be clear, what Lemma \ref{enough injs lemma} shows is a bit stronger than just showing that the category of inverse sequences of bounded-above graded $\Gamma$-comodules has enough injectives, since 
\[ \parbox{270pt}{``injective object in the category of inverse sequences of bounded-above graded $\Gamma$-comodules''}\]
is weaker than 
\[ \parbox{270pt}{``levelwise bounded-above injective object in the category of inverse sequences of graded $\Gamma$-comodules.''}\]
\end{remark}

We may apply the covariant embedding functor $\iota: \gr\Comod(\Gamma)\rightarrow \gr\Mod(\Gamma^*)$ levelwise to any inverse sequence. We will write $\iota^{\mathbb{N}^{\op}}$ for the resulting full, faithful, exact functor $\gr\Comod(\Gamma)^{\mathbb{N}^{\op}}\rightarrow \gr\Mod(\Gamma^*)^{\mathbb{N}^{\op}}$.
\begin{lemma}\label{enough injs lemma 2}
Maintain the same hypotheses as in Lemma \ref{enough injs lemma}. If we apply the functor $\iota^{\mathbb{N}^{\op}}$ levelwise to the cochain complex $I^{\bullet}(N_{\bullet})$ constructed in Lemma \ref{enough injs lemma}, the resulting cochain complex of inverse sequences $\iota^{\mathbb{N}^{\op}}\left(I^{\bullet}(N_{\bullet})\right)$ is an injective resolution of $\iota^{\mathbb{N}^{\op}}(N_{\bullet})$ in $\gr\Comod(\Gamma)^{\mathbb{N}^{\op}}$.
\end{lemma}
\begin{proof}
Each of the sequences $I^j(N_{\bullet})$ is levelwise bounded-above, and by Theorem \ref{main thm 1 from derived prods paper}, $\iota$ sends each of those bounded-above injective comodules to an injective graded $\Gamma^*$-module. Hence the cochain complex $\iota^{\mathbb{N}^{\op}}(I^j(N_{\bullet}))$, i.e.,
\begin{equation}\label{seq 03774} \dots
 \stackrel{\iota(g_2)}{\longrightarrow}
 \iota\left(I^j(N_{\bullet})(2)\right)
 \stackrel{\iota(g_1)}{\longrightarrow}
 \iota\left(I^j(N_{\bullet})(1)\right)
 \stackrel{\iota(g_0)}{\longrightarrow}
 \iota\left(I^j(N_{\bullet})(0)\right),
\end{equation}
is levelwise injective. Each of the transition maps $\iota(g_n)$ in \eqref{seq 03774} is $\iota$ applied to a split epimorphism, hence is itself a split epimorphism. 
The kernel of each transition map is $\iota$ applied to a bounded-above injective comodule, hence is also injective by Theorem \ref{main thm 1 from derived prods paper}.
Hence \eqref{seq 03774} is indeed an injective object in $\gr\Mod(\Gamma^*)^{\mathbb{N}^{\op}}$, as desired, by Lemma \ref{injectivity lemma}.
\end{proof}

\begin{observation}\label{enough injs observation}
For any graded ring $A$, we have the category $\gr\Mod(A)^{\mathbb{N}^{\op}}$ of inverse sequences of graded $A$-modules. For any such sequence 
\[ M_{\bullet} = \left( \dots \rightarrow M_2 \rightarrow M_1 \rightarrow M_0\right), \] 
there exists a monomorphism from $M_{\bullet}$ to an injective object $I_{\bullet}$ in the category $\gr\Mod(A)^{\mathbb{N}^{\op}}$, such that each of the transition maps 
\[  \dots \rightarrow I_2 \rightarrow I_1 \rightarrow I_0\]
are split epimorphisms in $\gr\Mod(A)$ with injective kernel. The construction of $I_{\bullet}$ is by the same method as in the construction of $I^0(N_{\bullet})$ in the proof of Lemma \ref{enough injs lemma}, but slightly easier, since here we make no claim about any (co)modules being bounded above. All that was necessary in the construction of $I^0(N_{\bullet})$ was that one is able to embed each object (here, a graded $A$-module; in Lemma \ref{enough injs lemma}, a graded comodule) into an injective object, {\em functorially}. Such a functorial embedding in Lemma \ref{enough injs lemma} was given by the unit map $\id\rightarrow E\circ U$ of the adjunction $U\dashv E$. In $\gr\Mod(A)$, the functorial embedding is given by sending each module $M$ to the product $G^{\hom_{\gr\Mod(A)}(M,G)}$ of copies of $G$, one copy for each morphism in $\hom_{\gr\Mod(A)}(M,G)$, where $G$ is an injective cogenerator for $\gr\Mod(A)$; a typical choice of injective cogenerator in a module category is the character module $\hom_{\mathbb{Z}}(A,\mathbb{Q}/\mathbb{Z})$. This classical construction dates back to Baer in 1940 \cite{MR2886}; see \cite[sections 1 and 2]{MR1967087} for a nice modern discussion.
\end{observation}

\begin{remark} \label{motivational remark}
Theorem \ref{ext iso 1} will make use of Lemma \ref{continuous ext is ext}, which we are about to state and prove. Lemma \ref{continuous ext is ext} shows that, over a connected graded algebra over a field, the $\Ext$-groups agree with the {\em continuous} $\Ext$-groups. In this remark we explain some historical and mathematical background behind such ``continuous $\Ext$-groups.'' Readers less interested in reading such background material can skip ahead to Definition \ref{def of cont ext}.

One of the fundamental insights of Jack Morava in the 1970s (see \cite{MR782555}) which began the subject of ``chromatic homotopy theory'' is that the {\em continuous} cohomology of the Morava stabilizer group $\Aut(\mathbb{G}_{1/n})$, with appropriate coefficients, agrees with $\Ext$-groups in the category of comodules over the Hopf algebroids $(E(n)_*,E(n)_*E(n))$ or $(K(n)_*,K(n)_*\otimes_{BP_*BP}K(n)_*)$, depending on the coefficients in question. Here $E(n)_*$ and $K(n)_*$ and $BP_*$ are Johnson--Wilson $E$-theory, Morava $K$-theory, and Brown--Peterson homology, respectively. It would take us too far afield here to provide a review of what these generalized homology theories are, but we recommend \cite[section 6.1]{MR860042} for an account of several variants of such an isomorphism. 

It is more immediately relevant to recall some properties of the {\em continuous cohomology} of a profinite group, after Tate \cite{MR429837}. This is the same notion of continuous cohomology invoked in the previous paragraph. Continuous cohomology of a profinite group is {\em not} the same as the usual profinite group cohomology which one most commonly finds in standard textbooks like \cite{MR0554237}. Given a profinite group $G$ and a discrete $G$-module $M$, the ``usual'' cohomology $H^*(G; M)$ is defined as the colimit
\[ \colim_{N\trianglelefteq G} H^*(G/N; M^N),\]
where the colimit is taken over all the finite-index normal subgroups $N$ of $G$.

By contrast, the {\em continuous} cohomology is defined even when the coefficient module $M$ is not discrete, but rather has a nontrivial topology, such that the action of $G$ on $M$ remains continuous. Continuous cohomology is generally defined as the cohomology of the cochain complex that one obtains by taking the usual bar resolution of $\mathbb{Z}_{triv}$, and then taking only the {\em continuous} $M$-valued cochains. Here is an alternative description, in line with Jannsen's approach in \cite{MR0929536}, which applies when the continuous $G$-module $M$ is pro-discrete. Write $M$ as the limit $\lim_n M_n$ of some Mittag-Leffler sequence 
\begin{equation}\label{seq 0294}\dots \rightarrow M_2\rightarrow M_1 \rightarrow M_0\end{equation}
of morphisms of discrete $G$-modules. On the category $\Mod(G)^{\mathbb{N}^{\op}}$ of inverse sequences of discrete $G$-modules, we have the functor $H^0: \Mod(G)^{\mathbb{N}^{\op}} \rightarrow \Ab^{\mathbb{N}^{\op}}$ given by the profinite group cohomology $H^0(G; -) \simeq \colim_{N\trianglelefteq G} H^0(G/N; (-)^N)$. We also have the limit functor $\lim: \Ab^{\mathbb{N}^{\op}} \rightarrow \Ab$. The $n$th continuous cohomology of the pro-discrete $G$-module $M$ is the $n$th right-derived functor of the composite $\lim\circ H^0$, applied to the sequence \eqref{seq 0294}.

Back to the setting of this paper, where we work with some connected graded coalgebra $\Gamma$ rather than a group algebra $k[G]$, or its completion $k[[G]]$, or a linear dual thereof. As in the case of group cohomology, it is true that $\Ext$ in the category of graded $\Gamma$-comodules agrees with a suitable notion of continuous $\Ext$ in the category of graded $\Gamma^*$-modules (this will be proven along the way, during the proof of Theorem \ref{ext iso 1}). The right notion of {\em continuous} $\Ext^n_{\Gamma^*}(M,N)$ in this situation, analogous to the functor $R^n(\lim\circ H^0)$ from continuous group cohomology, is defined as follows.
\end{remark}
\begin{definition}\label{def of cont ext}
Let $k$ be a field, and let $A$ be a connected graded $k$-algebra. We define the $i$th {\em continuous $\Ext$-group} $\Ext_{A,c}^i(M,N)$ as the $i$th right derived functor of the functor 
\begin{align}
\label{continuous ext def} \lim\circ \hom_{\gr\Mod(A)}(M,-): \gr\Mod(A)^{\mathbb{N}^{\op}} &\rightarrow \gr\Ab 
\end{align}
applied to the sequence $S_A(N) = \left( \dots N_{\leq 2}\rightarrow N_{\leq 1}\rightarrow N_{\leq 0}\right)$ defined in Definition \ref{def of S}.
\end{definition}
The sequence $S_{A}(N)$ plays a role for continuous $\Ext$ which is like the role played by the expression of a pro-discrete $G$-module as a limit of a sequence of discrete $G$-modules, in the theory of continuous group cohomology, described in Remark \ref{motivational remark}.

However, unlike what happens in continuous group cohomology, the continuous $\Ext$-group defined in \eqref{continuous ext def} is isomorphic to an ordinary $\Ext$-group. We prove this isomorphism now, in Lemma \ref{continuous ext is ext}.

\begin{lemma}\label{continuous ext is ext}
Let $A$ be a connected graded $k$-algebra, and let $M,N$ be graded $A$-modules. Then, for every integer $n$, we have an isomorphism
\begin{align}
\label{ext iso 2} \Ext^n_A(M,N) &\cong \Ext^n_{A,c}(M,N)
\end{align}
natural in the variables $M$ and $N$.
\end{lemma}
\begin{proof}
Since $\lim$ sends injectives to injectives\footnote{This is elementary, well-known, and quite general, but we include a proof because the author is tired of his own habit of forgetting the cute proof of this fact and then wasting time piecing together a more laborious one. Here is the cute proof: recall that, for any category $\mathcal{C}$, $\Delta: \mathcal{C}\rightarrow \mathcal{C}^{\mathbb{N}^{\op}}$ is a standard notation for the functor that sends an object $M$ to the constant sequence $\dots\stackrel{\id}{\longrightarrow} M\stackrel{\id}{\longrightarrow} M$. Suppose $\mathcal{C}$ is abelian, complete, and co-complete. Then the functor $\Delta$ is exact, since it has both left and right adjoints (namely, the $\colim$ and $\lim$ functors). Hence $\lim$ has an exact left adjoint, hence $\lim$ sends injectives to injectives.}, we get a Grothendieck spectral sequence
\begin{align*}
 E_2^{s,t} &\cong \left(R^s\hom_{\gr\Mod(A)}(M,-)\right)(R^t\lim)\left(S(N)\right) \\
 &\Rightarrow R^{s+t}\left(\lim\hom_{\gr\Mod(A)}(M,-)\right)\left(S(N)\right) \\
 &\cong \Ext^{s+t}_{A,c}(M,N).
\end{align*}
The derived limit $R^1\lim$ of the sequence $S(N) = \left(\dots \rightarrow N_{\leq 1} \rightarrow N_{\leq 0}\right)$ vanishes due to the Mittag--Leffler condition, so the spectral sequence collapses immediately to the $(t=0)$-line, yielding the isomorphism \eqref{ext iso 2}.
\end{proof}

\begin{theorem}\label{ext iso 1}
Let $k$ be a field, and let $\Gamma$ be a finite-type connected graded $k$-coalgebra whose dual $\Gamma^*$ is a $\mathcal{P}$-algebra. Let $M,N$ be graded $\Gamma$-modules. Then, for each integer $n$, we have an isomorphism 
\begin{align}
\label{ext iso 1 1} \Ext^n_{\gr\Comod(\Gamma)}(M,N) &\stackrel{\cong}{\longrightarrow} \Ext^n_{\gr\Mod(\Gamma^*)}(\iota(M),\iota(N)),
\end{align}
natural in the variables $M$ and $N$.
\end{theorem}
\begin{proof}
We would like to use a Grothendieck spectral sequence 
\begin{align}
\label{gss 1001} R^*\left(\lim \hom_{\gr\Mod(\Gamma^*)}\left(\iota(M), -\right)\right)\left( R^*\iota\right)(-) &\Rightarrow R^*\left(\lim \hom_{\gr\Mod(\Gamma^*)}\left(\iota(M), \iota(-)\right)\right)
\end{align}
for calculating the derived functors of the composite functor
\begin{align*}
 \lim \hom_{\gr\Mod(\Gamma^*)}\left(\iota(M), \iota(-)\right) : \gr\Comod(\Gamma)^{\mathbb{N}^{\op}} &\rightarrow \gr\Ab.
\end{align*}
However, since $\iota$ does not send injectives to injectives (see discussion preceding Proposition \ref{better version of main thm 1}), one cannot rely on any such spectral sequence existing in general. However, if the object of $\gr\Comod(\Gamma)^{\mathbb{N}^{\op}}$ which we input into the functor $\lim \hom_{\gr\Mod(\Gamma^*)}\left(\iota(M), \iota(-)\right)$ admits a resolution by injectives which {\em are} sent by $\iota$ to injectives, then we indeed get a Grothendieck spectral sequence of the form of \eqref{gss 1001} {\em when evaluated on that object of } $\gr\Comod(\Gamma)^{\mathbb{N}^{\op}}$.
 That is the reason we have bothered to prove Lemma \ref{enough injs lemma 2}: it is precisely what guarantees that the sequence $S_{\Gamma}(N)$ from Definition \ref{def of S}, for any graded $\Gamma$-comodule $N$, admits an injective resolution in $\gr\Comod(\Gamma)^{\mathbb{N}^{\op}}$ {\em which is sent by $\iota$ to an injective resolution of $\iota(S_{\Gamma}(N))$.}

The resulting spectral sequence \eqref{gss 1001} collapses by the exactness of $\iota$, yielding isomorphism \eqref{hard iso 1} in the chain of isomorphisms
\begin{align}
\nonumber&\Ext^s_{\gr\Mod(\Gamma^*)}(\iota(M),\iota(N)) \\
\label{iso 001 0} & \cong \Ext^s_{\Gamma^*,c}(\iota(M),\iota(N)) \\
\nonumber &\cong R^s\left(\lim \hom_{\gr\Mod(\Gamma^*)}\left(\iota(M), -\right)\right)\left(S_{\Gamma^*}(\iota(N))\right)\\
\nonumber  &\cong R^s\left(\lim \hom_{\gr\Mod(\Gamma^*)}\left(\iota(M), -\right)\right)\left(\iota(S_{\Gamma}(N))\right) \\
\label{hard iso 1}  &\cong R^s\left(\lim \hom_{\gr\Mod(\Gamma^*)}\left(\iota(M), \iota(-)\right)\right)\left(S_{\Gamma}(N)\right) \\
\nonumber &\cong R^s\left(\lim \hom_{\gr\Comod(\Gamma)}\left(M, (\tr\circ \iota)(-)\right)\right)\left(S_{\Gamma}(N)\right)\\
\nonumber &\cong R^s\left(\lim \hom_{\gr\Comod(\Gamma)}\left(M, -\right)\right)\left(S_{\Gamma}(N)\right)\\
\label{iso 001 2} &\cong R^s\left(\hom_{\gr\Comod(\Gamma)}\left(M, \lim{}^{\Gamma}(-)\right)\right)\left(S_{\Gamma}(N)\right).
\end{align}
Isomorphism \eqref{iso 001 0} is due to Lemma \ref{continuous ext is ext}. 
Meanwhile, we may simplify \eqref{iso 001 2} by means of another Grothendieck spectral sequence: 
since $\lim^{\Gamma}$ sends injectives to injectives  (see footnote in proof of Lemma \ref{continuous ext is ext} for the argument, which is well-known), we have a Grothendieck spectral sequence 
\begin{align}
\label{gss 1000000} E_2^{s,t} &\cong \Ext_{\gr\Comod(\Gamma)}^s\left(M, R^t\lim{}^{\Gamma}(S_{\Gamma}(N))\right) \\
\nonumber &\Rightarrow R^{s+t}\left(\hom_{\gr\Comod(\Gamma)}\left(M, \lim{}^{\Gamma}(-)\right)\right)(S_{\Gamma}(N)).
\end{align}
The spectral sequence also collapses on to the $t=0$ line, by the second part of Lemma \ref{easy limit in comodules}. The $E_2$-page of the spectral sequence is $\Ext_{\gr\Comod(\Gamma)}^s\left(M, N\right)$, by the first part of Lemma \ref{easy limit in comodules}.
Hence the chain of isomorphisms \eqref{iso 001 0} through \eqref{iso 001 2} yields an isomorphism between the two sides of \eqref{ext iso 1 1}.
\end{proof}
To be clear, the only place in the proof of Theorem \ref{ext iso 1} where we have used the assumption that $\Gamma^*$ is a $\mathcal{P}$-algebra is when we invoked Lemma \ref{easy limit in comodules}, which includes in its hypotheses that $\Gamma^*$ is a $\mathcal{P}$-algebra.

The simplest topological corollary of Theorem \ref{ext iso 1} is the following description of the Adams $E_2$-page, which has the virtue of being phrased entirely in terms of modules, not comodules, but also not requiring any finiteness hypotheses whatsoever, unlike the traditional description \eqref{old adams e2 1} of the Adams $E_2$-page in terms of $\Ext$ over the Steenrod algebra.
\begin{corollary}\label{adams e2 cor}
Let $p$ be a prime, and let $X,Y$ be spectra. Then the $E_2$-page of the Adams spectral sequence converging to $\left[ \Sigma^{t-s}X,\hat{Y}_{H\mathbb{F}_p}\right]$ is 
\begin{align*}
 E_2^{s,t} &\cong \Ext_{\gr\Mod(\Gamma^*)}^s\left( \iota H_{-*}(X;\mathbb{F}_p), \iota H_{-*}(Y;\mathbb{F}_p)\right)^{-t},
\end{align*}
where $\Gamma^*$ is the mod $p$ Steenrod algebra. In keeping with our conventions in this paper (\cref{conventions}), we use {\em cohomological} gradings on the homology comodules $H_*(X; \mathbb{F}_p)$ and $H_*(Y;\mathbb{F}_p)$, which is why we see $-t$ rather than $t$, and $-*$ rather than $*$, in the $E_2$-page description.

Rephrasing the result using the more traditional notations from stable homotopy theory, including the {\em homological} grading convention on homology comodules: we equip $H_*(X;\mathbb{F}_p)$ and $H_*(Y;\mathbb{F}_p)$ each with the adjoint action of the Steenrod algebra $\Gamma^*$, and then the Adams $E_2$-page is given by
\begin{align*}
 E_2^{s,t} &\cong \Ext_{\Gamma^*}^{s,t}\left(  H_*(X;\mathbb{F}_p),  H_*(Y;\mathbb{F}_p)\right).
\end{align*}
\end{corollary}

\section{No nonzero projectives in graded comodules over the dual Steenrod algebra}

The isomorphism between comodule $\Ext$ and module $\Ext$, provided by Theorem \ref{ext iso 1}, gives us a useful perspective from which to look at projectives in the category of graded $\Gamma$-comodules. In this section we use that perspective to prove the old conjecture that there are no nonzero such projectives in the case that $\Gamma$ is the dual Steenrod algebra.

We first need some lemmas. 
\begin{lemma}\label{mitchell iso lemma}
Let $k$ be a field, and let $\Gamma$ be a connected graded $k$-coalgebra which is not bounded below.
Write $I_j$ for the homogeneous left ideal of $\Gamma^*$ generated by all homogeneous elements of degree $\geq j$. Suppose that $\Gamma^*$ has the property that, for every graded $\Gamma^*$-module $M$, we have an isomorphism of abelian groups
\begin{align}
\label{mitchell iso 1}
 H^n_{dist}(M) &\cong \underset{j\rightarrow\infty}{\colim} \Ext^n_{\gr\Mod(\Gamma^*)}(\Gamma^*/I_j,M),
\end{align}
natural in the variable $M$, for each integer $n$.

Then, for any graded $\Gamma$-comodule $M$, the $\Gamma^*$-module $\iota(M)$ has no nonzero free summands.
\end{lemma}
\begin{proof}
If $M$ is a graded $\Gamma^*$-module with a nonzero free summand, then in particular it has a graded $\Gamma^*$-module summand of the form $\Sigma^i \Gamma^*$ for some $i$. Since $\Gamma^*$ is not bounded above, the element $1\in \Sigma^i \Gamma^*$ is not left $I_j$-torsion for any $j$. Consequently $1$ cannot be in the image of the natural map $H^0_{dist}(M) \hookrightarrow M$. Hence the natural map $H^0_{dist}(M) \hookrightarrow M$ is not an isomorphism, hence $M$ is not a rational graded $\Gamma^*$-module by \cite[Theorem 3.7]{MR2012570}, i.e., $M$ cannot be isomorphic to $\iota(N)$ for any graded $\Gamma$-comodule $N$.
\end{proof}

In \cite[Theorem 5.9]{MR5014030}, it is proven that the isomorphism \eqref{mitchell iso 1} holds when $\Gamma^*$ is the mod $p$ Steenrod algebra, for any prime $p$. More generally, \cite[Theorem 5.9]{MR5014030} shows that condition \eqref{mitchell iso 1} holds when $\Gamma$ is any finite-type ``Mitchell coalgebra'' over a field. The definition of a Mitchell coalgebra is rather technical, and its only motivating examples are the dual Steenrod algebras, so we do not reproduce the definition of a Mitchell coalgebra here. Interested readers can consult \cite[Definition 5.3]{MR5014030} for that definition.

\begin{lemma}\label{projectivity lemma}
Let $k$ be a field, and let $\Gamma$ be a connected graded $k$-coalgebra. Let $P$ be a graded $\Gamma^*$-module such that $\Ext^s_{\gr\Mod(\Gamma^*)}(P,N)$ vanishes for all $s>0$ for all {\em bounded-above} graded $\Gamma^*$-modules $N$. Then $P$ is a projective graded $\Gamma^*$-module. 
\end{lemma}
We give two proofs of Lemma \ref{projectivity lemma}, with significant similarities to one another. The first proof will be more to the taste of a reader who wants to organize everything in life into a spectral sequence, while the second proof exercises some restraint by only invoking two spectral sequences instead of three.
\begin{proof}[First proof of Lemma \ref{projectivity lemma}]
Write $N^{>j}$ for the graded $\Gamma$-subcomodule of $N$ which is $k$-linearly spanned by all homogeneous elements of degree $>j$. We have a short exact sequence of graded $\Gamma$-comodules
\begin{equation}\label{ses 028} 0 \rightarrow N^{>0} \rightarrow N \rightarrow N_{\leq 0} \rightarrow 0.\end{equation}

By hypothesis, $\Ext^s_{\gr\Mod(\Gamma^*)}(P,N_{\leq 0})$ is zero for all $s>0$. We claim that $\Ext^s_{\gr\Mod(\Gamma^*)}(P,N^{>0})$ is {\em also} trivial for all $s>0$. We have the exhaustive, complete (in the graded category), Hausdorff filtration of $N^{>0}$
\begin{equation}\label{filt 029} \dots\subseteq N^{>2}\subseteq N^{>1} \subseteq N^{>0}.\end{equation}
If we agree to write $N^j$ for the degree $j$ summand of the graded $\Gamma^*$-comodule $\Gamma$, then the filtration quotient $N^{>j-1}/N^{>j}$ is isomorphic to $N^j$, which is bounded-above for each $j$. 
Consequently, in the spectral sequence
\begin{align}
\label{ss 203000} E_1^{s,t} \cong \Ext^s_{\gr\Mod(\Gamma^*)}(P,N^t) &\Rightarrow \Ext^s_{\gr\Mod(\Gamma^*)}(P,N^{>0}) \\
\nonumber d_r: E_r^{s,t} &\rightarrow E_r^{s+1,t+r}
\end{align}
obtained by applying $\Ext^*_{\gr\Mod(\Gamma^*)}(P,-)$ to \eqref{filt 029}, 
we have immediate collapse to the $(s=0)$-line, yielding the claimed triviality of $\Ext^s_{\gr\Mod(\Gamma^*)}(P,N^{>0})$ for all $s>0$---once we are certain that the spectral sequence really has the right convergence properties.

To be careful, we should check the convergence properties of the spectral sequence. We use Boardman's techniques for convergence analysis, from \cite{MR1718076}. Spectral sequence \eqref{ss 203000} arises from the unrolled exact couple \begin{equation*}
\xymatrix{
\vdots \ar[rd] & \\
 & 0 \ar[ld] \\
\Ext^*_{\gr\Mod(\Gamma^*)}(P,N^{>0})\ar[uu]^{\id}\ar[rd] \\
 & 0 \ar[ld] \\
\Ext^*_{\gr\Mod(\Gamma^*)}(P,N^{>0})\ar[uu]^{\id}\ar[rd] \\
 & \Ext^*_{\gr\Mod(\Gamma^*)}(P,N^1) \ar[ld] \\
\Ext^*_{\gr\Mod(\Gamma^*)}(P,N^{>1})\ar[uu]\ar[rd]\\
 & \Ext^*_{\gr\Mod(\Gamma^*)}(P,N^2) \ar[ld] \\
\Ext^*_{\gr\Mod(\Gamma^*)}(P,N^{>2})\ar[uu]\ar[rd]\\
 & \Ext^*_{\gr\Mod(\Gamma^*)}(P,N^3) \ar[ld] \\
\vdots\ar[uu]
}\end{equation*}
which we clearly intend to converge to the colimit, which is constant: $\Ext^*_{\gr\Mod(\Gamma^*)}(P,N^{>0})$. To satisfy Boardman's ``conditional convergence to the colimit'' \cite[Definition 5.10]{MR1718076}, we should check that the limit $\lim_j \Ext^*_{\gr\Mod(\Gamma^*)}(P,N^{>j})$ and derived limit $\lim_j^1 \Ext^*_{\gr\Mod(\Gamma^*)}(P,N^{>j})$ are each trivial.

To that end, we have a Grothendieck spectral sequence
\begin{align}
\label{gss 1} {}^IE_2^{s,t}& \cong \left(R^s\hom_{\gr\Mod(\Gamma^*)}(P,-)\right)\left( R^t \lim\right)\left(N^{>\bullet}\right) \\
 &\Rightarrow R^{s+t}\left( \hom_{\gr\Mod(\Gamma^*)}(P,\lim(-))\right)\left(N^{>\bullet}\right) 
\end{align}
where $N^{>\bullet}$ denotes the sequence $\dots \rightarrow N^{>2} \rightarrow N^{>1}$. This spectral sequence collapses to the $(t=0)$-line, since the limit is calculated in the graded category. For why $\lim$ sends injectives to acyclics, so that the spectral sequence exists, see the footnote in the proof of Lemma \ref{continuous ext is ext}. The limit $\lim_j N^{>j}$ is trivial, so we have isomorphisms
\begin{align}
\nonumber 0 = \Ext^s_{\gr\Mod(\Gamma^*)}(P,\lim_j N^{>j})
 &\cong R^{s}\left( \hom_{\gr\Mod(\Gamma^*)}(P,\lim(-))\right)\left(N^{>\bullet}\right) \\
\label{iso 10000000} &\cong R^{s}\left( \lim\hom_{\gr\Mod(\Gamma^*)}(P,-)\right)\left( N^{>\bullet}\right) 
\end{align}
for all $s$.
We have a second Grothendieck spectral sequence
\begin{align}
\label{gss 1000003}  {}^{II}E_2^{s,t} &\cong \left(R^s\lim\right)\left(R^t\hom_{\gr\Mod(\Gamma^*)}(P,-)\right)\left( N^{>\bullet}\right) \\
\nonumber &\Rightarrow R^{s+t}\left( \lim\hom_{\gr\Mod(\Gamma^*)}(P,-)\right)\left(N^{>\bullet}\right) 
\end{align}
As usual, to know that this Grothendieck spectral sequence exists, we are supposed to check that $\hom_{\gr\Mod(\Gamma^*)}(P,-)$ sends injectives to $\lim$-acyclics---or at least that every object in $\gr\Mod(\Gamma^*)^{\mathbb{N}^{\op}}$ admits an injective resolution which is sent by $\hom_{\gr\Mod(\Gamma^*)}(P,-)$ to a $\lim$-acyclic object of $\Ab^{\mathbb{N}^{\op}}$. To verify that the latter condition indeed holds, use Observation \eqref{enough injs observation} to embed any given object $X_{\bullet}$ of $\hom_{\gr\Mod(\Gamma^*)}(P,-)$ into an injective object $I_{\bullet}$ of the category of sequences $\gr\Mod(\Gamma^*)^{\mathbb{N}^{\op}}$, such that the transition maps in the sequence $I_{\bullet}$ are all split epimorphisms. Since split epimorphisms are preserved by any functor whatsoever, the sequence $\hom_{\gr\Mod(\Gamma^*)}(P,I_{\bullet})$ is Mittag--Leffler and hence $\lim$-acyclic. One can iterate this construction, as we did in the proof of Lemma \ref{enough injs lemma}, to build a resolution of $X_{\bullet}$,  by injectives which are each sent by the functor $\hom_{\gr\Mod(\Gamma^*)}(P,-)$ to $\lim$-acyclic sequences.

We conclude that we indeed have the Grothendieck spectral sequence \cref{gss 1000003}. That spectral sequence collapses immediately for degree reasons (its $E_2$-page is trivial unless $s=0$ or $s=1$), yielding a short exact sequence
\[
\hspace{-55pt} 0 \rightarrow
 \lim_j{}^1 \Ext^{n-1}_{\gr\Mod(\Gamma^*)}(P,N^{>j}) \rightarrow
 R^{n}\left( \lim\hom_{\gr\Mod(\Gamma^*)}(P,-)\right)\left(N^{>\bullet}\right) \rightarrow
 \lim_j \Ext^{n}_{\gr\Mod(\Gamma^*)}(P,N^{>j}) \rightarrow
 0
\]
for each $n$. In \eqref{iso 10000000} we showed that the middle term in this short exact sequence is trivial for all $n$, so the terms on the ends must also be trivial. Those terms on the ends are precisely the limit and derived limit terms whose triviality gives us conditional convergence of the spectral sequence \eqref{ss 203000}. 

While \eqref{ss 203000} appears to be a first-quadrant spectral sequence, its differentials do not have the property that, for any single bidegree, any sufficiently long differential into or out of that bidegree is zero. Hence \eqref{ss 203000} is {\em not} first quadrant in Boardman's sense. Instead, using Boardman's conventions, \eqref{ss 203000} is a half-plane spectral sequence with entering differentials, and we have just checked that it is conditionally convergent to the colimit. By \cite[Theorem 7.3]{MR1718076}, such a spectral sequence is strongly convergent if the derived limit $\lim_r^1 E_r^{*,*}$ of its pages is zero. This condition is satisfied by \eqref{ss 203000} since its $E_1$-page is concentrated on the $(s=0)$-line, so it collapses immediately and hence $\lim_r^1 E_r^{*,*}$ is the derived limit of a constant sequence.

Consequently, from the long exact sequence obtained by applying $\Ext^*_{\gr\Mod(\Gamma^*)}(P,-)$ to \eqref{ses 028}, we get the vanishing of $\Ext^s_{\gr\Mod(\Gamma^*)}(P,N)$ for all $s>0$ and for all $N$. Hence $P$ is projective in $\gr\Mod(\Gamma^*)$.
\end{proof}
\begin{proof}[Second proof of Lemma \ref{projectivity lemma}]
For each $j$, we have the short exact sequences of graded $\Gamma$-comodules
\begin{equation}\label{ses 028a} 0 \rightarrow N^{>j} \rightarrow N \rightarrow N_{\leq j} \rightarrow 0\mbox{\ \ and}\end{equation}
\begin{equation}\label{ses 028b} 0 \rightarrow N^{>j} \rightarrow N^{>j-1} \rightarrow N^j \rightarrow 0.\end{equation}
We have long exact sequences induced in $\Ext^s_{\gr\Mod(\Gamma^*)}(P,-)$ by \eqref{ses 028a} and \eqref{ses 028b}. From those long exact sequences and from the vanishing of $\Ext^s_{\gr\Mod(\Gamma^*)}(P,-)$ on bounded-above modules for $s>0$, we have isomorphisms
\begin{align*}
\dots\stackrel{\cong}{\longrightarrow} 
 \Ext^s_{\gr\Mod(\Gamma^*)}(P,N_{\leq 1}) &\stackrel{\cong}{\longrightarrow}\Ext^s_{\gr\Mod(\Gamma^*)}(P,N_{\leq 0}) \\
&\stackrel{\cong}{\longrightarrow}\Ext^s_{\gr\Mod(\Gamma^*)}(P,N_{\leq -1})\\
&\stackrel{\cong}{\longrightarrow}\dots \\
 &\stackrel{\cong}{\longrightarrow}\Ext^s_{\gr\Mod(\Gamma^*)}(P,N)
\end{align*}
for all $s>1$, and surjections
\begin{align*}
\dots \twoheadrightarrow\Ext^1_{\gr\Mod(\Gamma^*)}(P,N_{\leq 1}) 
&\twoheadrightarrow\Ext^s_{\gr\Mod(\Gamma^*)}(P,N_{\leq 0})\\
&\twoheadrightarrow\Ext^s_{\gr\Mod(\Gamma^*)}(P,N_{\leq -1})\\
&\twoheadrightarrow\dots \\
&\twoheadrightarrow\Ext^s_{\gr\Mod(\Gamma^*)}(P,N).
\end{align*}
Hence, if we can show that the limit $\lim_{j\rightarrow\infty} \Ext^s_{\gr\Mod(\Gamma^*)}(P,N_{\leq j})$ vanishes for all $s\geq 1$, then $\Ext^s_{\gr\Mod(\Gamma^*)}(P,N)$ must vanish for all $s\geq 1$ as well. For this limit vanishing, one can use the two Grothendieck spectral sequences \eqref{gss 1} and \eqref{gss 1000003} from the first proof of Lemma \ref{projectivity lemma}, as we did to show the vanishing of $\lim_{j\rightarrow\infty} \Ext^s_{\gr\Mod(\Gamma^*)}(P,N_{\leq j})$ in that proof. All that we have avoided in this second proof of the lemma is direct use of the first spectral sequence in the first proof, \eqref{ss 203000}.
\end{proof}

Finally we are ready to state and prove the main theorem of this section, Theorem \ref{no nonzero projs}. The motivating example of a coalgebra $\Gamma$ satisfying all the hypotheses of Theorem \ref{no nonzero projs} is the mod $p$ dual Steenrod algebra for any prime $p$. 
\begin{theorem}\label{no nonzero projs}
Let $k$ be a field, and let $\Gamma$ be a finite-type connected graded $k$-coalgebra which satisfies the hypothesis of Lemma \ref{mitchell iso lemma} and whose dual $\Gamma^*$ is a $\mathcal{P}$-algebra. Then there are no nonzero projective graded $\Gamma$-comodules.
\end{theorem}
\begin{proof}
Let $P$ be a projective graded $\Gamma$-comodule. By Theorem \ref{ext iso 1}, we have the isomorphism
\begin{align*}
 \Ext^n_{\gr\Comod(\Gamma)}(P,M) &\cong 
 \Ext^n_{\gr\Mod(\Gamma^*)}(\iota(P),\iota(M))
\end{align*}
for all $n$ and all graded $\Gamma$-comodules $M$. Hence $\Ext^n_{\gr\Mod(\Gamma^*)}(\iota(P),N)$ vanishes for all $n>0$ and all rational $\Gamma^*$-modules $N$. Every bounded-above graded $\Gamma^*$-module is rational, by Theorem \ref{main thm 1 from derived prods paper}, so $\Ext^n_{\gr\Mod(\Gamma^*)}(\iota(P),N)$ vanishes for all bounded-above $N$ and all $n>0$. Hence, by Lemma \ref{projectivity lemma}, the graded $\Gamma^*$-module $\iota(P)$ is projective. It is well-known (e.g. see \cite[section 11.2]{MR738973}) that over a connected algebra, projective graded modules are free. Hence the graded $\Gamma^*$-module $\iota(P)$ is free. But by Lemma \ref{mitchell iso lemma}, $\iota(P)$ can have no nonzero free summands. Hence $\iota(P)\cong 0$, and since $\iota$ is faithful, $P\cong 0$.
\end{proof}

\section{$V$ for a $\mathcal{P}$-algebra}

We resume our analysis of the behavior of the right adjoint $V: \gr\Comod(\Gamma)\rightarrow\gr\Mod(k)$ of the extended comodule functor for certain classes of connected graded $k$-coalgebras $\Gamma$. In \cref{poincare alg section}, we carried out this analysis when the dual of $\Gamma$ is a Poincar\'{e} algebra. In \cref{I in the case of a P-algebra} we will carry out the analysis when the dual of $\Gamma$ is a $\mathcal{P}$-algebra, as defined in Definition \ref{def of p-alg}. As a special case, in \cref{I in the case of the Steenrod alg} we specialize to the case where $\Gamma^*$ is the mod $p$ Steenrod algebra for some prime $p$. We derive various topological consequences which generalize the 1976 theorem of Lin \cite{MR0402738}.

\subsection{General behavior of $V$ for a $\mathcal{P}$-algebra}
\label{I in the case of a P-algebra}

\begin{observation}\label{margolis observation}
In the proof of Proposition \ref{cotorsion vanishes on boundeds}, we will make use of the following observation of Margolis \cite[item (4) from list on pg. 193]{MR738973}: if $A$ is a $\mathcal{P}$-algebra, the union of a nested sequence $A(0)\subsetneq A(1)\subsetneq \dots$ as in Definition \ref{def of p-alg}, then the top nonvanishing degree of $A(n+1)$ must be strictly greater than the top nonvanishing degree of $A(n)$. This is because $A(n)\subsetneq A(n+1)$, both are connected $k$-algebras, and $A(n+1)$ is flat over $A(n)$ and hence also free over $A(n)$ since $A(n)$ is a Poincar\'{e} algebra.
\end{observation}

\begin{prop}\label{cotorsion vanishes on boundeds}
Let $\Gamma$ be a connected finite-type graded $k$-algebra whose $k$-linear dual $\Gamma^*$ is a $\mathcal{P}$-algebra. 
Let $M$ be a bounded-below graded $\Gamma$-comodule. Then $V(M)$ is trivial, i.e., $M$ is cotorsion.
\end{prop}
\begin{proof}
We use the covariant embedding $\iota$ of $\gr\Comod(\Gamma)$ into $\gr\Mod(\Gamma^*)$: since $M$ is bounded below, so is $\iota(M)$. Choose an integer $n$ such that $\iota(M)$ is trivial in degrees $<n$. 

Now choose some integer $j$. Suppose that $f: \Sigma^j \iota(\Gamma)\rightarrow M$ is a graded $\Gamma^*$-comodule morphism from $\Sigma^j \iota(\Gamma)$ to $M$. Since $V(M) \cong \hom_{\gr\Mod(\Gamma^*)}(\iota(\Gamma),\iota(M))$ as a consequence of Proposition \ref{I prop 1} and Theorem \ref{ext iso 1}, we aim to show that $f$ must be the zero morphism. 

For any nonnegative integer $d$, the underlying $A(d)$-module of $\iota(\Gamma)$ is free, since the same is true for $\Gamma^* = \iota(\Gamma)^*$  \cite[item (3) from list on pg. 193]{MR738973} and since $A(d)$ is Poincar\'{e}. 
When we split the underlying graded $A(d)$-module of $\Sigma^j\iota(\Gamma)$ into a direct sum of suspensions of $A(d)$, each of those summands can only map nontrivially to $M$ if it sends the generator of the summand to a nonzero element. 

But $\Sigma^j\iota(\Gamma)$ is {\em free} over $A(d)$, and furthermore vanishes in degrees $>j$ since $\Gamma$ is concentrated in nonpositive degrees. Hence, if the top degree of $A(d)$ is in a sufficiently high degree, then the generators of the underlying graded $A(d)$-module of $\Sigma^j\iota(\Gamma)$ must all live in degree $< n$, hence cannot map nontrivially into $M$ (since $M$ is trivial in degrees $<n$).

Hence all we need is to know that, for every integer $N$, we may choose $d$ so that the top degree of $A(d)$ is greater than $N$. The existence of such an integer $d$ is guaranteed by Observation \ref{margolis observation}.\end{proof}
Proposition \ref{cotorsion vanishes on boundeds} shows that, under appropriate hypotheses, a bounded-below comodule $M$ is cotorsion, i.e., $V(M)\cong 0$ . It is natural to ask whether it is also true that a bounded-below comodule is {\em derived} cotorsion, i.e., $R^*V(M)\cong 0$. For general $\Gamma$, the author does not know the answer to that question. However, in the case of greatest topological interest, the case where the coalgebra $\Gamma$ is the mod $p$ dual Steenrod algebra for some prime number $p$, we shall prove in Corollary \ref{top thm 1 cor 2} that bounded-below $\Gamma$-comodules are indeed derived cotorsion. Perhaps surprisingly, the proof is topological, using some ideas from chromatic homotopy theory. The author does not know a purely algebraic proof.

\begin{prop}\label{vanishing above R1V}
Let $k$ be a field, and let $\Gamma$ be a finite-type connected graded $k$-coalgebra  which satisfies the hypothesis of Lemma \ref{mitchell iso lemma} and whose dual algebra $\Gamma^*$ is a $\mathcal{P}$-algebra.
Then the functor $V: \gr\Comod(\Gamma) \rightarrow \gr\Mod(k)$ is left exact, but not right exact. Its right-derived functors $R^0V \cong V$ and $R^1V$ are not always zero, but its higher derived functors, $R^nV$ for all $n>1$, are zero.
\end{prop}
\begin{proof}
Left exactness follows from $V$ being a right adjoint. 
By Proposition \ref{I prop 1}, $R^nV(M)\cong\Ext^n_{\gr\Comod(\Gamma)}(\Gamma,M)$, so if $R^nV(M)$ were to vanish for all $n>0$ and all $M$, then $\Gamma$ would be projective in $\gr\Comod(\Gamma)$, which cannot be the case, by Theorem \ref{no nonzero projs}. Hence $R^nV(M)$ is nonzero for {\em some} $n>0$ and {\em some} $M$.

Hence our claim that $R^nV(M)$ is trivial for all $n>1$ implies our claim that $R^1V(M)$ is nontrivial for some $M$. To show that $R^nV(M)$ is trivial for all $n>1$, we use Theorem \ref{ext iso 1} in conjunction with Proposition \ref{I prop 1}:
\begin{align}
\label{ext rnv formula 1} R^nV(M) &\cong \Ext^n_{\gr\Mod(\Gamma^*)}\left(\iota(\Gamma),\iota(M)\right).
\end{align}
Since $\Gamma$ is a bounded-above injective in $\gr\Comod(\Gamma)$, Theorem \ref{main thm 1 from derived prods paper} gives us that $\iota(\Gamma)$ is injective in $\gr\Mod(\Gamma^*)$. By Theorem \ref{margolis thm 1}, the graded $\Gamma^*$-module $\iota(\Gamma)$ must have projective dimension $\leq 1$, hence $R^nV(M)\cong 0$ for $n>1$ by \eqref{ext rnv formula 1}. (In fact $\iota(\Gamma)$ must have projective dimension exactly $1$, by making one more use of Theorem \ref{no nonzero projs}. But the proof is finished without this observation.)
\end{proof}

\subsection{$V$ for the mod $p$ dual Steenrod algebra}
\label{I in the case of the Steenrod alg}
We now restrict attention further: throughout this section, fix a prime $p$, and let $\Gamma$ denote the mod $p$ dual Steenrod algebra.

Recall the theorems of Margolis and T.Y. Lin, which characterize maps from Eilenberg-Mac Lane spectra into bounded-below spectra, under appropriate hypotheses:
\begin{theorem} \label{margolis top thm} {\bf (Margolis \cite[Theorem 3]{MR341488}.)} 
Let $Y$ be a bounded-below finite-type spectrum. Let $\Gamma^*$ denote the mod $p$ Steenrod algebra. Then we have an isomorphism of graded abelian groups
\begin{align*}
 \left[ \Sigma^* H\mathbb{F}_p,Y\right]
  &\cong \hom_{\gr\Mod(\Gamma^*)}\left( H^*(Y;\mathbb{F}_p), \Gamma^*\right).
\end{align*}
\end{theorem}

\begin{theorem} \label{lins thm} {\bf (Lin \cite[Theorem 4.2]{MR0402738}.)} Let $Y$ be the suspension spectrum of a finite-type CW-complex.
Let $A$ be an abelian group. Then we have an isomorphism of graded abelian groups
\begin{align*} \left[ \Sigma^* HA,Y\right] 
 &\cong \Ext_{\mathbb{Z}}^1\left( A\otimes_{\mathbb{Z}}\mathbb{Q},H_{*+1}(Y;\mathbb{Z})\right).
\end{align*}
\end{theorem}

With the results already proven in this paper, here is a generalization of Margolis's theorem to those bounded-below spectra $Y$ which are not necessarily finite-type. 
\begin{theorem}\label{top thm 1}
Let $p$ be a prime, and let $Y$ be a spectrum. Then there exists a split short exact sequence of graded $\mathbb{F}_p$-vector spaces
\begin{equation}\label{ses 11049} 0 \rightarrow
 R^1V\left(H_*(\Sigma Y;\mathbb{F}_p)\right) \rightarrow
 \left[ \Sigma^*H\mathbb{F}_p, \hat{Y}_{H\mathbb{F}_p}\right] \rightarrow
V\left(H_*(Y;\mathbb{F}_p)\right) \rightarrow
 0.
\end{equation}
The sequence, but not necessarily the splitting, is natural in the variable $Y$.

In particular, if $Y$ is bounded below, then we have an isomorphism
\begin{align}
\label{iso 204500} \left[ \Sigma^*H\mathbb{F}_p, Y\right]
  &\cong R^1V\left(H_*(\Sigma Y;\mathbb{F}_p)\right)\oplus V\left(H_*(Y;\mathbb{F}_p)\right).
\end{align}
Consequently the graded $\Gamma$-comodule $H_*(Y; \mathbb{F}_p)$ is derived cotorsion if and only if the mapping spectrum $F(H\mathbb{F}_p,Y)$ is contractible.
\end{theorem}
Out of convenience, in the statement of Theorem \ref{top thm 1}, we have temporarily dropped the cohomological grading convention from \cref{conventions}. The grading on $H_*(Y; \mathbb{F}_p)$ and on $H_*(\Sigma Y;\mathbb{F}_p)$ in the statement of the theorem is the homological grading, i.e., $H_n(Y;\mathbb{F}_p)$ is indeed in degree $n$.

As mentioned in the introduction, to see that Theorem \ref{top thm 1} generalizes Margolis's theorem, observe the following:
\begin{itemize}
\item as a consequence of Proposition \ref{I prop 1} and basic properties of linear duals, $V(H_*(Y;\mathbb{F}_p))$ is isomorphic to $\hom_{\gr\Mod(\Gamma^*)}(H^*(Y;\mathbb{F}_p),\Gamma^*)$, because $Y$ is assumed finite-type.
\item For $Y$ bounded-below and finite-type, linear dualization yields an isomorphism (surely this was well-known around the time of Milnor--Moore's paper \cite{MR0174052}, but we don't know many explicit mentions in the literature; \cite[section 2]{bakerPalgebraspreprint2} has a nice exposition, however):
\begin{align}\label{dualization iso 1}
 \Ext^*_{\gr\Comod(\Gamma)}\left(\Gamma,H_*(Y;\mathbb{F}_p) \right)
  &\cong\Ext^*_{\gr\Mod(\Gamma^*)}\left(H^*(Y;\mathbb{F}_p),\Gamma^* \right)
\end{align}
By Margolis's theorem \cite{MR341488} that $\Gamma^*$ is injective in the category of {\em bounded-below} graded $\Gamma^*$-modules, the right-hand side of \eqref{dualization iso 1} vanishes in positive cohomological degrees. 
Hence $R^1V(H_*(Y;\mathbb{F}_p)) \cong \Ext^1_{\gr\Comod(\Gamma)}(\Gamma, H_*(Y;\mathbb{F}_p))$ vanishes.
\end{itemize}
\begin{proof}[Proof of Theorem \ref{top thm 1}]
By Propositions \ref{adams E2 description} and \ref{vanishing above R1V}, the $H\mathbb{F}_p$-Adams spectral sequence converging to $\left[ \Sigma^* H\mathbb{F}_p,\hat{Y}_{H\mathbb{F}_p}\right] \cong \pi_*\left(  H\mathbb{F}_p,\hat{Y}_{H\mathbb{F}_p}\right)$ is concentrated on the $\Ext^0$ and $\Ext^1$-lines already at the $E_2$-page. The horizontal vanishing line at $E_2$ yields complete convergence of the spectral sequence, in the sense of Bousfield, and it leaves no room for nonzero differentials and only one extension problem, which is precisely the short exact sequence \eqref{ses 11049}.

For bounded-below $Y$, the $H\mathbb{F}_p$-nilpotent completion $\hat{Y}_{H\mathbb{F}_p}$ of $Y$ is weakly homotopy equivalent to
the $p$-completion of $Y$, i.e., the Bousfield localization of $Y$ at the Moore spectrum $S/p$, by \cite[Theorem 6.6]{MR551009}. This is why there are no $H\mathbb{F}_p$-nilpotent completion symbols in \eqref{iso 204500}. We identify $[\Sigma^t H\mathbb{F}_p,Y]$ with 
$\left[ \Sigma^t H\mathbb{F}_p,L_{S/p}Y\right]$
using the homotopy pullback square
\[\xymatrix{
 Y \ar[r]\ar[d] & \prod_{\ell\text{\ prime}} L_{S/\ell}Y \ar[d] \\
 H\mathbb{Q}\wedge Y \ar[r] & H\mathbb{Q}\wedge\left( \prod_{\ell\text{\ prime}} L_{S/\ell}Y \right) 
}\] 
of \cite[Proposition 2.9]{MR551009} and the fact that $H\mathbb{F}_p$ cannot map non-nulhomotopically into a rational spectrum or a spectrum completed away from $p$. 

The splittings of these short exact sequences exist simply because these are sequences of graded $\mathbb{F}_p$-vector spaces.
\end{proof}

\begin{corollary}\label{top thm 1 cor}
Let $p$ be a prime, and let $\Gamma$ be the mod $p$ dual Steenrod algebra. 
Let $Y$ be a bounded-below $p$-complete spectrum. Suppose either that the homology groups of $Y$ are concentrated in finitely many degrees, or that $Y$ is the suspension spectrum of a finite-type CW-complex. Then the graded $\Gamma$-comodule $H_*(Y; \mathbb{F}_p)$ is derived cotorsion.
\end{corollary}
\begin{proof}
If $Y$ is the suspension spectrum of a finite-type CW-complex,
then by Lin's theorem, $F(H\mathbb{F}_p,Y)$ is contractible. Hence the homology of $Y$ is derived cotorsion as a consequence of Theorem \ref{top thm 1}.
If $Y$ is instead assumed to have homology concentrated in finitely many degrees, then we use a different argument. 
Recall, e.g. from \cite[section 1]{MR737778}, that a $p$-local spectrum is {\em harmonic} if it is local, in the sense of Bousfield, with respect to the wedge $E(0)\vee E(1)\vee E(2)\vee\dots$ of the $p$-local Johnson--Wilson theories\footnote{We refer the reader to \cite{MR737778} for a complete introduction, but here is at least a footnote for readers who are unfamiliar with this circle of ideas. For each prime $p$, there is a sequence of generalized homology theories $E(0)_*, E(1)_*, E(2)_*, \dots$, the {\em Johnson--Wilson theories,} which depend on $p$ but for which the prime $p$ is traditionally suppressed from the notation. The zeroth Johnson--Wilson theory $E(0)_*$ is rational homology, while the first Johnson--Wilson theory $E(1)_*$ is the Adams summand of $p$-local complex $K$-theory. The coefficient ring of $E(n)_*$ is the graded algebra $\mathbb{Z}_{(p)}[v_1, \dots ,v_{n-1},v_n^{\pm 1}]$ over the coefficient ring $\mathbb{Z}_{(p)}[v_1, xv_2, \dots ]$ of $p$-local Brown--Peterson homology. Basics of Bousfield localization are reviewed in Recollection \ref{bousfield recollection}.}. Similarly, a $p$-local spectrum is {\em dissonant} if it is acyclic with respect to $E(0)\vee E(1)\vee \dots$, i.e., its $E(n)$-homology vanishes for all $n\geq 0$. Finite spectra are harmonic \cite[Corollary 4.5]{MR737778}, and more generally, suspension spectra are harmonic \cite{MR1317578}. 

A bounded-below spectrum whose homology groups are concentrated in finitely many degrees admits a CW-decomposition with cells in only finitely many degrees. Any such spectrum is harmonic, as one can prove by a finite induction up the skeleta of the CW-decomposition. The initial step in the induction is the observation that any coproduct of spheres is a suspension spectrum and hence harmonic. The inductive step is the observation that, given a cofiber sequence $Y_1\rightarrow Y_2 \rightarrow Y_3$ in which two of the three spectra are harmonic, the third is also harmonic, by generalities about Bousfield localization.

Hence the spectrum $Y$, from the statement of the corollary, is harmonic. Consequently every map from a dissonant spectrum to $Y$ is nulhomotopic. It is classical \cite[Theorem 4.7]{MR737778} that any suspension of $H\mathbb{F}_p$ is dissonant. Hence $\left[ \Sigma^*H\mathbb{F}_p, Y\right]$ is trivial, hence $R^nV(H_*(Y; \mathbb{F}_p))$ vanishes for $n=0,1$ by Theorem \ref{top thm 1}, hence $R^nV(H_*(Y; \mathbb{F}_p))$ vanishes for all $n$ by Proposition \ref{vanishing above R1V}.
\end{proof}

Here is a purely algebraic consequence of Corollary \ref{top thm 1 cor}, a strengthening of Proposition \ref{cotorsion vanishes on boundeds}. The statement of this corollary is purely algebraic, but our proof passes through stable homotopy theory, since the proof of Corollary \ref{top thm 1 cor} used some properties of Bousfield localization of spectra. Surely a purely algebraic proof is possible, but the author does not know one.
\begin{corollary}\label{top thm 1 cor 2}
Let $p$ be a prime, and let $\Gamma$ be the mod $p$ dual Steenrod algebra. 
Let $M$ be a bounded-below graded $\Gamma$-comodule. Then $M$ is derived cotorsion, i.e., $R^*V(M)$ vanishes.
\end{corollary}
\begin{proof}
Let $N$ be a graded $\Gamma$-comodule which is concentrated in a single degree, i.e., $N$ is a coproduct of copies of $\Sigma^n \mathbb{F}_p$ for some single integer $n$. Then $N$ is the mod $p$ homology of a wedge of $(-n)$-spheres (by our cohomological grading convention from \cref{conventions}), so by Corollary \ref{top thm 1 cor}, $N$ is derived cotorsion. 

Now let $d$ be the least degree in which $M$ is nonzero.
We have the sequence of surjective maps of graded $\Gamma$-comodules 
\begin{equation}\label{seq 3003a} \dots \rightarrow M_{\leq d+2} \rightarrow M_{\leq d+1}\rightarrow M_{\leq d}\rightarrow M_{\leq d-1} = 0\end{equation}
defined in Lemma \ref{easy limit in comodules}. 
Since $M_{\leq d}$ is concentrated in a single degree, it is derived cotorsion. 

That was the initial step in an induction. The inductive step is as follows:
suppose we have already shown that $M_{\leq i}$ is derived cotorsion for some integer $i$. Since the kernel of $M_{\leq i+1}\rightarrow M_{\leq i}$ is concentrated in a single degree, $M_{\leq i+1}$ is an extension of a derived cotorsion comodule by a derived cotorsion comodule. Hence $M_{\leq i+1}$ is also derived cotorsion. 

Hence $M \cong \lim_i^{\Gamma} M_{\leq i}$ is a limit of comodules which are each derived cotorsion\footnote{See comment preceding Lemma \ref{easy limit in comodules} for the notation $\lim^{\Gamma}$, which denotes the limit taken in the category of graded $\Gamma$-comodules.}. The isomorphism $M \cong \lim_i^{\Gamma} M_{\leq i}$, and the vanishing of $R^n\lim_i^{\Gamma}M_{\leq i}$ for $n>0$, were shown in Lemma \ref{easy limit in comodules}. We now have isomorphisms
\begin{align}
\nonumber R^nV(M)
  &\cong \Ext^n_{\gr\Comod(\Gamma)}\left(\Gamma,\lim_i{}^{\Gamma} M_{\leq i} \right) \\
\label{iso 1000001}  &\cong R^n\left(\hom_{\gr\Comod(\Gamma)}\left(\Gamma,\lim_i{}^{\Gamma}(-) \right)\right)\left( \dots \rightarrow M_{\leq 1}\rightarrow M_{\leq 0}\rightarrow \dots\right) \\
\label{iso 1000002}  &\cong R^n\left(\lim\hom_{\gr\Comod(\Gamma)}\left(\Gamma, -\right)\right)\left( \dots \rightarrow M_{\leq 1}\rightarrow M_{\leq 0}\rightarrow \dots\right) ,
\end{align}
where the isomorphism \eqref{iso 1000001}
is due to collapse of the spectral sequence \eqref{gss 1000000} from the proof of Theorem \ref{ext iso 1}. The collapse of the spectral sequence was proven in that theorem's proof as well.

The right derived functor \eqref{iso 1000002} is calculable using another Grothendieck spectral sequence
\begin{align*}
 E_2^{s,t} &\cong R^s\lim_i \left( \Ext^t_{\gr\Comod(\Gamma)}(\Gamma,M_{\leq i})\right) \\
 &\Rightarrow R^{s+t}\left(\lim\hom_{\gr\Comod(\Gamma)}\left(\Gamma, -\right)\right)\left( \dots \rightarrow M_{\leq 1}\rightarrow M_{\leq 0}\rightarrow \dots\right).
\end{align*}
See the note in the last paragraph of the proof of Lemma \ref{enough injs lemma} for an explanation of why this Grothendieck spectral sequence exists, i.e., why every object in the category $\gr\Comod(\Gamma)^{\mathbb{N}^{\op}}$ of inverse sequences of graded comodules embeds into an injective object in $\gr\Comod(\Gamma)^{\mathbb{N}^{\op}}$ which is sent by the functor 
\begin{align*}
 \hom_{\gr\Comod(\Gamma)}\left(\Gamma, -\right)
 : \gr\Comod(\Gamma)^{\mathbb{N}^{\op}} &\rightarrow \gr\Ab
\end{align*}
to a $\lim$-acyclic sequence of abelian groups.

Now we simply observe that the $E_2$-page of this spectral sequence is zero, since we already showed that each $M_{\leq i}$ is derived cotorsion. Hence $M$ itself is derived cotorsion.
\end{proof}

With Theorem \ref{top thm 1} in hand, we also get generalizations of the theorem of Lin (Theorem \ref{lins thm}) to the case of a bounded-below spectrum $Y$ which is not a suspension spectrum and not finite-type. None of the generalizations are as simply stated as Theorem \ref{lins thm}, because it is simply a fact that, without Lin's assumptions, it is a bit more complicated to describe all maps from an Eilenberg-Mac Lane spectrum to $Y$. Most of the complications really arise from the derived functors of $V$ applied to the homology comodule $H_*(Y;\mathbb{F}_p)$. When $R^*V(H_*(Y;\mathbb{F}_p))$ vanishes (i.e., when $H_*(Y;\mathbb{F}_p)$ is derived cotorsion), we get a satisfying and simple statement, Corollary \ref{nice lin cor}.

We prefer to state several generalizations of Lin's theorem, each of which works in a certain reasonable level of generality, more general than Lin's theorem. While it is possible to formulate a single generalization that covers all these cases, its statement is so cumbersome that we feel it is not worth including, since we suspect that each case of likely application of the general result is also a case of one of the special results, Theorems \ref{torsion group top thm} and \ref{derived cotorsion top thm} and \ref{torsion-free HA}, which are much more easily stated.

For the rest of this section, our results are most conveniently stated using homological grading on the homology comodule $H_*(Y;\mathbb{F}_p)$ of a spectrum $Y$, so that $H_n(Y;\mathbb{F}_p)$ is in degree $n$. From now on, we abandon the cohomological grading convention from \cref{conventions}.
\begin{theorem}\label{torsion group top thm}
Let $p$ be a prime, let $Y$ be a bounded-below spectrum, and let $A$ be an abelian $p$-group. Write $A[p^n]$ for the $p^n$-torsion subgroup of $A$. Then there exists a conditionally convergent spectral sequence
\begin{align*}
 E_1^{s,t} &\cong \hom_{\mathbb{F}_p}\left( A[p^s]/A[p^{s-1}],\left(R^1V(H_*(\Sigma Y;\mathbb{F}_p))\oplus V(H_*(Y;\mathbb{F}_p))\right)^t \right)\\
  & \Rightarrow \left[ \Sigma^t HA,Y\right] \\
 d_r: E_r^{s,t} &\rightarrow E_r^{s+r,t-1},
\end{align*} 
where $\left(R^1V(H_*(\Sigma Y;\mathbb{F}_p))\oplus V(H_*(Y;\mathbb{F}_p))\right)^t$ denotes the degree $t$ summand in the graded $\mathbb{F}_p$-vector space $R^1V(H_*(\Sigma Y;\mathbb{F}_p))\oplus V(H_*(Y;\mathbb{F}_p))$.
\end{theorem}
\begin{proof}
We have the tower of fiber sequences
\[\xymatrix{
 \pt \ar[r] & H(A[p]) \ar[r]\ar[d] & H(A[p^2]) \ar[r]\ar[d] & H(A[p^3])\ar[d] \ar[r] & \dots \\
 & H(A[p]) & H(A[p^2]/A[p]) & H(A[p^3]/A[p^2]) & 
 }\]
and consequently, upon applying the functor $\left[\Sigma^* -, Y\right]$,
a conditionally convergent spectral sequence
\begin{align*}
 E_1^{s,t} \cong \left[ \Sigma^t H(A[p^s]/A[p^{s-1}]),Y\right]
  & \Rightarrow \left[ \Sigma^t HA,Y\right] .
\end{align*}
To get the claimed $E_1$-page description, for each $s$ we choose an $\mathbb{F}_p$-linear basis $B_s$ for $A[p^s]/A[p^{s-1}]$, and then we use the chain of isomorphisms
\begin{align*}
        \left[ \Sigma^t H(A[p^s]/A[p^{s-1}]),Y\right] 
 &\cong \left[ \Sigma^t \coprod_{b\in B_s}H\mathbb{F}_p ,Y\right] \\
 &\cong \hom_{\mathbb{F}_p}\left( A[p^s]/A[p^{s-1}], \left[ \Sigma^t H\mathbb{F}_p ,Y\right]\right),
\end{align*}
and then we identify $[\Sigma^t H\mathbb{F}_p,Y]$ with 
$\left[ \Sigma^t H\mathbb{F}_p,L_{S/p}Y\right]$ 
using the homotopy pullback square argument from the proof of Theorem \ref{top thm 1}.
Finally, 
$\left[ \Sigma^t H\mathbb{F}_p,L_{S/p}Y\right]$ is described in terms of $R^*V(H(Y; \mathbb{F}_p))$ by Theorem \ref{top thm 1}.
\end{proof}

\begin{theorem}\label{derived cotorsion top thm}
Let $P$ be a set of primes, let $Y$ be a bounded-below spectrum, and let $A$ be an abelian group. Make the following assumptions:
\begin{itemize}
\item Every prime in $P$ acts isomorphically on $A$, i.e., $A$ is a $\mathbb{Z}[P^{-1}]$-module.
\item For every prime number $p\notin P$, the comodule $H_*(Y; \mathbb{F}_p)$ over the mod $p$ dual Steenrod algebra is derived cotorsion, i.e. (by Proposition \ref{vanishing above R1V}), $V(H_*(Y; \mathbb{F}_p))$ and $R^1V(H_*(Y; \mathbb{F}_p))$ are each trivial.
\end{itemize}
Then for each integer $n$, there is a short exact sequence
\begin{equation}\label{ses 00112} 0 \rightarrow \Ext^1_{\mathbb{Z}}\left( \mathbb{Q}\otimes_{\mathbb{Z}}A, \pi_{n+1}Y\right) \rightarrow \left[ \Sigma^n HA,Y\right] \rightarrow \hom_{\mathbb{Z}}\left(\mathbb{Q}\otimes_{\mathbb{Z}}A, \pi_nY\right)\rightarrow 0,\end{equation}
natural in the variable $Y$.
\end{theorem}
\begin{proof}
Let $T$ be a torsion abelian group. Then $T$ splits as the direct sum $\coprod_{p\ \text{prime}} T_p$ of torsion abelian $p$-groups $T_p$ for various primes $p$. If $p\notin P$, then the mapping spectrum $F(H(T_p),Y)$ is contractible, by Theorem \ref{torsion group top thm}. 

In particular, since each prime $p\in P$ acts isomorphically on $A$, the torsion subgroup $\tors(A)$ is a product of torsion abelian $p$-groups for $p\notin P$, and consequently $F(H(\tors(A)),Y)$ is contractible. Hence $F(H(A/\tors(A)),Y) \rightarrow F(HA,Y)$ is a weak equivalence. 

We have the fiber sequence
\[ H(A/\tors (A)) \rightarrow H\left((A/\tors(A))\otimes_{\mathbb{Z}}\mathbb{Q}\right) \rightarrow H\left((A/\tors(A))\otimes_{\mathbb{Z}}\mathbb{Q}/\mathbb{Z}\right)
\]
and since $(A/\tors(A))\otimes_{\mathbb{Z}}\mathbb{Q}/\mathbb{Z}$ is again a product of torsion abelian $p$-groups for $p\notin P$, we have contractibility of $F(H\left((A/\tors(A))\otimes_{\mathbb{Z}}\mathbb{Q}/\mathbb{Z}\right),Y)$ and consequently a natural weak equivalence
\begin{align}
\label{iso 00111} F(HA,Y) &\stackrel{\simeq}{\longrightarrow} F\left( H(A\otimes_{\mathbb{Z}}\mathbb{Q}),Y\right),
\end{align}
using the isomorphism $A\otimes_{\mathbb{Z}}\mathbb{Q} \cong (A/\tors(A))\otimes_{\mathbb{Z}}\mathbb{Q}$.

The sequence \eqref{ses 00112} now is a consequence of equivalence \eqref{iso 00111} together with the short exact sequence
\begin{equation}\label{ses 00113} 0 \rightarrow \Ext^1_{\mathbb{Z}}\left( W, \pi_{n+1}Y\right) \rightarrow \left[ \Sigma^n HW,Y\right] \rightarrow \hom_{\mathbb{Z}}\left(W, \pi_nX\right)\rightarrow 0\end{equation}
which one has for any $\mathbb{Q}$-vector space $W$ and any spectrum $X$. The sequence \eqref{ses 00113} is quite old: when Lin makes use of a similar argument in \cite{MR0402738}, he cites Hilton's 1965 treatment \cite[chapter 5]{MR198466}, which is entirely in the unstable setting, and also written in terms which are today unfamiliar to many homotopy theorists. For self-containedness, we give a somewhat more modern derivation of the short exact sequence \eqref{ses 00113}, as follows. Choose a $\mathbb{Q}$-linear basis $B$ for $W$, and write $SB$ for the wedge $\coprod_{b\in B}S$ of copies of the sphere spectrum, one for each member of the basis $B$.
Then $HW$ is the homotopy colimit of the sequence of spectra
\begin{equation}\label{seq 00114} SB \stackrel{2}{\longrightarrow} SB \stackrel{2\cdot 3}{\longrightarrow} SB 
\stackrel{2\cdot 3\cdot 5}{\longrightarrow} SB 
\stackrel{2\cdot 3\cdot 5\cdot 7}{\longrightarrow} \dots .\end{equation}
Hence the Milnor sequence for the homotopy groups of the mapping spectrum
\[ F(HW,X)\simeq \holim\left( \dots \stackrel{2\cdot 3\cdot 5}{\longrightarrow} \prod_{b\in B} X\stackrel{2\cdot 3}{\longrightarrow} \prod_{b\in B} X \stackrel{2}{\longrightarrow} \prod_{b\in B} X \right)\]
expresses $[\Sigma^n HW, X]$ as an extension of 
\[\lim\left( \dots \stackrel{2\cdot 3\cdot 5}{\longrightarrow} \prod_{b\in B} \pi_nX\stackrel{2\cdot 3}{\longrightarrow} \prod_{b\in B} \pi_nX \stackrel{2}{\longrightarrow} \prod_{b\in B} \pi_nX \right)\]
by the derived limit
\[\lim{}^1\left( \dots \stackrel{2\cdot 3\cdot 5}{\longrightarrow} \prod_{b\in B} \pi_{n+1}X\stackrel{2\cdot 3}{\longrightarrow} \prod_{b\in B} \pi_{n+1}X \stackrel{2}{\longrightarrow} \prod_{b\in B} \pi_{n+1}X \right).\]
It is a straightforward algebraic exercise to check that the former is isomorphic to $\hom_{\mathbb{Z}}(\mathbb{Q},\prod_{b\in B}\pi_nX)$ and consequently $\hom_{\mathbb{Z}}(W,\pi_nX)$, while the latter is isomorphic to $\Ext^1_{\mathbb{Z}}(\mathbb{Q},\prod_{b\in B}\pi_{n+1}X)$ and consequently $\Ext^1_{\mathbb{Z}}(W,\pi_{n+1}X)$.
\end{proof}
The sequence \eqref{ses 00112} involves the homotopy groups of $Y$, while Lin's theorem, Theorem \ref{lins thm}, instead is stated in terms of the homology groups of $Y$. In Lin's situation, since he assumes $Y$ is finite-type, he knows that the kernel and cokernel of the Hurewicz map $\pi_*(Y) \rightarrow H_*(Y;\mathbb{Z})$ are each {\em finite} torsion groups in each degree. This lets Lin replace the role of $\pi_*Y$ in the universal coefficient sequence \eqref{ses 00113} with $H_*(Y;\mathbb{Z})$. Without the assumption that $Y$ is finite-type, one cannot make this leap, unless one makes some other assumptions on $Y$.

\begin{corollary}\label{nice lin cor}
Let $Y$ be a bounded-below spectrum whose homology groups are concentrated in a finite range of degrees. 
Let $A$ be an abelian group.
Then for each integer $n$, there is a short exact sequence
\begin{equation*}
 0 \rightarrow \Ext^1_{\mathbb{Z}}\left( \mathbb{Q}\otimes_{\mathbb{Z}}A, \pi_{n+1}Y\right) \rightarrow \left[ \Sigma^n HA,Y\right] \rightarrow \hom_{\mathbb{Z}}\left(\mathbb{Q}\otimes_{\mathbb{Z}}A, \pi_nY\right)\rightarrow 0.\end{equation*}
\end{corollary}
\begin{proof}
By Corollary \ref{top thm 1 cor 2}, the mod $p$ homology $H_*(Y;\mathbb{F}_p)$ is derived cotorsion for all $p$. Hence Theorem \ref{derived cotorsion top thm} applies, letting $P = \emptyset$ in the statement of the theorem. 
\end{proof}

The most complicated generalization of Lin's theorem which we will bother to state is Theorem \ref{torsion-free HA}, which involves aspects of both Theorems \ref{torsion group top thm} and \ref{derived cotorsion top thm}.
\begin{theorem}\label{torsion-free HA}
Let $p$ be a prime, let $Y$ be a bounded-below $p$-complete spectrum, and let $A$ be a torsion-free abelian group in which all primes other than $p$ are inverted, i.e., $A$ is a torsion-free $\mathbb{Z}_{(p)}$-module. 
Then the mapping spectrum $F(HA,Y)$ sits in the fiber sequence
\[ F(HA,Y) \leftarrow F\left( H( A\otimes_{\mathbb{Z}}\mathbb{Q}),Y\right) \leftarrow
F\left( H( A\otimes_{\mathbb{Z}}\mathbb{Q}/\mathbb{Z}),Y\right)\]
whose other two terms are in principle calculable as follows:
for each $n$ there is a short exact sequence
\begin{equation*}
 0 \rightarrow \Ext^1_{\mathbb{Z}}\left( \mathbb{Q}\otimes_{\mathbb{Z}}A, \pi_{n+1}Y\right) \rightarrow \pi_nF\left( H(A\otimes_{\mathbb{Z}}\mathbb{Q}),Y)\right) \rightarrow \hom_{\mathbb{Z}}\left(\mathbb{Q}\otimes_{\mathbb{Z}}A, \pi_nY\right)\rightarrow 0,\end{equation*} 
for the calculation of $\pi_*F\left( H(A\otimes_{\mathbb{Z}}\mathbb{Q}),Y)\right)$, and there is a conditionally convergent spectral sequence
\begin{align}
\label{ss 77801} E_1^{s,t} &\cong \hom_{\mathbb{F}_p}\left( p^sA/p^{s+1}A, \left(R^1V(H_*(\Sigma Y;\mathbb{F}_p))\oplus V(H_*(Y;\mathbb{F}_p))\right)^t \right)\\
\nonumber  & \Rightarrow \left[ \Sigma^t H(A\otimes_{\mathbb{Z}}\mathbb{Q}/\mathbb{Z}),Y\right] \\
\nonumber d_r: E_r^{s,t} &\rightarrow E_r^{s+r,t-1},
\end{align}
for the calculation of $\pi_*F\left( H(A\otimes_{\mathbb{Z}}\mathbb{Q}/\mathbb{Z}),Y)\right)$.
\end{theorem}
\begin{proof}
This is a consequence of Theorems \ref{torsion group top thm} and \ref{derived cotorsion top thm}. Spectral sequence \eqref{ss 77801}, in particular, is the spectral sequence of Theorem \ref{torsion group top thm}. The $E_1$-page of \eqref{ss 77801} is as claimed because it is an elementary algebraic exercise to prove that the $p^s$-torsion subgroup of $A\otimes_{\mathbb{Z}}\mathbb{Q}/\mathbb{Z}$ is naturally (in the variables $s$ and $A$) isomorphic to $A/p^sA$. 
\end{proof}

\section{Topological interpretation of $V$ for some examples of generalized homology theories}
\label{Topological interpretation of V...}

The main topological case of interest for the functor $V: \gr\Comod(E_*E) \rightarrow \gr\Mod(E_*)$ in this paper is the case where $E = H\mathbb{F}_p$. Nevertheless it is natural to wonder how the ideas play out for other familiar and often-used generalized homology theories $E_*$. In this section we explain the topological meaning of $V$ when $E$ is either rational homology or the Adams summand of complex $K$-theory localized at an odd prime.

This section does not use the covariant embedding of a comodule category into a module category, or our Theorem \ref{ext iso 1} about when that embedding preserves $\Ext$-groups. Consequently there is no need to use the cohomological grading convention from \cref{conventions}. Instead, the homological grading convention is clearer and more convenient for the ideas in this section, so we use the homological convention throughout.

\begin{recollection}\label{bousfield recollection}
This is a reasonable time to recall the basics of Bousfield localization. A spectrum $X$ is said to be {\em $E$-acyclic} if $E_*X = 0$. We say that $Y$ is {\em $E$-local} if there are no non-nulhomotopic maps from an $E$-acyclic spectrum to $Y$. Finally, a map of spectra $X\rightarrow Y$ is said to be an {\em $E$-local equivalence} if the induced map $E_*X\rightarrow E_*Y$ is an isomorphism.

It is a nontrivial theorem of Bousfield \cite{MR551009} that every spectrum $Y$ admits a map, which is an $E$-local equivalence, to an $E$-local spectrum $L_EY$, unique up to weak equivalence. The spectrum $L_EY$ is called the {\em Bousfield localization of $Y$ at $E$.} A nice exposition of these ideas can be found in \cite[section 1]{MR737778}.

The $E$-nilpotent completion $\hat{Y}_E$ of $Y$ was defined in \cref{Topological interpretation of V}.
There is a canonical (up to homotopy) map $L_EY\rightarrow \hat{Y}_E$.
Bousfield's analysis of convergence conditions in the $E$-Adams spectral sequence in \cite[section 6]{MR551009} is partly motivated by the interest in determining conditions under which the canonical map is a weak equivalence, so that the $E$-Adams spectral sequence calculates homotopy classes of maps into the localization $L_EY$ (which is an object characterized by meaningful homotopical properties with nontrivial, useful consequences) rather than just homotopy classes of maps into the completion $\hat{Y}_E$ (which is not nearly as natural-seeming and is not characterized by the same kinds of useful properties). 

Bousfield proves that, when one finds a horizontal vanishing line of the same height on the same page in the $E$-Adams spectral sequence converging to $[\Sigma^*S^0,\hat{Y}_E]$ for every finite spectrum $Y$, then the map $L_EY \rightarrow\hat{Y}_E$ is a weak equivalence for all $Y$. The strong/complete convergence condition for the $E$-Adams spectral sequence, from Proposition \ref{adams E2 description}, is also satisfied under those conditions.
\end{recollection}

\subsection{$V$ for rational homology}
Let $E = H\mathbb{Q}$, so that $E_*$ is rational homology $H_*(-;\mathbb{Q})$. This is a trivial case: $E_*E = \mathbb{Q} = E_*$, so the associated Hopf algebroid is $(\mathbb{Q},\mathbb{Q})$, and the graded $E_*E$-comodules are simply graded $\mathbb{Q}$-vector spaces. 

In this setting, the extended comodule functor $\gr\Mod(E_*)\rightarrow \gr\Comod(E_*E)$ is an equivalence of categories, and hence its adjoint $V$ is as well. The category $\gr\Comod(E_*E)$ is consequently semisimple, so one has a vanishing line at the $(s=0)$-line already at the $E_2$-page of the $E$-Adams spectral sequence, so the vanishing line condition from Recollection \ref{bousfield recollection} is satisfied, and consequently $L_{H\mathbb{Q}}Y\simeq \hat{Y}_{H\mathbb{Q}}$ for all $Y$. Semisimplicity also implies that the higher derived functors of $V$ vanish, and the spectral sequence of Proposition \ref{adams E2 description} reduces to the isomorphism
\begin{align*}
 H_s(Y;\mathbb{Q}) &\cong [\Sigma^{s} H\mathbb{Q},L_{H\mathbb{Q}}Y],
\end{align*}
which is also easily proven by other means. 

\subsection{$V$ for the Adams summand of complex $K$-theory at an odd prime}
\label{v for e1}

Here is a more interesting example. Let $p$ be an {\em odd} prime, and let $E$ be the $p$-local periodic complex $K$-theory spectrum $KU_{(p)}$, or alternatively, its Adams summand $E(1)$. The spectrum $KU_{(p)}$ splits as the wedge sum 
\begin{align*}
 KU_{(p)} &\simeq E(1) \vee \Sigma^2 E(1) \vee \Sigma^4E(1) \vee \dots \vee \Sigma^{2p-4}E(1) ,\end{align*}
and for many purposes $KU_{(p)}$ and $E(1)$ are interchangeable, e.g. $L_{KU_{(p)}}$ is naturally equivalent to $L_{E(1)}$.
The ring of homotopy groups $\pi_*(KU_{(p)})$ is, by the famous calculation of Bott, isomorphic to $\mathbb{Z}_{(p)}[u^{\pm 1}]$, with the Bott element $u$ in degree $-2$. Meanwhile, $\pi_*(E(1))$ is isomorphic to $\mathbb{Z}_{(p)}[v_1^{\pm 1}]$, with $v_1$ in degree $2(p-1)$. The element $v_1$ can be chosen to be $u^{1-p}$.

Bousfield gave a thorough algebraic analysis of the category of graded $E(1)_*E(1)$-comodules in the paper \cite{MR0796907}. This analysis requires recalling a bit about Adams operations on localized complex $K$-theory. We recall those ideas now. Among the homotopy classes of maps of spectra $KU_{(p)} \rightarrow KU_{(p)}$ (i.e., the stable operations in $KU_{(p)}$-theory), there are the {\em Adams operations}, one such operation $\Psi^q$ for each element $q\in \mathbb{Z}_{(p)}$ coprime to $p$. Under composition, these operations comprise a copy of the unit group $\mathbb{Z}_{(p)}^{\times}$ acting on $KU_{(p)}$-theory by stable operations. The Adams operations $\Psi^q$ with $q\equiv 1$ modulo $p$ are well-defined on $E(1)$. 

Fix an integer $q$ which generates the $p$-torsion subgroup of $\mathbb{Z}/p^2\mathbb{Z}$.
Then, given a space or spectrum $X$, we may record the data of the $E(1)$-homology $E(1)_*X$ in the following way:
\begin{itemize}
\item We have a $\mathbb{Z}_{(p)}$-module $E(1)_0(X)$,
\item we have a $\mathbb{Z}_{(p)}$-module $E(1)_1(X)$,
\item and so on, until we reach the $\mathbb{Z}_{(p)}$-module $E(1)_{2p-3}(X)$. (We need not proceed to $E(1)_{2p-2}(X)$, since the $2(p-1)$-periodicity of $E(1)$-homology implies that $E(1)_{(2p-2) + n}(X)$ agrees with $E(1)_{n}(X)$ as $\mathbb{Z}_{(p)}$-modules. This isomorphism twists the action of the stable operations in $E(1)$-homology in a predictable way; see \cite{MR0139178} or \cite{MR293617}, or \cref{tate rel 1}, below.)
\item On each of the $\mathbb{Z}_{(p)}$-modules $E(1)_n(X)$, with $n\in \{ 0,1,\dots, 2p-3\}$, we record the action of the Adams operation 
$\Psi^q$ on $E(1)_n(X)$.
\end{itemize}
For topological reasons, the following conditions are automatically satisfied: 
\begin{description}
\item[Continuity] 
For each $p$-power-torsion element $x\in E(1)_n(X)$, there exists some integer $j\geq 0$ such that $(\Psi^q)^{p^j}.x = x$. (In a nutshell, this condition gets satisfied because $\hat{\mathbb{Z}}_p^{\times}$ acts {\em continuously} on {\em $p$-complete} $K$-theory.)
\item[Rational diagonalizability] The $\mathbb{Q}$-vector space
$E(1)_n(X)\otimes_{\mathbb{Z}}\mathbb{Q}$ 
admits a decomposition into a direct sum $\bigoplus_{j\in\mathbb{Z}} W_{j(p-1)}$ of $\mathbb{Q}$-vector spaces, such that 
$\Psi^q(X)$ acts on $W_{j(p-1)}$ by the formula 
\begin{align} \label{tate twist formula}\Psi^q.v &= q^{j(p-1)}\cdot v.\end{align} (This condition gets satisfied because all the higher stable homotopy groups of spheres are torsion, i.e., rationally every spectrum splits as a wedge of suspensions of rationalized spheres, and the Adams operations on $E(1)$-homology of the $j(p-1)$-sphere act by the ``Tate twist'' formula \eqref{tate twist formula}.)
\end{description}
Bousfield's category $\mathcal{B}(p)_*^q$ is defined so that it is simply the receptacle for data of the above kind. That is, 
\begin{definition}\label{def of bousfields B cat} $\mathcal{B}(p)_*^q$ is the category of sequences $(\dots ,M_{-1},M_0,M_1,M_2,\dots)$
of 
$\mathbb{Z}_{(p)}$-modules, each equipped with a $\mathbb{Z}_{(p)}$-linear operator $\Psi^q$, such that
\begin{itemize}
\item for each $n$, 
$M_n$ 
satisfies the above 
``Continuity'' and ``Rational diagonalizability'' conditions, and
\item for each $n$, $M_n = M_{n+2(p-1)}$ as $\mathbb{Z}_{(p)}$-modules, with the action of $\Psi^q$ satisfying the equation
\begin{align}
\label{tate rel 1} \Psi^q.x \mbox{\ calculated\ in\ } M_{n+2(p-1)}\ \ 
  &=\ \  q^{p-1}\cdot (\Psi^q.x) \mbox{\ calculated\ in\ } M_n.
\end{align}
\end{itemize}

Bousfield's category $\mathcal{B}(p)^q$ is an ``ungraded'' analogue of $\mathcal{B}(p)_*^q$: rather than a sequence of $\mathbb{Z}_{(p)}$-modules $\dots, M_{-1},M_0,M_1, \dots$, each equipped with a linear operator, an object of $\mathcal{B}(p)^q$ is a single $\mathbb{Z}_{(p)}$-module equipped with a $\mathbb{Z}_{(p)}$-linear operator $\Psi^q$ satisfying the ``Continuity'' and ``Rational diagonalizability'' conditions.
\end{definition}
\begin{remark}
The relation \eqref{tate rel 1} can be rephrased by saying that $M_{n + 2(p-1)}$ is the Tate twist $M_n(p-1)$, in the sense of Definition \ref{def of tate twist}.

The definition of $\mathcal{B}(p)^q_*$ has some redundancy in it. All the information in the sequence $\dots , M_{-1},M_0,M_1,M_2,\dots$ can be recovered from just the ordered $2(p-1)$-tuple $M_0, M_1, \dots, M_{2p-3}$, or for that matter, the ordered $2(p-1)$-tuple $M_{i_0}, M_{i_1}, \dots ,M_{i_{2p-3}}$ for any sequence of integers $i_0, \dots ,i_{2p-3}$ representing each of the residue classes modulo $2p-2$. There is wisdom in retaining the redundant information in  $\mathcal{B}(p)_*^q$, however: it allows for the straightforward formula 
\begin{align}
\label{adams e2 description}
 E_2^{s,t}\cong \Ext^s_{\mathcal{B}(p)_*^q}(E(1)_*X,E(1)_*Y)^t
  &\cong \coprod_{u\in\mathbb{Z}}\Ext^s_{\mathcal{B}(p)^q}(E(1)_{t+u}X,E(1)_uY) 
\end{align}
relating $\Ext^*_{\mathcal{B}(p)_*^q}$ and $\Ext^*_{\mathcal{B}(p)^q}$ and the $E(1)$-Adams $E_2$-page, avoiding an irritating re-indexing which is necessary in order to give such a formula using an alternative ``irredundant'' definition of $\mathcal{B}(p)^q_*$ in terms of ordered $2(p-1)$-tuples.
\end{remark}

Bousfield \cite{MR0796907} proves that the categories $\mathcal{B}(p)^q$ and $\mathcal{B}(p)_*^q$ are abelian, and furthermore that the latter is equivalent to the category of graded $E(1)_*E(1)$-comodules. This matters because the input for the $E(1)$-Adams spectral sequence is $\Ext$ in the category of graded $E(1)_*E(1)$-comodules. 
%
As an application of these ideas, Bousfield constructs \cite[section 7.6]{MR0796907} a spectral sequence for calculating $\Ext$-groups in the category of graded $\mathcal{B}(p)^q$-modules:
\begin{theorem} {\bf (Bousfield.)}\label{bousfield ss 1}
Given any objects $L,M$ of $\mathcal{B}(p)^q$, there exists a spectral sequence strongly converging to $\Ext_{\mathcal{B}(p)^q}^{s+t}(L,M)$, with $E_1$-page
\begin{equation}\label{ss picture 1}\begin{tikzpicture}[xscale=3.8,yscale=0.7]
\draw[->,color=red] (0.3,0) -- ($(0,0)!0.7!(1,0)$) node{};
\draw[->,color=red] (1.3,0) -- ($(1,0)!0.45!(2,0)$) node{};
\draw[->,color=red] (0.3,1) -- ($(0,1)!0.7!(1,1)$) node{};
\draw[->,color=blue] (0.3,.85) -- ($(0,1)!0.72!(2,0)$) node{};
\draw (-0.35,1.7) -- (-0.35,-0.35) -- (2.6,-0.35);
\draw (-0.5,0) node{$s=0$};
\draw (-0.5,1) node{$s=1$};
\draw (0,-0.6) node{$t=0$};
\draw (1,-0.6) node{$t=1$};
\draw (2,-0.6) node{$t=2$};
\draw (0,0) node{$\hom_{\mathbb{Z}_{(p)}}(L,M)$};
\draw (1,0) node{$\hom_{\mathbb{Z}_{(p)}}(L,M)$};
\draw (2,0) node{$\hom_{\mathcal{B}(p)^q}(L\otimes_{\mathbb{Z}}\mathbb{Q},M\otimes_{\mathbb{Z}}\mathbb{Q})$};
\draw (0,1) node{$\Ext^1_{\mathbb{Z}_{(p)}}(L,M)$};
\draw (1,1) node{$\Ext^1_{\mathbb{Z}_{(p)}}(L,M)$};
\end{tikzpicture} \end{equation}
with $d_1$-differentials pictured in red, and $d_2$-differentials pictured in blue. 
The spectral sequence is zero in all bidegrees not pictured, i.e., it is concentrated in only five bidegrees.

The $d_1$-differential from bidegree $(s=0,t=0)$ to bidegree $(s=0,t=1)$ is given by sending a function $f$ to the commutator $f\circ\Psi^q - \Psi^q\circ f$. The $d_1$-differential from bidegree $(s=0,t=1)$ to bidegree $(s=0,t=2)$ is given by sending $f: L \rightarrow M$ to the sum of the homogeneous components of $f\otimes_{\mathbb{Z}}\mathbb{Q}$ with respect to the eigenspace decomposition of the action of $\Psi^q$.
\end{theorem}

We will now apply these ideas of Bousfield to get some understanding of the topological behavior of $V$ in the case $E = E(1)$. To avoid any misunderstanding, we reiterate the assumption that has been in force throughout this subsection: $p$ is an {\em odd} prime.
\begin{definition}\label{def of tate twist}
Suppose we are given an integer $q$ not divisible by $p$, and a $\mathbb{Z}_{(p)}$-module $M$ with a $\mathbb{Z}_{(p)}$-linear operator $\Psi^q: M \rightarrow M$. Let $n$ be an integer.
\begin{itemize}
\item We write $(\mathbb{Q}\otimes_{\mathbb{Z}}M)^{\Psi^q = q^{n}}$ for the eigenspace of the $\mathbb{Q}$-vector space $\mathbb{Q}\otimes_{\mathbb{Z}}M$ on which $\Psi^q$ acts as multiplication by the integer $q^n$.
\item We define the {\em $n$th Tate twist of $M$}, written $M(n)$, as the same underlying $\mathbb{Z}_{(p)}$-module $M$, but equipped with the $\mathbb{Z}_{(p)}$-linear operator given by the formula
\begin{align*}
 M&\rightarrow M \\
 m&\mapsto q^n\cdot (\Psi^q.m).
\end{align*}
\end{itemize}
\end{definition}

\begin{definition}
Given an object $M$ of $\mathcal{B}(p)^q$, by the {\em Bousfield complex of $M$} we mean the cochain complex
\begin{equation}\label{bousfield cplx 1} 0 \rightarrow
 \prod_{n\in\mathbb{Z}}M
 \stackrel{d^1}{\longrightarrow}
\prod_{n\in\mathbb{Z}}M
 \stackrel{d^2}{\longrightarrow}
\prod_{n\in\mathbb{Z}} \left( \mathbb{Q}\otimes_{\mathbb{Z}} M\right)^{\Psi^q = q^{n(p-1)}}
 \rightarrow 0,
\end{equation}
where $d^1$ is given 
as follows. We can write any element of $\prod_{n\in\mathbb{Z}}M$
as a vector $\vec{v}$, with $\vec{v}_n$ denoting the component in the $n$th factor in the product. Then $d_1(\vec{v})$ is the vector in 
$\prod_{n\in\mathbb{Z}}M$ 
whose $n$th component $d_1(\vec{v})_n$ is $q^{n(p-1)}\cdot \vec{v}_n - \Psi^q.\vec{v}_n$. The differential $d^2$ is given by letting $d^2(\vec{w})$ be the vector in $\prod_{n\in\mathbb{Z}} \left( \mathbb{Q}\otimes_{\mathbb{Z}} M\right)^{\Psi^q = q^{n(p-1)}}$ whose $n$th component $d^2(\vec{w})_n\in \left( \mathbb{Q}\otimes_{\mathbb{Z}} M\right)^{\Psi^q = q^{n(p-1)}}$ is simply the projection of 
$\vec{w}_n \in M$ to the eigenspace $\left(\mathbb{Q}\otimes_{\mathbb{Z}} M\right)^{\Psi^q = q^{n(p-1)}}$ of $\mathbb{Q}\otimes_{\mathbb{Z}}M(n(p-1))$.
The Bousfield complex is indexed so that 
$\prod_{n\in\mathbb{Z}} \left( \mathbb{Q}\otimes_{\mathbb{Z}} M\right)^{\Psi^q = q^{n(p-1)}}$ is the group of $2$-cochains, and the other two nonzero terms in \eqref{bousfield cplx 1} are the $0$-cochains and $1$-cochains.
\end{definition}
Here is an alternative description of the differentials $d^1$ and $d^2$ in the Bousfield complex. We may think of $M$ as an extension 
\[ 0 \rightarrow \tors(M) \rightarrow M \rightarrow \tf(M) \rightarrow 0,\]
where $\tors(M)$ is the torsion subgroup of $M$, and $\tf(M)$ is then the largest torsion-free quotient of $M$. Recall from Definition \ref{def of bousfields B cat} that $\Psi^q$ is rationally diagonalizable with certain constrained eigenvalues, so we may view $\tf(M)$ as a $\mathbb{Z}_{(p)}$-linear lattice in a diagonal representation $\prod_{n\in\mathbb{Z}} \mathbb{Q}_{triv}(n(p-1))^{m_n}$ for some sequence of cardinal numbers $\dots ,m_{-1},m_0,m_1,m_2,\dots$. Given an element $\vec{w}\in \prod_{n\in \mathbb{Z}} M$, write $\vec{w}_{n,j}$ for the projection of $\vec{w}_n\in M$ to the summand $\mathbb{Q}_{triv}(j(p-1))^{m_j}$ of $\mathbb{Q}_{triv}\otimes_{\mathbb{Z}}M$. Then $d^2(\vec{w})$ is simply the vector with $n$th component $d^2(\vec{w})_n = \vec{w}_{n,n}$. We can get a similarly nice formula for $d^1$ if we neglect the torsion in $M$ so that we may think of the elements of the domain of $d^1$ again as vectors in a lattice in a diagonal representation: we then have $d^1(\vec{v})_{n,j} = \left(q^{n(p-1)} - q^{j(p-1)}\right)\cdot \vec{v}_{n,j}$.

\begin{prop}\label{bousfield cplx prop}
Let $Y$ be any spectrum. 
Then the $s$th cohomology group of the Bousfield complex of $E(1)_*(\Sigma^{s}Y)$ is isomorphic to $R^sV(E(1)_*\Sigma^{s}Y)$, where 
\begin{align*} V: \gr\Comod(E(1)_*E(1)) &\rightarrow \gr\Mod(E(1)_*)\end{align*} is the right adjoint of the extended comodule functor. Both are isomorphic to the $s$th row in the $E_2$-page of the $E(1)$-Adams spectral sequence converging to $[\Sigma^*E(1),L_{E(1)}Y]$. 
\end{prop}
\begin{proof}
Adams--Clarke \cite{MR454977} and Adams--Harris--Switzer \cite{MR293617} have proven that $E(1)_*E(1)$ is a free $\mathbb{Z}_{(p)}$-module. Consequently, in the case $L = E(1)$, Bousfield's spectral sequence \eqref{ss picture 1} collapses on to the $(s=0)$-row. The connection to $V$ is a simple application of Proposition \ref{I prop 1}. The identification $L_{E(1)}Y\simeq \hat{Y}_{E(1)}$ is classical, an early application of Bousfield's vanishing line conditions from Recollection \ref{bousfield recollection}. 

All that remains is to explain why the $(s=0)$-row in spectral sequence \eqref{ss picture 1} is isomorphic to the Bousfield complex of $E(1)_*Y$. We have the isomorphisms
\begin{align*} 
 \hom_{\gr\Mod(\mathbb{Z}_{(p)})}(E(1)_*E(1),M) &\cong
  \prod_{n\in\mathbb{Z}}M\mbox{\ \ and } \\
 \hom_{\mathcal{B}(p)^q_*}(E(1)_*E(1)\otimes_{\mathbb{Z}}\mathbb{Q},E(1)_*Y\otimes_{\mathbb{Z}}\mathbb{Q}) &\cong 
  \prod_{n\in\mathbb{Z}} \left( \mathbb{Q}\otimes_{\mathbb{Z}} M\right)^{\Psi^q = q^{n(p-1)}}
\end{align*}
because $E(1)_*E(1)$ splits, as an object of $\mathcal{B}(p)^q_*$, as the direct sum \[ \coprod_{n\in\mathbb{Z}} E(1)_*(n(p-1))\] of Tate twists of $E(1)_*$; this is a consequence of the Brown--Gitler--like splitting of the spectrum $E(1)\wedge E(1)$ as the wedge sum $\coprod_{n\in\mathbb{Z}} E(1)\wedge S^{2n(p-1)}$, proven in \cite[section 23]{MR634210}. 

Hence the terms in the $E_1$-page of Bousfield's spectral sequence are precisely the terms in what we have called the Bousfield complex. Showing that the $d_1$-differential in the spectral sequence agrees with the differential in the Bousfield complex is a matter of unwinding, in our special situation, Bousfield's description of the 
$d_1$-differential, recounted in Theorem \ref{bousfield ss 1}. 

For degree reasons, the spectral sequence collapses at $E_2$ with no nontrivial extension problems, so the cohomology of the Bousfield complex recovers the $E_2$-page of the $E(1)$-Adams spectral sequence, as claimed.
\end{proof}
Recall from Definition \ref{def of relative derived cotorsion} that a graded comodule $M$ is said to be {\em relative derived cotorsion} 
if the right derived functor $R^n_{rel}V(M)$ 
vanishes for all $n$. 
We need not worry about the distinction between relative and absolute right-derived functors here, since Proposition \ref{cotorsion and relative cotorsion} implies that $R^*_{rel}V(M)$ vanishes if and only if $R^*V(M)$ vanishes.
\begin{theorem}\label{torsion ht 1 cor}
Let $Y$ be a spectrum satisfying at least one of the following conditions:
\begin{itemize}
\item $E(1)_*Y$ vanishes in all even degrees.
\item $E(1)_*Y$ vanishes in all odd degrees.
\item $E(1)_*Y$ is a torsion group.
\item $\Ext^2_{E(1)_*E(1)}(E(1)_*Y, E(1)_*Y)$ is trivial.
\item $Y$ is a ``generalized $E(1)_*$-Moore spectrum'' in the sense of Bousfield \cite{MR0796907}.
\end{itemize}
Then $R^sV(E(1)_*Y)$ vanishes for all $s>2$, and there exists a two-step filtration of the graded abelian group $\left[\Sigma^* E(1),L_{E(1)}Y\right]$ whose filtration quotients are $V(E(1)_*Y)$ and $R^1V(E(1)_*\Sigma Y)$ and $R^2V(E(1)_*\Sigma^2 Y)$. 
In particular, $E(1)_*Y$ is relative derived cotorsion if and only if the mapping spectrum $F(E(1),L_{E(1)}Y)$ is contractible.
\end{theorem}
\begin{proof}
For degree reasons, in the $E(1)$-Adams spectral sequence at an odd prime converging to $\left[\Sigma^*X,\Sigma^*\hat{Y}_{E(1)}\right]$, there is no room for nonzero differentials longer than the $d_2$-differential. Bousfield gives a remarkable purely algebraic formula for the $d_2$-differential \cite[Proposition 8.10]{MR0796907}, as follows:
\begin{align}\label{bousfield formula} d_2(f) &= k_Yf \pm fk_{X},\end{align}
where $k_X,k_Y$ are the ``$E(1)$-theoretic $k$-invariants'' of $X$ and $Y$, respectively. The $E(1)$-theoretic $k$-invariant of a spectrum $Y$ is a particular element 
\[ k_Y\in \Ext^2_{\mathcal{B}(p)^r_*}(E(1)_*Y,E(1)_*Y),\] which is zero if $E(1)_*Y$ vanishes in all even degrees or vanishes in all odd degrees. We avoid defining or giving a full treatment to the $E(1)$-theoretic $k$-invariant here, since its only properties that we shall use are the properties we have just described in this paragraph.

The five conditions in the statement of the theorem are each conditions which suffice for the vanishing of $k_Y$. Since $E(1)$ itself is concentrated in even degrees, its own $E(1)$-theoretic $k$-invariant $k_{E(1)}$ is also zero. Hence our assumptions are enough to imply that both $k_{Y}$ and $k_{E(1)}$ vanish, and consequently Bousfield's formula \eqref{bousfield formula} yields the vanishing of the $d_2$-differential in the $E(1)$-Adams spectral sequence converging to $\left[\Sigma^*E(1),\Sigma^*\hat{Y}_{E(1)}\right] \cong \left[\Sigma^*E(1),\Sigma^*L_{E(1)}Y\right]$.

The rest is straightforward from the resulting collapse of the $E(1)$-Adams spectral sequence.
\end{proof}
We regard Corollary \ref{torsion ht 1 cor} as a reasonably satisfying topological interpretation of the right adjoint $V$ of the extended comodule functor $\gr\Mod(E_*)\rightarrow \gr\Comod(E_*E)$, in the case $E = E(1)$: the vanishing of $R^sV(E_*Y)$ for $s=0,1,2$ is equivalent to the contractibility of the mapping spectrum $F(E,L_EY)$, so we may think of $V(E_*Y)$ and $R^1V(E_*Y)$ and $R^2V(E_*Y)$ as measuring the presence of nontrivial maps from $E(1)$ to $L_{E(1)}Y$. 

However, Theorem \ref{torsion ht 1 cor} only furnishes that straightforward relationship when $E(1)_*Y$ satisfies some condition (like the five given conditions in Theorem \ref{torsion ht 1 cor}) which ensures the triviality of the $E(1)$-Adams $d_2$-differential. A nontrivial $E(1)$-Adams $d_2$-differential would complicate any such relationship by allowing $[\Sigma^*E(1),L_{E(1)}Y]$ to be smaller, perhaps much smaller, than $R^*V(E(1)_*Y)$.

We now show that this indeed occurs for some spectra $Y$:
\begin{prop}\label{nonzero diffs}
There exist spectra $Y$ such that $\Ext^2_{\mathcal{B}(p)^r_*}(E(1)_*E(1),E(1)_*Y)$ is nontrivial, and furthermore such that the $d_2$-differential is nonzero in the $E(1)$-Adams spectral sequence converging to $[\Sigma^*E(1), L_{E(1)}Y]$.
\end{prop}
\begin{proof}
Since $\pi_*(E(1)) = E(1)_*(S^0)$ and since we have the unit map $S^0\rightarrow E(1)$ of $E(1)$ as a ring spectrum, by naturality we have the map of $E(1)$-Adams spectral sequences
\[\xymatrix{
 E_2^{s,t} \cong \Ext_{\gr\Comod(E(1)_*E(1))}^s(E(1)_*E(1), E(1)_*Y)^t \ar@{=>}[r]\ar[d] & \left[\Sigma^t E(1), L_{E(1)}Y\right]\ar[d] \\
 {}^IE_2^{s,t} \cong \Ext_{\gr\Comod(E(1)_*E(1))}^s(E(1)_*(S^0), E(1)_*Y)^t \ar@{=>}[r] & \left[\Sigma^t S^0, L_{E(1)}Y\right] .
}
\]
We claim that that this map of spectral sequences is, on the $E_2$-page, surjective on the $(s=0)$-line. 
By Proposition \ref{bousfield cplx prop}, the $(s=0)$-line $E_2^{0,*}$ in the top spectral sequence is the product, over all integers $n$, of the fixed points of the action of $\Psi^q$ on the Tate twist $(E(1)_*Y)(n(p-1))$ of $E(1)_*Y$. Meanwhile, the $(s=0)$-line ${}^IE_2^{0,*}$ in the bottom spectral sequence is the fixed points of the action of $\Psi^q$ on $E(1)_*Y$ itself. The map $E_2^{0,*}\rightarrow {}^IE_2^{0,*}$ is merely the projection to a factor, which is certainly surjective.

It is straightforward to see that, given a map $f$ of spectral sequences and a nonzero differential $d_r(f(x)) = y$ in the codomain spectral sequence, $y$ must lift to an element $\tilde{y}$ in the domain spectral sequence such that $d_r(x) = \tilde{y}$. Hence, if there is a nonzero differential in the $E(1)$-Adams spectral sequence converging to $\pi_*(L_{E(1)}Y)$, then there must be a nonzero differential in the $E(1)$-Adams spectral sequence converging to $[\Sigma^* E(1), L_{E(1)}Y]$. It is well-known that there exist spectra $Y$ admitting nonzero differentials of the former kind, hence they must also admit nonzero differentials of the latter kind.
\end{proof}
Proposition \ref{nonzero diffs} is supposed to be a kind of negative result: it tells us that there is no completely straightforward algebraic relationship between $R^*V(E(1)_*Y)$ and $\left[ \Sigma^*E(1),L_{E(1)}Y\right]$ without making some assumption on $Y$, like the hypothesis that $E(1)_*Y$ vanishes either in all even or in all odd degrees, from Theorem \ref{torsion ht 1 cor}. Presumably this difficulty grows as we consider the Johnson--Wilson theories $E(n)_*$ for increasingly large heights $n$, since even at large primes where the $E(n)$-Adams differentials can admit purely algebraic descriptions as they do at odd primes in the case $n=1$, there is room for increasingly long differentials (up to $d_{n^2+n}$) in the spectral sequence. This makes it all the more surprising that, in the limiting case $n=\infty$ (i.e., the case $E(\infty) = H\mathbb{F}_p$), by Theorem \ref{top thm 1}, one gets a relationship between $R^*V(E(\infty)_*Y)$ and $\left[ \Sigma^*E(\infty),\hat{Y}_{E(\infty)}\right]$ which is {\em simpler} than even the $n=1$ case.

One more remark in the case $n=1$. One can at least get a straightforward algebraic relationship when $Y$ is rationally acyclic, i.e., $H_*(Y;\mathbb{Q})\cong 0$:
\begin{prop}\label{torsion ht 1 cor a}
Let $Y$ be a rationally acyclic spectrum. Then $R^sV(E(1)_*Y)$ vanishes for all $s>1$, and there exists a short exact sequence
\[
0 \rightarrow R^1V(E(1)_*\Sigma Y)
 \rightarrow
\left[\Sigma^* E(1),L_{E(1)}Y\right] \rightarrow  V(E(1)_*Y)  \rightarrow 0.\]
In particular, it is again true that $E(1)_*Y$ is relative derived cotorsion if and only if there are no non-nulhomotopic maps from suspensions of $E(1)$ to the $E(1)$-localization $L_{E(1)}Y$.
\end{prop}
\begin{proof}
For rationally acyclic $Y$, the $2$-cochains are trivial in the Bousfield complex of $Y$. Consequently $R^2V(E(1)_*Y)$ is trivial. The rest is as argued in Theorem \ref{torsion ht 1 cor}.
\end{proof}
Proposition \ref{torsion ht 1 cor a} is curiously similar to Theorem \ref{top thm 1}.

\appendix

\section{Relative homological-algebraic ideas}
This section explains basics of relative homological algebra and its relationship to generalized Adams spectral sequences. The purpose is to provide the necessary background to understand why:
\begin{itemize}
\item in some relevant cases in the literature, a generalized Adams $E_2$-page is described in terms of {\em relative} Ext groups in a category of comodules, 
\item while in other relevant cases in the literature, a generalized Adams $E_2$-page is described in terms of {\em absolute} Ext groups in a category of comodules,
\item but for the purposes of this paper, we may be sloppy about the distinction between the two, since (for slightly nontrivial reasons) the absolute and relative Ext groups agree in the cases of interest for this paper.
\end{itemize}
There is no new mathematics in this section.

\subsection{Review of basics of relative homological algebra}
\label{rel hom alg review}
We recall the ``allowable class'' conventions for relative homological algebra, from \cite[chapter IX]{MR1344215}: an {\em allowable class} in an abelian category $\mathcal{C}$ is simply a class of short exact sequences in $\mathcal{C}$ which is closed under isomorphism and which contains all the short exact sequences whose kernel term or cokernel term is zero. Given an allowable class, a {\em relative epimorphism} is a map which occurs as the epimorphism in one of the short exact sequences in the allowable class. A {\em relative projective} is an object $P$ such that every map from $P$ to the codomain of a relative epimorphism lifts to the domain. Relative monomorphisms and relative injectives are defined dually. A {\em relative injective resolution} of an object $M$ is a resolution of $M$ whose objects are relative injectives and whose differentials each factor as a relative epimorphism followed by a relative monomorphism. The {\em relative right-derived functors} of a functor $F$ are defined as the cohomology of the cochain complex obtained by applying $F$ to a relative injective resolution. Relative $\Ext$-groups are the special case where $F$ is a $\hom$-functor.

When working with categories of comodules over a Hopf algebroid, there are two allowable classes of clear interest:
\begin{itemize}
\item the {\em absolute} allowable class, which consists of {\em all} the short exact sequences in the category of comodules. This corresponds to doing classical, i.e., absolute, non-relative, homological algebra in the category of comodules.
\item The {\em standard} allowable class, defined as follows: given a graded Hopf algebroid $(A,\Gamma)$, the standard allowable class consists of those short exact sequences of graded left $\Gamma$-comodules whose underlying short exact sequences of graded $A$-modules are split. The cohomology of the cobar complex of $(A,\Gamma)$ yields the relative $\Ext$-groups with respect to this allowable class.
\end{itemize}

\subsection{Relative Ext groups and the $E$-Adams $E_2$-page}
\label{rel ext and adams E2}
Given a ring spectrum $E$ satisfying reasonable conditions (see Proposition \ref{adams E2 description}), and given spectra $X$ and $Y$, one can construct the $E$-Adams spectral sequence converging to $\left[ \Sigma^* X,\hat{Y}_E\right]$, as in Proposition \ref{adams E2 description}. 

In the most classical and important cases, the spectrum $X$ is the sphere spectrum, so that the spectral sequence converges to $\pi_*(\hat{Y}_E)$. The $E_2$-page of this spectral sequence is well-known to be the cohomology of the cobar complex of the Hopf algebroid $(\pi_*E,E_*E)$ with coefficients in the graded comodule $E_*Y$; see \cite[appendix 1]{MR860042} for the cobar complex, or \cite[chapter 2]{MR860042} for the construction of the spectral sequence. The comodules in the cobar resolution are not generally injective, though: they are only {\em relatively} injective, with respect to the standard allowable class. Hence the Adams $E_2$-page is given by the {\em relative} $\Ext$-groups $\Ext_{\text{rel.\ }\gr\Comod(E_*E)}^*(\pi_*(E),E_*(Y))$.

On the other hand, it is standard (\cite[section III.15]{MR1324104},\cite[chapter 2]{MR860042}) to describe the $E_2$-page of the $E$-Adams spectral sequence---apparently the same spectral sequence---as the {\em absolute} $\Ext$-groups $\Ext_{\gr\Comod(E_*E)}^*(\pi_*(E),E_*(Y))$, not the relative $\Ext$-groups. This is no mistake: the relative $\Ext$-groups and the absolute $\Ext$-groups coincide in this case. More generally, it is an easy exercise to show that there is a natural isomorphism
\begin{align*}
 \Ext_{\text{rel.\ }\gr\Comod(E_*E)}^*(E_*(X),E_*(Y)) &\cong \Ext_{\gr\Comod(E_*E)}^*(E_*(X),E_*(Y))
\end{align*}
as long as the underlying $\pi_*(E)$-module of $E_*(X)$ is projective.
This is explained in \cite[section III.15]{MR1324104}. See the proof of Proposition \ref{cotorsion and relative cotorsion} for additional references.

If $E_*(X)$ is not projective over $E_*E$, one can still construct the spectral sequence and identify the $E_2$-page (under the same hypotheses as described in Proposition \ref{adams E2 description}) as the {\em absolute} $\Ext$-groups $\Ext_{\gr\Comod(E_*E)}^*(E_*(X),E_*(Y))$. This is an old point, but one that seems not as well-known as it ought to be; see \cite[pages 50 and 54]{MR0251716} and \cite{MR686121} for explanation and discussion, and \cite[last sentence of first paragraph in section 8]{MR0796907} for an illuminating (but brief and oblique) observation about this.

In this paper, we only ever consider $E$-Adams spectral sequences in cases where $X$ is either $E$ or the sphere spectrum. Consequently the hypothesis that $E_*E$ is projective over $\pi_*(E)$ is sufficient to ensure that we need not be careful about whether we are using absolute or relative $\Ext$-groups. In the cases studied in this paper, $E$ is either one of the Eilenberg-Mac Lane spectra $H\mathbb{F}_p$ or $H\mathbb{Q}$, or $E$ is the height $1$ Johnson--Wilson theory $E(1)$ (i.e., the Adams summand of $p$-local periodic complex $K$-theory). In those cases, it is indeed true that $E_*E$ is projective over $\pi_*(E)$. This is trivially seen in the Eilenberg-Mac Lane cases. In the case $E = E(1)$, projectivity of $E_*E$ over $\pi_*E$ is a consequence of the theorem \cite{MR454977},\cite{MR293617} that $E(1)_*E(1)$ is a free $\mathbb{Z}_{(p)}$-module as well as the observation that a graded module over $E(1)_*\cong \mathbb{Z}_{(p)}[v_1^{\pm 1}]$ is free over $E(1)_*$ if and only if it is free over $\mathbb{Z}_{(p)}$. This also implies that the relative right-derived functors of $V$ coincide with the absolute right-derived functors of $V$ in all cases studied in this paper (see Propositions \ref{I prop 1} and \ref{cotorsion and relative cotorsion}).

Hence we need not make much fuss over the distinction between relative and absolute $\Ext$-groups in this paper. In principle one might study some ring spectrum $E$ such that $E_*E$ is flat but not projective over $\pi_*(E)$, and then one would need to be careful about the distinction, e.g. $R^*V(M)$ could fail to agree with $R^*_{rel}V(M)$. That does not occur in this paper, though.

\section{Review of the Steenrod algebra}
\label{Review of the Steenrod algebras}

The basic algebraic structure of the Steenrod algebras is very well-known in stable homotopy theory. The book \cite{MR0145525} is a classical and standard reference. In this section we provide a quick review of these facts for pure algebraists who read this paper and who are curious about the primary topological application but who do not already know the Steenrod algebras.

For each prime $p$, the mod $p$ Steenrod algebra is succinctly described as the algebra of stable natural transformations $H^*(-; \mathbb{F}_p)\rightarrow H^*(-;\mathbb{F}_p)$ of the mod $p$ cohomology functor on topological spaces. Some of these operations do not preserve the grading degree, e.g. the Bockstein operator is a stable operation of the form $H^n(-; \mathbb{F}_p)\rightarrow H^{n+1}(-;\mathbb{F}_p)$ for each $n$. The Steenrod algebra is graded so that the stable operations which raise degree by $d$ are homogeneous of degree $d$. Hence, for example, the Bockstein operator is an element of degree $1$ in the mod $p$ Steenrod algebra.

The mod $p$ Steenrod algebra is an infinite-dimensional, connected, co-commutative, noncommutative Hopf $\mathbb{F}_p$-algebra. For the purposes of this paper, it is most convenient to focus on its linear dual, the mod $p$ {\em dual} Steenrod algebra. 
By the well-known 1958 calculation of Milnor \cite{MR0099653}, the $2$-primary dual Steenrod algebra is isomorphic to a polynomial algebra, $\mathbb{F}_2[\xi_1, \xi_2, \xi_3, \dots]$, with each $\xi_n$ in degree $1-2^n$ using our cohomological grading conventions from \cref{conventions}, and with augmentation $\epsilon(\xi_n) = 0$ and coproduct $\Delta(\xi_n) = \sum_{i=0}^n \xi_i^{2^{n-i}}\otimes \xi_{n-i}$, where $\xi_0$ is understood to be $1$. For an odd prime $p$, Milnor calculated that the $p$-primary dual Steenrod algebra is isomorphic to the tensor product of an exterior and a polynomial algebra, 
\begin{equation*}
\mathbb{F}_p[\xi_1, \xi_2, \xi_3, \dots]\otimes_{\mathbb{F}_p}\Lambda_{\mathbb{F}_p}(\tau_0,\tau_1, \tau_2,\dots),
\end{equation*}
with $\xi_n$ in degree $2(1-p^k)$ and with $\tau_n$ in degree $1-2p^k$, with augmentation $\epsilon(\xi_n) = 0$ and $\epsilon(\tau_n) = 0$, and with coproduct
\begin{align*}
 \Delta(\xi_n)  &= \sum_{i=0}^n \xi_i^{p^{n-i}} \otimes \xi_{n-i},\\
 \Delta(\tau_n) &= \tau_n\otimes 1 + \sum_{i=0}^n \xi_i^{p^{n-i}} \otimes \tau_{n-i}.
\end{align*}
For any space or spectrum $X$, its homology $H_*(X;\mathbb{F}_p)$ is a graded comodule over the mod $p$ dual Steenrod algebra $\Gamma$. A large part of the interest in this comodule structure is because, if one can calculate $\Ext$ in the category of graded comodules over the dual Steenrod algebra, then one can calculate the $E_2$-page of the Adams spectral sequence, which is as follows, using the cohomological grading convention from \cref{conventions}:
\begin{align}
\label{adams ss 1} E_2^{s,t} \cong \Ext^s_{\gr\Comod(\Gamma)}(\mathbb{F}_p,H_{-*}(X;\mathbb{F}_p))^{-t} & \Rightarrow  \pi_{t-s}(\hat{X}_{H\mathbb{F}_p}). 
\end{align}
Here $X_{H\mathbb{F}_p}$ is the ``$H\mathbb{F}_p$-nilpotent completion'' of $X$, described in \cref{Topological interpretation of V}. In the case that $X$ is bounded below (e.g. if $X$ is a pointed space, or what comes to the same thing, the suspension spectrum of a pointed space), then by the analysis in \cite{MR551009}, the stable homotopy groups $\pi_*(X_{H\mathbb{F}_p})$ agree with $L_0\Lambda_p\pi_*(X)$, where $\Lambda_p$ is the $p$-adic completion functor on abelian groups, $L_0\Lambda_p$ is its zeroth left derived functor, and $\pi_*(X)$ denotes the stable homotopy groups of $X$. In particular, if each stable homotopy group $\pi_n(X)$ is known to be finitely generated, then \eqref{adams ss 1} converges to the $p$-adic completion of the stable homotopy groups of $X$.

Adams's original treatment of the spectral sequence identified its $E_2$-page without the need for comodules, as $\Ext^s_{\gr\Mod(\Gamma^*)}\left(\Sigma^t H^*(X;\mathbb{F}_p),\mathbb{F}_p\right)$, but (as Adams was careful about) this requires a finiteness hypothesis on $X$, namely that $X$ must admit a stable CW-decomposition with finitely many cells in each dimension. For general spaces and spectra $X$ without the finiteness condition, one must describe the $E_2$-page in terms of {\em comodule} $\Ext$, as we did above in \eqref{adams ss 1}---or, using the ideas in this paper, use Corollary \ref{adams e2 cor} to describe the Adams $E_2$-page in terms of modules without the finite-type hypothesis.

\bibliographystyle{plain}

\def\cprime{$'$} \def\cprime{$'$} \def\cprime{$'$} \def\cprime{$'$}

\end{document}